\documentclass[a4paper, 11pt]{article}

\usepackage[utf8]{inputenc}
\usepackage[T1]{fontenc}
\usepackage{lmodern}
\usepackage{graphicx}
\usepackage{xcolor}
\usepackage[english]{babel}
\usepackage[font=small,labelfont=bf]{caption}
\usepackage{amsmath, amstext, amsopn, amssymb, amsthm}
\usepackage{thmtools}
\usepackage{mathtools}
\usepackage{etoolbox}
\usepackage{stmaryrd}
\usepackage{nameref}
\usepackage{hyperref}
\usepackage{fancyhdr}
\usepackage{setspace}
\usepackage{float}
\usepackage[left=3cm,right=3cm,top=3cm,bottom=3cm]{geometry}
\usepackage{microtype}
\usepackage{footmisc}

\addto\extrasenglish{

}
\hypersetup{
	colorlinks   = true,
	urlcolor     = blue,
	linkcolor    = blue,
	citecolor    = blue
}

\newcommand{\R}{\mathbb{R}}
\newcommand{\RR}{\mathbb{R}}
\newcommand{\N}{\mathbb{N}}
\newcommand{\NN}{\mathbb{N}}

\renewcommand{\c}{c}

\newcommand{\XX}{\mathcal{X}}
\newcommand{\YY}{\mathcal{Y}}

\newcommand{\XC}{\XX}
\newcommand{\YC}{\YY}

\newcommand{\eps}{\varepsilon}
\renewcommand{\epsilon}{\varepsilon}
\newcommand{\sP}{\mathcal{P}}
\newcommand{\PC}{\mathcal{P}}
\newcommand{\sF}{\mathcal{F}}
\newcommand{\sD}{\mathcal{D}}
\newcommand{\sR}{\mathcal{R}}
\newcommand{\sU}{\mathcal{U}}
\newcommand{\sC}{\mathcal{C}}
\newcommand{\sA}{\mathcal{A}}

\newcommand{\bigO}{\mathcal{O}}
\newcommand{\borel}{\mathcal{B}}
\newcommand{\hmu}{\hat{\mu}}
\newcommand{\hnu}{\hat{\nu}}
\newcommand{\tphi}{\tilde{\phi}}
\newcommand{\tpsi}{\tilde{\psi}}
\newcommand{\tmu}{\tilde{\mu}}
\newcommand{\tnu}{\tilde{\nu}}
\newcommand{\trho}{\tilde{\c}}

\newcommand{\blankdot}[1]{\ifblank{#1}{\argdot}{#1}}
\DeclarePairedDelimiter{\ceil}{\lceil}{\rceil}
\DeclarePairedDelimiter{\floor}{\lfloor}{\rfloor}
\DeclarePairedDelimiterX{\abs}[1]{\lvert}{\rvert}{\blankdot{#1}}
\DeclarePairedDelimiterX{\norm}[1]{\lVert}{\rVert}{\blankdot{#1}}
\DeclarePairedDelimiterXPP{\norminf}[1]{}{\lVert}{\rVert}{_\infty}{\blankdot{#1}}
\DeclarePairedDelimiterX{\scalp}[2]{\langle}{\rangle}{{#1},{#2}}
\DeclarePairedDelimiter{\mset}{\lbrace}{\rbrace}

\DeclarePairedDelimiter{\bbrac}{\llbracket}{\rrbracket}

\DeclarePairedDelimiterXPP{\covnum}[3]{N}{(}{)}{}{#1, \, #2, \, #3}
\DeclarePairedDelimiterXPP{\EV}[1]{\mathop{}\!\mathbb{E}}{[}{]}{}{#1}
\DeclarePairedDelimiterXPP{\EVV}[2]{\mathop{}\!\mathbb{E}_{#1}}{[}{]}{}{#2}
\DeclarePairedDelimiterXPP{\Var}[2]{\mathop{}\!\operatorname{Var}_{#1}}{[}{]}{}{#2}

\DeclareMathOperator{\id}{id}
\DeclareMathOperator{\unifdist}{Unif}

\DeclareMathOperator{\diam}{diam}

\DeclareMathOperator{\OT}{T}
\DeclareMathOperator{\KL}{KL}
\DeclareMathOperator{\SG}{SG}
\DeclareMathOperator{\GW}{GW}
\DeclareMathOperator{\supp}{supp}
\newcommand{\Lexp}[1][\eps]{L^{\exp}_{#1}}
\newcommand{\enttrasf}[4]{#1^{(#2, #3 \ifblank{#4}{}{,#4})}}
\newcommand{\entt}[2]{\enttrasf{#1}{\c}{\eps}{#2}}
\newcommand{\entta}[3]{\enttrasf{#1}{\c_{#2}}{\eps}{#3}}
\newcommand{\enttd}[2]{\enttrasf{#1}{\sD}{\eps}{#2}}
\newcommand{\FC}[1]{\sF_{\c, \eps \ifblank{#1}{}{,#1}}}
\newcommand{\FCd}{\sF_{\sD, \eps}}
\newcommand{\EOT}{\OT_{\c,\eps}}
\newcommand{\D}[2]{D_{\c, \eps}^{#1,#2}}
\newcommand{\hOT}{\widehat{\OT}}
\newcommand{\hEOT}{\hOT_{\c,\eps,n}}
\newcommand{\hEGW}{\widehat{\GW}_{\eps, n}}
\newcommand{\sint}[1]{\underline{#1}}
\newcommand{\diffx}[1]{\frac{\partial}{\partial x_{#1}}}
\newcommand{\diffu}[1]{\frac{\partial}{\partial u_{#1}}}
\newcommand{\idn}[1]{\bbrac{#1}}

\makeatletter
\newcommand*{\bigcdot}{}%
\DeclareRobustCommand*{\bigcdot}{%
	\mathbin{\mathpalette\bigcdot@{}}%
}
\newcommand*{\bigcdot@scalefactor}{.5}
\newcommand*{\bigcdot@widthfactor}{1.15}
\newcommand*{\bigcdot@}[2]{%
	\sbox0{$#1\vcenter{}$}%
	\sbox2{$#1\cdot\m@th$}%
	\hbox to \bigcdot@widthfactor\wd2{%
		\hfil
		\raise\ht0\hbox{%
			\scalebox{\bigcdot@scalefactor}{%
				\lower\ht0\hbox{$#1\bullet\m@th$}%
			}%
		}%
		\hfil
	}%
}
\makeatother
\newcommand{\argdot}{\,\bigcdot\,}
\makeatletter
\newcommand*{\transp}{%
	{\mathpalette\@transpose{}}%
}
\newcommand*{\@transpose}[2]{%
	\raisebox{\depth}{$\m@th#1\intercal$}%
}
\makeatother
\newcommand{\restr}[2]{{\left.\kern-\nulldelimiterspace #1 \vphantom{\big|} \right|_{#2}}}
\newcommand{\indic}{\mathop{}\!\mathbb{I}}

\newcommand{\indicfunc}[1]{\indic({#1})}

\newcommand{\given}{\mid}

\newcommand{\defeq}{\coloneqq}

\newcommand{\Prob}{\mathop{}\!\mathbb{P}}
\newcommand{\de}{\mathop{}\!\mathrm{d}}
\newcommand{\De}{\mathop{}\!\mathrm{D}}
\newcommand{\dP}[3][\paren]{\mathop{}\!{#2}#1{\de{#3}}}
\newcommand{\ddP}[4][\paren]{\mathop{}\!{#2}#1{\de{#3},\de{#4}}}

\theoremstyle{plain}
\newtheorem{sectioncount}{xxxxxxx}[section]
\newtheorem{theorem}[sectioncount]{Theorem}
\newtheorem{proposition}[sectioncount]{Proposition}
\newtheorem{lemma}[sectioncount]{Lemma}
\newtheorem{corollary}[sectioncount]{Corollary}
\theoremstyle{definition}
\newtheorem{example}[sectioncount]{Example}

\newtheorem{remark}[sectioncount]{Remark}

\newtheorem{assumptionXXX}{Assumption}
\newenvironment{assumption}[1]
{\renewcommand{\theassumptionXXX}{\textup{\textbf{(#1)}}}\begin{assumptionXXX}}
{\end{assumptionXXX}}
\newcommand{\enumref}[2]{\autoref{#1}.\ref{#2}}

\usepackage{natbib}
\bibpunct{(}{)}{;}{a}{,}{,}

\newcommand{\footremember}[2]{
	\footnote{#2}
	\newcounter{#1}
	\setcounter{#1}{\value{footnote}}
}
\newcommand{\footrecall}[1]{
	\footnotemark[\value{#1}]
}

\begin{document}

\title{Lower Complexity Adaptation \\ for Empirical Entropic Optimal Transport}
\author{Michel Groppe
	\hspace{-0.6em}
	\footremember{ims}{\scriptsize
		Institute for Mathematical
		Stochastics, University of G\"ottingen,
		Goldschmidtstra{\ss}e 7, 37077 G\"ottingen} \\
	\footnotesize{\href{mailto:michel.groppe@uni-goettingen.de}{michel.groppe@uni-goettingen.de}} \\[2ex]
	Shayan Hundrieser
	\hspace{-0.6em}
	\footrecall{ims} \\
	\footnotesize{\href{mailto:s.hundrieser@math.uni-goettingen.de}{s.hundrieser@math.uni-goettingen.de}}
}

\maketitle

\begin{abstract}%
	\noindent Entropic optimal transport (EOT) presents an effective and computationally viable alternative to unregularized optimal transport (OT), offering diverse applications for large-scale data analysis. In this work, we derive novel statistical bounds for empirical plug-in estimators of the EOT cost and show that their statistical performance in the entropy regularization parameter $\eps$ and the sample size $n$ only depends on the simpler of the two probability measures. For instance, under sufficiently smooth costs this yields the parametric rate $n^{-1/2}$ with factor $\eps^{-d/2}$, where $d$ is the minimum dimension of the two population measures. This confirms that empirical EOT also adheres to the \emph{lower complexity adaptation} principle, a hallmark feature only recently identified for unregularized OT. As a consequence of our theory, we show that the empirical entropic Gromov-Wasserstein distance and its unregularized version for measures on Euclidean spaces also obey this principle. Additionally, we comment on computational aspects and complement our findings with Monte Carlo simulations. Our technique employs empirical process theory and relies on a dual formulation of EOT over a single function class. Central to our analysis is the observation that the entropic cost-transformation of a function class does not increase its uniform metric entropy by~much.
\end{abstract}

\noindent \textit{Keywords}: optimal transport, convergence rate, metric entropy, curse of dimensionality
\noindent \textit{MSC 2020 subject classification}: primary  49Q22, 62G20, 62R07; secondary
62F35

\section{Introduction} \label{sec:intro}

The mathematical theory of optimal transport (OT) offers versatile methods to compare complex objects that are modeled as probability distributions. From the problem of optimally moving soil as considered by the French mathematician \citet{Monge1781}, and its use in economics paved by the work of the Soviet mathematician \citet{Kantorovitch1942,Kantorovitch1958} in the 20th century, it has evolved by now to a well-studied area of mathematics \citep{Rachev1998a,Rachev1998b, Villani2003,Villani2009, Santambrogio2015, panaretos2020invitation} and an active field of modern research. Its scope of applications spans across various disciplines such as machine learning \citep{Courty2014,Courty2017,Flamary2016,Arjovsky2017,Gulrajani2017,ho2017multilevel,grave2019unsupervised}, econometrics \citep{Galichon2018, hallin2022center, hallin2022efficient}, computational biology \citep{Evans2012,Schiebinger2019,tameling2021colocalization,Wang2021,Weitkamp2022}, statistics \citep{Munk1998,Barrio1999,sommerfeld2018inference,Panaretos2019,Deb2021,Nies2021,Mordant2022}, and computer vision \citep{solomon2015convolutional, bonneel2023survey}. Nevertheless, despite continuous progress, OT-based methodology is often burdened by computational and statistical limitations, restricting its utility for large-scale data analysis. This gave rise to the consideration of OT surrogates like entropic optimal transport (EOT) \citep{Cuturi2013, Peyre2019, nutz2021introduction} in the hope of more favorable properties.

\subsection{Optimal Transport}
To formalize the OT problem, let $\XX$ and $\YY$ be Polish spaces equipped with Borel-$\sigma$-fields $\borel(\XX)$ and $\borel(\YY)$. Then, the OT cost between two probability measures $\mu \in \sP(\XX)$,~$\nu \in \sP(\YY)$ with respect to (w.r.t.)\ a Borel measurable cost function $\c : \XX\times\YY \to \R$ is defined as
\begin{equation}\label{eq:def_OT}
	\OT_\c(\mu, \nu) \defeq \inf_{\pi \in \Pi(\mu, \nu)} \int_{\XX \times \YY} \c \de{\pi}.
\end{equation}
Herein, $\Pi(\mu, \nu)$ denotes the set of all transport plans between $\mu$ and $\nu$, i.e., every element $\pi \in \Pi(\mu, \nu)$ is a probability measure on $\XX \times \YY$ such that $\pi(A \times \YY) = \mu(A)$ for all $A \in \borel(\XX)$ and $\pi(\XX \times B) = \nu(B)$ for all $B \in \borel(\YY)$.
In case of a common Polish metric space $\XX=\YY$ and costs chosen as the power of the metric, the OT cost quantifies the dissimilarity of $\mu$ and $\nu$ in a manner that is consistent with the ground space geometry \citep[Chapter~6]{Villani2009}.

For most applied purposes the measures $\mu$ and $\nu$ are not available and instead one has only access to independent and identically distributed (i.i.d.)\ random variables $X_1, \ldots, X_n \sim \mu$ and independent thereof $Y_1, \ldots Y_n \sim \nu$. This gives rise to empirical measures
\begin{equation*}
\hat{\mu}_n \defeq \frac{1}{n} \sum_{i=1}^n \delta_{X_i}\,, \qquad \hat{\nu}_n \defeq \frac{1}{n} \sum_{i=1}^n \delta_{Y_i}\,,
\end{equation*}
which serve to define empirical plug-in estimators for the OT cost,
\begin{equation}\label{eq:Emp_OT}
	\hOT_{\c, n} \in \mset{\OT_\c(\mu, \hnu_n), \OT_\c(\hmu_n, \nu), \OT_\c(\hmu_n, \hnu_n)}.
\end{equation}

The statistical performance of such estimators has been investigated extensively for various settings of which we only provide a selective overview, for a detailed recent account we refer to \citet[Section 2]{Staudt2023Convergence}. First contributions were made for metric costs with identical measures \citep{dudley1969speed, boissard2014mean, fournier2015rate,weed2019sharp}, but recently the analysis was refined to more general costs with possibly different measures \citep{Chizat2020, niles2022estimation, Hundrieser2022, manole2021sharp, Staudt2023Convergence}. A~central quantity of interest in all these works is the mean absolute deviation of the empirical estimator to the true OT cost,
\begin{equation*}
	\EV{\abs{ \hOT_{\c,n} - \OT_\c(\mu, \nu) }}\,,
\end{equation*}
whose convergence behavior intricately depends on the regularity of the cost function, the intrinsic dimension of the measures $\mu$ and $\nu$ as well as their concentration. Elaborating on this, let us exemplarily consider Euclidean ground spaces $\XX = \YY = [0, 1]^d$ with $d \neq 4$ and squared Euclidean costs $\c(x,y) = \norm{x-y}_2^2$ for which it has been shown that \citep{Chizat2020,Hundrieser2022,manole2021sharp}\footnote{In case of $d = 4$ the upper bound in \eqref{eq:rate_OT} admits an additional logarithmic term while the lower bound does not. The sharp logarithmic order remains open.}
\begin{equation}\label{eq:rate_OT}
	\sup_{\mu, \nu \in \sP([0,1]^d)} \EV{\abs{ \hOT_{\c,n} - \OT_\c(\mu, \nu) }} \asymp n^{-2/(d \vee 4)}.
\end{equation}
\citet{manole2021sharp} in fact also show that the empirical plug-in estimator is minimax rate optimal (up to logarithmic terms) among all measurable functions $\widehat{\OT}$ of $X_1, \ldots, X_n$ and $Y_1, \ldots, Y_n$. In particular, we observe that the statistical performance of the empirical plug-in estimator deteriorates exponentially with higher dimension $d$ and that estimation of OT costs is affected by the curse of dimensionality.

Hence, the only way to hope for faster convergence rates is by imposing additional structural assumptions on the underlying population measures. For instance, if both $\mu$ and $\nu$ are supported on $[0,1]^s \times \mset{0}^{d-s} \subseteq [0, 1]^d$ for $s \neq 4$ the ground space is effectively only $s$-dimensional and the convergence rate improves to the better rate $n^{-2/s}$. A more surprising result by \citet{Hundrieser2022} is that only one of the probability measures needs to be concentrated on a low-dimensional domain. If, say, $\mu$ is concentrated on an $s$-dimensional sub-manifold and $\nu \in \sP([0,1]^d)$ is arbitrary, then the empirical OT cost estimator converges with rate $n^{-2/(s \lor 4)}$ and thus adapts to the lower complexity of~$\mu$. This phenomenon is termed \emph{lower complexity adaptation} (LCA) principle. Nonetheless, note that the LCA principle only yields close to parametric rates in $n$ if one of the measures has sufficiently low intrinsic dimension. Hence, unless $s < 4$, slower than parametric convergence rates for the estimation of OT costs can still occur.

In addition, OT is generally plagued by a high computational complexity \citep{Peyre2019}. For instance, the Auction algorithm \citep{Bertsekas1981,Bertsekas1989} or Orlin's algorithm \citep{Orlin1988} have a worst-case computational complexity of $\bigO(n^3)$ (up to polylogarithmic terms) where $n$ is the number of support points of the input measures. This effectively delimits the use of OT based methodology to measures with $n\sim 10^4$ support points.

\subsection{Entropic Optimal Transport}

To deal with the high computational complexity,  \citet{Cuturi2013} proposed to regularize the objective \eqref{eq:def_OT} by adding an entropic penalization term. For probability measures $\mu \in \sP(\XX)$ and $\nu \in \sP(\YY)$ the EOT cost w.r.t.\ to the cost function $\c : \XX \times \YY \to \R$ is defined as
\begin{equation}\label{eq:def_EOT}
	\EOT(\mu, \nu) \defeq \inf_{\pi \in \Pi(\mu, \nu)} \int_{\XX \times \YY} \c \de{\pi} + \eps \KL(\pi \given \mu \otimes \nu) \,,
\end{equation}
where $\eps > 0$ is the entropic regularization parameter and $\KL(\pi \given \mu \otimes \nu)$ denotes the Kullback-Leibler divergence of $\pi$ relative the independent coupling of $\mu$ and $\nu$, defined by $\int_{\XX \times \YY} \log ( \frac{\de{\pi}}{\de{[\mu \otimes \nu]}}) \de{\pi}$ if $\pi\ll \mu\otimes \nu$ and $+\infty$ else. This regularization allows the use of an efficient and simple computational scheme called the Sinkhorn algorithm \citep{Cuturi2013,Peyre2019,Schmitzer2019}. Given a fixed precision, the Sinkhorn algorithm has computational complexity of order $\bigO(n^2)$ (up to polylogarithmic terms) \citep{Altschuler2017, dvurechensky2018computational,luo2023improved} to approximate the EOT cost and is thus one order of magnitude faster than OT solvers. Owing to the computational appeal, EOT has therefore been employed to approximate quantities from unregularized OT by reducing the regularization parameter $\epsilon \searrow 0$, which proves to be consistent under minimal regularity on the cost function \citep{pooladian2021entropic,altschuler2022asymptotics,bernton2022entropic,delalande2022nearly,nutz2022entropic,pal2019difference}.

The EOT cost also admits a more desirable sample complexity compared to that of unregularized OT \citep{Genevay2019, Mena2019, Chizat2020, Rigollet2022, bayraktar2022stability}. Following an empirical plug-in approach as in~\eqref{eq:Emp_OT},
\begin{equation}\label{eq:eot_emp_est}
	\hEOT \in \{\EOT(\mu, \hnu_n), \EOT(\hmu_n, \nu), \EOT(\hmu_n, \hnu_n)\}\,,
\end{equation}
it was first shown by \citet{Genevay2019} under smooth costs and compactly supported $\mu, \nu \in \sP(\R^d)$ that for a fixed $\epsilon>0$ the estimator $\hEOT$ based on i.i.d. observations $X_1, \ldots, X_n \sim \mu$ and $Y_1, \ldots, Y_n \sim \nu$ converges to the population quantity $\EOT(\mu, \nu)$ at the parametric rate $n^{-1/2}$.
\begin{subequations}\label{eq:eot_bounds}
In the setting where $\c$ is selected as the squared Euclidean norm their results imply that
\begin{equation}\label{eq:EOT_Genevay}
	\EV{\abs{\hEOT - \EOT(\mu, \nu)}} \lesssim (1 + \eps^{-\floor{d/2}}) \exp(D^2/\eps) n^{-1/2}\,,
\end{equation}
where $D$ is the diameter of the bounded subset $\mu$ and $\nu$ are supported on, and the implicit constant only depends on $d$ and $D$. Tailored to the squared Euclidean norm, \citet{Mena2019} and \citet{Chizat2020} improved upon these results by reducing the exponential dependency in $\epsilon^{-1}$ to a polynomial one. Notably, the former also provide statistical bounds for $\sigma^2$-sub-Gaussian measures $\mu$ and $\nu$,
\begin{equation}\label{eq:EOT_Mena}
	\EV{\abs{\hEOT - \EOT(\mu, \nu)}} \lesssim \eps \left( 1 + \sigma^{\ceil{5d/2}+6}\eps^{-\ceil{5d/4}-3} \right) n^{-1/2}\,,
\end{equation}
where the implicit constant only depends on $d$. For probability measures $\mu$ and $\nu$ that are supported on a joint set of diameter $1$, the latter give for $c$ being the squared Euclidean norm the bound
\begin{equation}\label{eq:EOT_Chizat}
	\EV{\abs{\hEOT - \EOT(\mu, \nu)}} \lesssim (1 + \eps^{-\floor{d/2}}) n^{-1/2}\,.
\end{equation}
More recently, \citet{Rigollet2022} showed for any bounded cost $\c$, without imposing smoothness assumptions, that
\begin{equation}\label{eq:EOT_Rigollet}
	\EV{\abs{\hEOT - \EOT(\mu, \nu)}} \lesssim \exp(\norminf{\c} / \eps) n^{-1/2}\,,
\end{equation}
where we again observe an exponential dependency in $\eps^{-1}$ and $\norminf{\c}$.
\end{subequations}
For all these results we see that for fixed $\eps>0$ the empirical EOT cost admits faster rates in $n$ than the empirical unregularized OT cost.
Such results are complemented by extensive research on distributional limits for the empirical OT cost at scaling rate $n^{1/2}$ which establish the parametric rate to be sharp \citep{bigot2019CentralLT, Mena2019,klatt2020empirical,del2022improved,Goldfeld2022, Goldfeld2022LimitTF, gonzalez2022weak, GonzalezSanz2023, hundrieser2021limit}.

However, as the regularization parameter decreases to zero, the statistical error bound generally deteriorates in high dimensions polynomially or even exponentially in $\epsilon^{-1}$. This behavior is unavoidable, since for $\eps \searrow 0$ the EOT cost tends to the unregularized OT cost which suffers from slower rates. Recalling the setting of $\XX = \YY = [0,1]^d$ with squared Euclidean norm as $\c$, it holds by \citet{Eckstein2022}, see also \citet{Genevay2019}, for some constant $K = K(d)>0$ depending on $d$ that
\begin{equation*}
	\sup_{\mu,\nu \in \sP([0,1]^d)} |\OT_{\c, \eps}(\mu, \nu) - \OT_\c(\mu, \nu)| \leq d \eps \abs{\log \eps} + K \eps\,.
\end{equation*}
Combining this with the previously mentioned minimax result for the OT cost by \citet{manole2021sharp} yields,
\begin{equation*}
	\inf_{\widehat{\OT}} \sup_{\mu, \nu\in \sP([0,1]^d)} \EV{ \abs{\widehat{\OT} - \OT_{\c, \eps}(\mu, \nu)}} \gtrsim (n \log n)^{-2/d} - d \eps \abs{\log \eps} - K \eps\,,
\end{equation*}
where the infimum is taken over all measurable functions $\widehat{\OT}$ of $X_1, \ldots, X_n$ and $Y_1, \ldots, Y_n$. Hence, we see that for $\eps \searrow 0$ the statistical error in estimating the EOT cost in terms of $\eps$ must be affected by the ambient dimension $d$ and that the empirical EOT cost is also burdened by the curse of dimensionality.

Nevertheless, in practical contexts it is often reasonable to expect that the data obeys some additional structure, e.g., that it is concentrated on a low-dimensional domain. The validity of this hypothesis is reflected by the performance of nonlinear dimension reduction techniques like manifold learning \citep{lin2008riemannian,talwalkar2008large,zhu2018image}. Indeed, as for unregularized OT, if $\mu$ and $\nu$ are supported in $[0, 1]^s \times \mset{0}^{d-s}$, the bounds \eqref{eq:eot_bounds} hold for $d$ replaced by $s$, indicating that empirical EOT depends on the intrinsic dimension. Qualitative results of this type were verified by \citet{bayraktar2022stability} for settings where both measures $\mu$ and $\nu$ lie on $s$-dimensional domains. However, their results only assert slower than parametric convergence rates and do not shed light on the dependency of the constants in~$\epsilon$.

\subsection{Contributions}

The main primitive of this work is to provide a comprehensive statistical analysis of the empirical EOT cost with respect to $n$ and $\epsilon$ in terms of the intrinsic dimension of the (possibly different) underlying measures. Fueled by findings of \citet{Hundrieser2022}, which assert that empirical (unregularized) OT adapts to lower complexity, we show that the empirical EOT cost estimator also obeys this principle and thus enjoys better constants in $\epsilon^{-1}$ if only one measure admits low intrinsic dimension.

To formulate our general LCA result, we first show for Polish spaces $\XX$ and $\YY$ with a bounded and measurable cost function $\c : \XX \times \YY \to \R$ that the EOT cost admits a dual formulation over a single function class $\FC{}$, defined in \eqref{eq:def_fc} in Section \ref{sec:eot_duality_and_complexity}, which depends on the cost function $\c$ and the regularization parameter $\epsilon$ (\autoref{thm:eot_duality}),
\begin{equation*}
	\EOT(\mu, \nu) = \max_{\phi \in \FC{}} \int_{\XX} \phi \de{\mu} + \int_{\YY} \entt{\phi}{\mu} \de{\nu}\,,
\end{equation*}
where $\entt{\phi}{\mu}(\argdot)\defeq -\eps \log\int_{\XX} \exp[\eps^{-1}(\phi(x) - \c(x, \argdot))] \dP{\mu}{x}$ is the entropic $(\c, \eps)$-transform of $\phi$. Based on this formula, it follows from techniques of empirical process theory that the mean absolute deviation in \eqref{eq:general_LCA_bound} can be suitably bounded if the complexity of the function classes $\FC{}$ and $\entt{\FC{}}{\mu} = \mset{\entt{\phi}{\mu} \mid \phi \in \FC{}}$ as well as the empirical function class $\entt{\FC{}}{\hat\mu_n} = \mset{\entt{\phi}{\hat\mu_n} \mid \phi \in \FC{}}$ can be suitably quantified. We characterize this via \emph{covering numbers} w.r.t.\ to the uniform norm $\norminf{}$ for scale $\delta > 0$, defined by
\begin{equation*}
	\covnum{\delta}{\FC{}}{\norminf{}} \defeq
 	\inf \mset{ n \in \N \mid \exists f_1, \ldots,f_n : \XX \to \R \text{ s.t. } \sup_{f \in \sF_\c} \min_{1 \leq i \leq n} \norminf{f - f_i} \leq \delta}\,.
\end{equation*}
Notably, the logarithm of the uniform covering number is called the \emph{uniform metric entropy}.

A simple yet crucial observation for our approach is that the covering number of the function classes $\entt{\FC{}}{\mu}$ and $\entt{\FC{}}{\hat\mu_n}$ at any scale $\delta > 0$ is linked to that of $\FC{}$ (\autoref{lemma:entt_cov_num}),
\begin{equation*}
	\max\left(\covnum{\delta}{\entt{\FC{}}{\mu}}{\norminf{}}, \covnum{\delta}{\entt{\FC{}}{\hat\mu_n}}{\norminf{}}\right) \leq \covnum{\delta/2}{\FC{}}{\norminf{}}\,.
\end{equation*}
Then, assuming  the existence of constants $K_\eps, k > 0$ such that for all $\delta>0$ sufficiently small,
\begin{equation}\label{eq:Intro_ComplexityFCBound}
	\log \covnum{\delta}{\FC{}}{\norminf{}} \leq K_{\eps} \delta^{-k}\,,
\end{equation}
we show in \autoref{thm:eot_lca} for all probability measures $\mu \in \sP(\XX)$, $\nu \in \sP(\YY)$ and any of the empirical estimators $\hEOT$ in \eqref{eq:eot_emp_est} based on independent samples that
\begin{equation}\label{eq:general_LCA_bound}
	\EV{\abs{\hEOT - \EOT(\mu, \nu)}} \lesssim \sqrt{1 + K_\eps} \begin{cases}
		n^{-1/2} & k < 2 \,, \\
		n^{-1/2} \log (n+1) & k = 2 \,, \\
		n^{-1/k} & k > 2\,.
	\end{cases}
\end{equation}
As we demonstrate in \autoref{sec:sample_comp}, suitable bounds as in \eqref{eq:Intro_ComplexityFCBound} can be guaranteed if the space $\XX$ and the partially evaluated cost $\{c(\argdot,y)\}_{y\in\YY}$ are sufficiently regular, while the space $\YY$ can be arbitrary, highlighting the adaptation to lower complexity.

For instance, in the semi-discrete setting, i.e., where $\XX$ consists of finitely many points we infer the parametric rate $n^{-1/2}$ without $\epsilon$-dependency (\autoref{thm:lca_sd}). Moreover, under Lipschitz continuous costs on metric spaces (\autoref{thm:lca_lip}) or semi-concave costs on Euclidean domains (\autoref{thm:lca_sc}), our theory implies $\epsilon$-independent rates which are of order $n^{-1/2}$ in low dimensions and slower in higher dimensions, matching those of the empirical OT cost. Under sufficiently smooth costs on Euclidean domains, our approach is also capable in asserting parametric rate $n^{-1/2}$ with lead constant in $\eps^{-d/2}$, where $d$ corresponds to the minimum dimension of the two domains $\XX$ and $\YY$, and can be arbitrarily large (\autoref{thm:lca_hoelder}). Tailored to the squared Euclidean norm, we also build on bounds by \citet{Mena2019} and demonstrate that an LCA principle also remains valid for sub-Gaussian measures (\autoref{thm:lca_sg}).

In \autoref{sec:comput_comp} we show that the Sinkhorn algorithm can be used to compute estimators that also fulfill the bound \eqref{eq:general_LCA_bound}. In \autoref{sec:gromov_wasserstein} we confirm, as an application of our theory, that the (entropic) Gromov-Wasserstein distance also obeys the LCA principle. In \autoref{sec:sims} we perform diverse Monte Carlo simulations which highlight that the implications of the LCA principle can be observed numerically. Lastly, in \autoref{sec:discussion} we discuss possible directions for future research. For the sake of exposition most proofs are deferred to \autoref{appendix:proofs}. \autoref{sec:unif_metric_entropy} contains auxiliary properties on uniform covering numbers, while \autoref{sec:rescaling_prop} gives various rescaling properties used throughout this work.

\subsection{Concurrent Work}

Parallel to this work, \citet{Stromme2023} derived statistical bounds for the empirical EOT cost which complement our work on the LCA principle. For compactly supported probability measures $\mu$ and $\nu$ on $\R^d$ and an $L$-Lipschitz continuous cost function $\c$, it is shown that
\begin{equation}\label{eq:BoundStromme}
	\EV{\abs{\hEOT - \EOT(\mu, \nu)}} \lesssim (1 + \eps) \sqrt{\covnum{\eps/L}{\mu}{\norm{}_2} \land \covnum{\eps/L}{\nu}{\norm{}_2} } n^{-1/2}\,,
\end{equation}
where $\covnum{\eps/L}{\mu}{\norm{}_2} \defeq \covnum{\eps/L}{\supp(\mu)}{\norm{}_2}$ is the covering number of the support of $\mu$ (and analogous for $\nu$), a phenomenon \citeauthor{Stromme2023} refers to as \emph{minimum intrinsic dimension scaling}. The proof employs concavity of the dual formulation and a suitable bound on the $L^2(\mu \otimes \nu)$-norm of the density of the EOT plan in terms of the covering number of the support. In contrast, our results are build on empirical process theory and covering numbers of $\FC{}$ at a scale that depends on the sample size $n$, see \autoref{rem:comparison_stromme} for a refined comparison of the notion of complexities. Notably, the two approaches are distinct and do not imply the results of one or the other, yet their implications are of comparable nature.

\subsection{Notation}

Let $\XX$ and $\YY$ be two Polish spaces that are equipped with their Borel-$\sigma$-fields $\borel(\XX)$ and $\borel(\YY)$, respectively. Denote with $\sP(\XX)$ the set of Borel-probability measures on $\XX$. Random elements are always implicitly defined on an underlying probability space $(\Omega, \sA, \Prob)$ and $\EV{\argdot}$ is the expectation operator. For a function $f : \XX \to \R$ denote the uniform norm $
	\norminf{f} \defeq \sup_{x\in\XX} \abs{f(x)}$.
Provided that $f$ is measurable, define for $p \geq 1$ and a probability measure $\mu \in \sP(\XX)$ the $L^p(\mu)$-norm $\norm{f}_{L^p(\mu)} \defeq \left( \int_{\XX} \abs{f}^p \de{\mu} \right)^{1/p}$.
With $f:\XX \to \R$ and $g : \YY \to \R$, denote the outer sum $f \oplus g : \XX \times \YY \to \R$, $(x, y) \mapsto f(x) + g(y)$. For vectors $x, y \in \R^d$ and $p \geq 1$, we denote the $p$-norm by $\norm{x}_p \defeq ( \sum_{i=1}^{d} \abs{x_i}^p)^{1/p}$ and the inner product by $\scalp{x}{y} \defeq \sum_{i=1}^d x_i y_i$. Denote with $I_d \in \R^{d \times d}$ the $d$-dimensional identity matrix and for a matrix $A$ write $\norm{A}_2$ for its Frobenius norm. Given a subset $\XX \subseteq\R^d$ we write $\mathring{\XX}$ to denote its interior. For $n \in \N$, we write $\idn{n} \defeq \mset{1, \ldots, n}$. To denote the floor and ceiling function evaluated at $a\in \RR$ we write $\lfloor a \rfloor$ and $\lceil a \rceil$, respectively. Further, the minimum and maximum of $a,b \in \R$ are denoted as $a \land b$ and $a \lor b$, respectively.

\section{Entropic Optimal Transport} \label{sec:eot}

In this section, we first introduce preliminary facts about EOT (\autoref{sec:eot_duality_and_complexity}), and then derive the statistical rates for the EOT cost which reflect the LCA principle based on the dual formulation (\autoref{sec:eot_lca_dual}) as well as a projection approach (\autoref{sec:eot_lca_primal}). Unless stated otherwise proofs are deferred to \autoref{appendix:proofs}. Our general results are stated under the following assumption on the cost function $\c : \XX \times \YY \to \R$.

\begin{assumption}{C} \label{ass:eot_cost}
	The cost function $\c$ is measurable and bounded in absolute value by $1$.
\end{assumption}

Note that \autoref{ass:eot_cost} essentially demands that the cost function $\c$ is bounded. Indeed, if $\norminf{\c} \in (1, \infty)$, we can rescale the cost function as detailed in \autoref{rem:eot_rescale_mean_abs_dev} in \autoref{sec:rescaling_prop} and still infer quantitative convergence statements for the original cost.

\subsection{Duality and Complexity} \label{sec:eot_duality_and_complexity}

Central to our approach for the analysis of the empirical EOT cost is a suitable dual representation. To this end, we follow first the work by \citet{Marino2020}. For $\eps > 0$ we define the function $\exp_\eps \defeq \exp(\argdot / \eps)$ and the set
\begin{equation*}
	\Lexp(\mu) \defeq \mset*{ \phi : \XX \to [-\infty, \infty) \text{ measurable with } 0 < \int_{\XX} \exp_\eps(\phi) \de{\mu} < \infty} \,,
\end{equation*}
and analogously $\Lexp(\nu)$. Further, for $\phi \in \Lexp(\mu)$, $\psi \in \Lexp(\nu)$ we define the $\eps$-entropic cost-transform w.r.t.\ $\c$, abbreviated by $(\c, \eps)$-transform, as
\begin{align*}
	\entt{\phi}{\mu}(y) &\defeq - \eps \log \int_{\XX} \exp_\eps(\phi(x) - \c(x, y)) \dP{\mu}{x}\,, \quad y \in \YY\,, \\
	\entt{\psi}{\nu}(x) &\defeq - \eps \log \int_{\YY} \exp_\eps(\psi(y) - \c(x, y)) \dP{\nu}{y}\,, \quad x \in \XX\,.
\end{align*}
Note that $\entt{\phi}{\mu}$ and $\entt{\psi}{\nu}$ are again measurable functions. Lastly, we state the entropy-Kantorovich functional
\begin{equation*}
	\D{\mu}{\nu}(\phi, \psi) \defeq \int_{\XX} \phi \de{\mu} + \int_{\YY} \psi \de{\nu} - \eps \int_{\XX \times \YY} \exp_\eps(\phi \oplus \psi - \c) \de{[\mu \otimes \nu]} + \eps \,.
\end{equation*}
With this notation at our disposal, we can state a general dual formulation of EOT.

\begin{theorem}[{\citealt[Theorem~2.8 and Proposition~2.12]{Marino2020}}] \label{thm:eot_gen_duality}
	Let \autoref{ass:eot_cost} hold. Then, it holds for all probability measures $\mu \in \sP(\XX)$ and $\nu \in \sP(\YY)$ that
	\begin{equation} \label{eq:eot_gen_duality}
		\EOT(\mu, \nu) = \max_{\substack{\phi \in \Lexp(\mu) \\ \psi \in \Lexp(\nu)}} \D{\mu}{\nu}(\phi, \psi)\,.
	\end{equation}
	In particular, optimizers exist and a pair $(\phi, \psi) \in \Lexp(\mu) \times \Lexp(\nu)$ is a maximizer of the above if and only if
	\begin{equation} \label{eq:eot_dual_opt_as}
		\phi = \entt{\psi}{\nu} \text{ $\mu$-a.s.} \quad \text{ and } \quad \psi = \entt{\phi}{\mu} \text{ $\nu$-a.s.,}
	\end{equation}
	and they can be chosen such that $\norminf{\phi}, \norminf{\psi} \leq 3 / 2$.
	Furthermore, given a maximizing pair $(\phi, \psi)$ in \eqref{eq:eot_gen_duality}, the EOT plan in \eqref{eq:def_EOT} is given by
	\begin{equation*}
		\de{\pi} \defeq \exp_\eps(\phi \oplus \psi - \c) \de{[\mu \otimes \nu]}\,.
	\end{equation*}
	Conversely, if there are potentials $\phi$, $\psi$ such that $\pi \in \Pi(\mu, \nu)$, they are optimal.
\end{theorem}

\begin{remark}[Canonical extension] \label{rem:eot_dual_opt}
	\autoref{thm:eot_gen_duality} only asserts that a maximizing pair $\phi$, $\psi$ satisfies the optimality conditions \eqref{eq:eot_dual_opt_as} $\mu$- and $\nu$-almost surely. By defining $\tphi \defeq \entt{\psi}{\nu}$ and $\tpsi \defeq \entt{(\entt{\psi}{\nu})}{\mu}$ we can uniquely extend the potentials beyond the support of the underlying measures. Further, by possibly shifting the potentials $(\tphi, \tpsi)$ by $(a, -a)$ for suitable $a\in \R$ \citep[see][Lemma~2.7]{Marino2020}, it still holds that $\norminf{\tphi}, \norminf{\tpsi} \leq 3/2$.
\end{remark}

The canonical extension of dual potentials allows us to further rewrite the dual formulation in \eqref{eq:eot_gen_duality} to reduce it to a supremum over a single function class.

\begin{proposition}[Duality] \label{thm:eot_duality}
	Let \autoref{ass:eot_cost} hold and define the function class
	\begin{equation} \label{eq:def_fc}
		\FC{} \defeq \bigcup_{\xi \in \sP(\YY)} \mset*{
			\begin{aligned}
				&\phi : \XX \to \R \text{ such that } \exists \psi : \YY \to \R \text{ with } \\
				& \phi = \entt{\psi}{\xi} \text{ and } \norminf{\phi}, \norminf{\psi} \leq 3/2
		\end{aligned}}\,.
	\end{equation}
	Then, for any $\mu \in \sP(\XX)$ and $\nu \in \sP(\YY)$ it holds that
	\begin{equation} \label{eq:fc_duality}
		\EOT(\mu, \nu) = \max_{\phi \in \FC{}} \int_{\XX} \phi \de{\mu} + \int_{\YY} \entt{\phi}{\mu} \de{\nu} \,.
	\end{equation}
\end{proposition}

The dual representation of the EOT cost in \autoref{thm:eot_duality} implies the following stability bound for the EOT cost with respect to the underlying measures.

\begin{lemma}[Stability bound] \label{lemma:eot_stab_bound}
	Let \autoref{ass:eot_cost} hold. Then, it holds for any pairs of probability measures $\mu, \tmu \in \sP(\XX)$ and $\nu, \tnu \in \sP(\YY)$ that
	\begin{equation*}
		\abs{\EOT(\tmu, \tnu) - \EOT(\mu, \nu)}
		\leq 2 \sup_{\phi \in \FC{}} \abs*{\int_{\XX} \phi \de{[\tmu - \mu]}} + 2 \sup_{\phi \in \FC{}} \abs*{ \int_{\YY} \entt{\phi}{\tmu} \de{[\tnu - \nu]}}\,.
	\end{equation*}
\end{lemma}

Plugging in the empirical probability measures $\tmu = \hmu_n$ and $\tnu = \hnu_n$ in \autoref{lemma:eot_stab_bound}, it follows that the absolute difference between empirical estimators of the EOT cost is dominated by the sum of two suprema of empirical processes. Expectations of such quantities can be bounded with methods from empirical process theory by controlling the uniform metric entropy of the indexing function classes $\FC{}$ and $\entt{\FC{}}{\hmu_n} = \mset{\entt{\phi}{\hmu_n} \mid \phi \in \FC{}}$ \citep[see, e.g.,][Chapter~4]{Wainwright2019}. An important observation in this context is that the uniform metric entropy of a function class never considerably increases under entropic $(\c, \eps)$-transforms.

\begin{lemma} \label{lemma:entt_cov_num}
Let $\c\colon \XC\times \YC \to \RR$ be a measurable cost function that is bounded from below and let $\tmu \in \sP(\XX)$. Consider a function class $\sF\subseteq \Lexp(\tmu)$ on $\XX$. Then, it holds for the $(\c, \eps)$-transformed function class $\entt{\sF}{\tmu}\coloneqq \{\entt{\phi}{\tmu} \mid \phi \in \sF\}$ and any $\delta > 0$ that
	\begin{equation*}
		\covnum{\delta}{\entt{\sF}{\tmu}}{\norminf{}} \leq \covnum{\delta / 2}{\sF}{\norminf{}} \,.
	\end{equation*}
\end{lemma}

A similar result as in \autoref{lemma:entt_cov_num} was recently established by \citet{Hundrieser2022} in the context of unregularized OT for the unregularized $c$-transform, and serves as the basis of the LCA principle. In the following subsection we use this insight to establish results of similar nature for the empirical EOT cost.

\subsection{Lower Complexity Adaptation: Dual Perspective} \label{sec:eot_lca_dual}

As described before, \autoref{lemma:eot_stab_bound} together with tools from classical empirical process theory can be used to bound the statistical error in terms of the uniform metric entropy of the function classes $\FC{}$ and $\entt{\FC{}}{\hmu_n}$. An application of \autoref{lemma:entt_cov_num} further reduces this to controlling the complexity of $\FC{}$.

\begin{theorem}[General LCA] \label{thm:eot_lca}
	Let \autoref{ass:eot_cost} hold.
	Assume there exist constants $k > 0$, $K_\eps > 0$ and $\delta_0 \in (0, 1]$ such~that
	\begin{equation} \label{eq:ass_cov_num}
		\log\covnum{\delta}{\FC{}}{\norminf{}} \leq K_\eps \delta^{-k}\qquad \text{for } 0 < \delta \leq \delta_0\,.
	\end{equation}
	Then, for all probability measures $\mu \in \sP(\XX)$, $\nu \in \sP(\YY)$ and any of the empirical estimators $\hEOT$ from \eqref{eq:eot_emp_est} for i.i.d.\ random variables $X_1, \dots, X_n\sim \mu$ and (independent to that) $Y_1, \dots, Y_n\sim \nu$ it holds that
	\begin{equation*}
		\EV{\abs{\hEOT - \EOT(\mu, \nu)}} \lesssim \sqrt{1 + K_\eps} \begin{cases}
			n^{-1/2} & k < 2 \,, \\
			n^{-1/2} \log (n+1) & k = 2 \,, \\
			n^{-1/k} & k > 2\,,
		\end{cases}
	\end{equation*}
	where the implicit constant only depends on $k$ and $\delta_0$.
\end{theorem}

This theorem establishes the LCA principle for the empirical EOT cost. More precisely, under regularity of the cost function $\c$ and constraints on the complexity of the space $\XX$, we will see that suitable bounds on the covering numbers of $\FC{}$ are available, which do not depend on the complexity of $\YY$. Intuitively, the two measures $\mu$ and $\nu$ should be understood as being defined on high-dimensional ambient spaces, but where the measure $\mu$ is concentrated on a low-dimensional domain $\XX$. As a consequence, the statistical error bound for the empirical EOT cost will be suitably bounded if one measure admits low intrinsic dimension. By interchanging the roles of $\mu$ and $\nu$, we notice that the statistical error will depend on the minimum intrinsic dimension.

\begin{proof}[Proof of \autoref{thm:eot_lca}]
	We pursue a similar proof strategy as \citet[Theorem~2.2]{Hundrieser2022}, which is inspired by that of \citet{Chizat2020}, for $\hEOT = \EOT(\hmu_n, \hnu_n)$ and note that the remaining one-sample estimators from \eqref{eq:eot_emp_est} can be handled similarly. Using \autoref{lemma:eot_stab_bound}, it holds that
	\begin{align}
		\EV{\abs{\EOT(\hmu_n, \hnu_n) - \EOT(\mu, \nu)}}\notag
		&\leq 2 \EV*{\sup_{\phi \in \FC{}} \abs*{\int_{\XX} \phi \de{[\hmu_n - \mu]}}} \\
		&\qquad\qquad + 2 \EV*{\sup_{\psi \in \entt{\FC{}}{\hmu_n}} \abs*{ \int_{\YY} \psi \de{[\hnu_n-\nu]}}}\,.\label{eq:decouplingBound}
	\end{align}
	An application of Proposition~4.11 from \citet{Wainwright2019} yields
	\begin{equation*}
		\EV{\abs{\EOT(\hmu_n, \hnu_n) - \EOT(\mu, \nu)}} \leq 4 [ \sR_n(\FC{}) + \sR_n(\entt{\FC{}}{\hmu_n}) ]\,,
	\end{equation*}
	where $\sR_n(\FC{})$ and $\sR_n(\entt{\FC{}}{\hmu_n})$ are the Rademacher complexities of $\FC{}$ and $\entt{\FC{}}{\hmu_n}$, respectively, i.e., set
	\begin{equation*}
		\sR_n(\FC{}) \defeq \EV*{ \sup_{\phi \in \FC{}} \abs*{ \frac{1}{n} \sum_{i=1}^{n} \sigma_i \phi(X_i) }},
	\end{equation*}
	where $\sigma_1, \ldots, \sigma_n \sim \unifdist\mset{\pm 1}$ are i.i.d.\ Rademacher variables that are independent of $X_1, \ldots, X_n \sim \mu$, and define $\sR_n(\entt{\FC{}}{\hmu_n})$ similarly. Note that $\FC{}$ and $\entt{\FC{}}{\hmu_n}$ are both bounded in uniform norm by $3$.
	Using Theorem~16 from \citet{Luxburg2004}, we have
	\begin{align*}
		\sR_n(\FC{} / 3)
		&\leq \inf_{\delta \in [0, 1]} \left( 2 \delta + \sqrt{32} n^{-1/2} \int_{\delta/4}^{1} \sqrt{\log \covnum{3u}{\FC{}}{\norminf{}}} \de{u} \right).
	\end{align*}
	With \autoref{lemma:entt_cov_num} we obtain by independence of the samples $X_1, \dots, X_n$ and $Y_1, \dots, Y_n$ a similar upper bound for $\sR_n(\entt{\FC{}}{\hmu_n} / 3)$ with $3u$ in the covering number replaced by $\frac{3}{2}u$, i.e.,
	\begin{align*}
		\sR_n(\entt{\FC{}}{\hmu_n} / 3) &\leq  \mathbb{E}_{X_1, \dots, X_n}\left[ \mathbb{E}_{\sigma_1, \dots, \sigma_n, Y_1, \dots, Y_n}\left[ \sup_{\phi \in \entt{\FC{}}{\hmu_n} / 3} \abs*{ \frac{1}{n} \sum_{i=1}^{n} \sigma_i \phi(Y_i) } \;\, \Bigg| X_1, \dots, X_n\right]\right]\\
		&\leq \mathbb{E}_{X_1, \dots, X_n}\left[ \inf_{\delta \in [0, 1]} \left( 2 \delta + \sqrt{32} n^{-1/2} \int_{\delta/4}^{1} \sqrt{\log \covnum{\textstyle 3u\displaystyle }{\entt{\FC{}}{\hmu_n}}{\norminf{}}} \de{u} \right) \right]\\
		&\leq \inf_{\delta \in [0, 1]} \left( 2 \delta + \sqrt{32} n^{-1/2} \int_{\delta/4}^{1} \sqrt{\log \covnum{\textstyle \frac{3}{2}u\displaystyle }{\FC{}}{\norminf{}}} \de{u} \right).
	\end{align*}
	From the covering number bound \eqref{eq:ass_cov_num} in conjunction with the fact that covering numbers are non-decreasing for decreasing scale we conclude
	\begin{align}
		&\EV{\abs{\EOT(\hmu_n, \hnu_n) - \EOT(\mu, \nu)}} \notag \\
		&\qquad\leq 12 [ \sR_n(\FC{}/3) + \sR_n(\entt{\FC{}}{\hmu_n} /3) ] \notag\\
		&\qquad\leq 24 \inf_{\delta \in [0, 1]} \left( 2 \delta + \sqrt{32} n^{-1/2} \int_{\delta/4}^{1} \sqrt{\log \covnum{\textstyle\frac{3}{2}u\displaystyle}{\FC{}}{\norminf{}}} \de{u} \right) \notag\\
		&\qquad\leq 24 \inf_{\delta \in [0, 1]} \left( 2 \delta + \sqrt{32 K_\eps} n^{-1/2} \int_{\delta/4}^{1} (\textstyle[\frac{3}{2}u]\displaystyle\land \delta_0)^{-k/2} \de{u} \right) \label{eq:bound_mean_abs_dev_delta}
	\end{align}
	and the choices of $\delta \defeq 4 n^{-1/(k \lor 2)}$ yield the desired result.
\end{proof}

In \autoref{thm:eot_lca}, we see that the statistical error bound may depend on the entropic regularization parameter $\eps$ through the term $K_\eps$. Later in \autoref{subsec:hoelder_cost}, we observe a trade-off between $K_\eps$ and $k$ in leveraging the smoothness of the cost function. Namely, the rates in $n$ improve with higher degree of smoothness while the dependency on small $\eps$ gets worse. As a consequence, depending on how fast $\eps$ tends to $0$ with increasing sample size, an error bound leveraging less smoothness can assert a smaller mean absolute deviation of the empirical plug-in estimator.

\begin{corollary}[Comparison of rates] \label{cor:rates_eps_tradeoff}
	Let \autoref{ass:eot_cost} hold. Suppose that there exist constants $k_1$, $k_2$, $a > 0$ such that for sufficiently small $\delta > 0$ it holds that
	\begin{equation*}
		\log\covnum{\delta}{\FC{}}{\norminf{}} \lesssim \min\mset{\eps^{-a} \delta^{-k_1},\, \delta^{-k_2}}\,.
	\end{equation*}
	Choose $\eps = n^{-\gamma}$ for some $\gamma > 0$. Then, for all probability measures $\mu \in \sP(\XX)$, $\nu \in \sP(\YY)$ and any of the empirical estimators $\hEOT$ from \eqref{eq:eot_emp_est} it holds (up to $\log(n+1)$-factors) that
	\begin{equation*}
		\EV{\abs{\hEOT - \EOT(\mu, \nu)}} \lesssim \min\mset{n^{a\gamma / 2 - 1 / (k_1 \lor 2)},\, n^{-1/(k_2 \lor 2)}}\,.
	\end{equation*}
	In particular, it follows that the second term in the minimum is smaller if and only if
	\begin{equation*}
		\gamma > \left[ \frac{1}{k_1 \lor 2} - \frac{1}{k_2 \lor 2} \right] \frac{2}{a}\,.
	\end{equation*}
\end{corollary}

\begin{remark}[Unbounded costs]
	For the formulation of \autoref{thm:eot_lca} we impose boundedness of the cost function $\c$ via \autoref{ass:eot_cost} which ensures that dual potentials are appropriately bounded without assuming concentration properties of the underlying probability measures. Later, in \autoref{subsec:sub_Gaussian} we employ a similar approach to show the validity of the LCA principle for squared Euclidean costs in the setting where one measure is sub-Gaussian while the other is compactly supported and concentrated on a low-dimensional domain.
\end{remark}

\begin{remark}[Comparison with complexity scales of \citealt{Stromme2023}] \label{rem:comparison_stromme}
	The proof of \autoref{thm:eot_lca}, see Inequality \eqref{eq:bound_mean_abs_dev_delta},  illustrates that condition \eqref{eq:ass_cov_num} can be relaxed to an upper bound on the metric entropy of $\FC{}$ at scales $\delta \geq 4n^{-1/(k \lor 2)}$, independent of $\epsilon$, i.e.,
	\begin{align*}
		\log\covnum{\delta}{\FC{}}{\norminf{}} \leq K_\eps \delta^{-k} \qquad \text{for } 4n^{-1/(k \lor 2)} \leq \delta \leq \delta_0\,.
	\end{align*}  This highlights that the empirical EOT cost serves as an effective estimator for the population EOT cost, similar to unregularized OT \citep{weed2019sharp,Hundrieser2022,manole2021sharp}, if the uniform metric entropy of $\FC{}$ is small at scales of order $n^{-1/(k \lor 2)}$. Moreover, it complements the concurrently derived bound \eqref{eq:BoundStromme} by \citet{Stromme2023} which asserts that the empirical EOT cost also performs well if the minimum covering number of the supports of the underlying measures at an $\epsilon$-dependent scale is sufficiently small. Collectively, our results and those of \citet{Stromme2023} show that empirical EOT benefits from two \emph{distinct} notions of covering number bounds operating at different scales. We~like to stress that these two notions differ in their implications and allow for a trade-off in terms of the dependency in $n$ and $\eps$: While the bound \eqref{eq:BoundStromme} by \citet{Stromme2023} achieves the parametric rate $n^{-1/2}$ at the cost of a strong dependency in $\eps$, our results allow for potentially improved dependency in $\epsilon$ with possibly slower than parametric rates in~$n$.
\end{remark}

\begin{remark}[Comparison with unregularized OT]
	As opposed to the entropic $(c, \eps)$-transform, for unregularized OT the corresponding $c$-transform does not dependent on the probability measures.
	This allows for an LCA principle based on two function class $\sF_c$, $\sF_c^c$ that are independent of the underlying probability measures \citep{Hundrieser2022}. In our setting, it hinges on the two classes $\FC{}$, $\entt{\FC{}}{\hmu_n}$, where one is dependent on an empirical measure and thus causes additional technical difficulties. In particular, due to the complicated dependency of $\entt{\FC{}}{\hmu_n}$ on $\hmu_n$, we are limit ourselves to imposing independence between the two samples $X_1, \ldots, X_n$ and $Y_1, \ldots, Y_n$. The smooth nature of the entropic $(c, \eps)$-transform permits however (see \autoref{subsec:hoelder_cost}) a refinement in terms of the constants: The LCA principle for unregularized OT can only leverage smoothness up to order~$2$ of the underlying cost function, whereas here we benefit from arbitrarily large degrees of smoothness.
\end{remark}

\subsection{Lower Complexity Adaptation: Primal Perspective} \label{sec:eot_lca_primal}

Our main result for the validity of the LCA principle (\autoref{thm:eot_lca}) employs metric entropy bounds, however does not provide much intuition for why the LCA principle is reasonable to expect. We therefore explore in the following a simple setting where the EOT plan entails a projective action in order to foster some additional insight about the LCA principle.

\begin{proposition}[{\citealt[Theorem~4.2 and Remark~4.4]{nutz2021introduction}}] \label{thm:eot_lca_primal}
	Let $\YY = \YY_1 \times \YY_2$ be the product of two Polish spaces and suppose that the measurable cost function $\c$ decomposes for $x \in \XX$ and $y = (y_1, y_2) \in \YY$ into
	\begin{equation}\label{eq:dec_costs}
		\c(x, y) = \c_1(x, y_1) + \c_2(y_2)\,,
	\end{equation}
	where $\c_1 : \XX \times \YY_1 \to \R$ and $\c_2 : \YY_2 \to \R$. Then, it follows for all probability measures $\mu \in \sP(\XX)$ and $\nu \in \sP(\YY)$ such that $\c_1 \in L^1(\mu \otimes \nu_1)$ and $\c_2 \in L^1(\nu_2)$ that
	\begin{equation*}
		\EOT(\mu, \nu) = \OT_{\c_1, \eps}(\mu, \nu_1) + \int_{\YY_2} \c_2 \de{\nu_2}\,,
	\end{equation*}
	where $\nu_1$ and $\nu_2$ are the marginals of $\nu$ on $\YY_1$ and $\YY_2$, respectively.
\end{proposition}

The representation of the EOT cost under cost functions of the form \eqref{eq:dec_costs} asserts that the EOT plan can be decomposed into a projective action $\nu \mapsto \nu_2$, which contributes a cost of the magnitude $\int_{\YY_2} \c_2 \de{\nu_2}$, and the EOT plan between $\mu$ and $\nu_1$, which causes the term $\OT_{\c_1, \eps}(\mu, \nu_1)$.

Since for the squared Euclidean norm \autoref{thm:eot_lca_primal} remains valid for affine subspaces, we obtain the following representation. To ease notation, for two random variables $X$ and $Y$ on $\R^d$ we define $\EOT(X, Y)$ as the EOT cost between their laws.

\begin{corollary}[Squared Euclidean costs] \label{thm:eot_lca_ortho}
	Let $\c(x, y) = \norm{x - y}_2^2$ be the squared Euclidean norm on $\R^d$. Let $s \leq d$, $U \in \R^{d \times s}$ be orthogonal\,\footnote{By this we mean that the matrix $U \in \R^{d \times s}$ fulfills the relation $U^\transp U = I_s$.} and $v \in \R^d$. Then, it holds for random variables $X$ and $Y$ with finite second moments on $\R^s$ and $\R^d$, respectively, that
	\begin{equation*}
		\EOT(UX + v, Y) = \EOT(X, U^\transp (Y - v)) + \EV{\norm{(I_d - UU^\transp)[Y - v]}_2^2}\,.
	\end{equation*}
\end{corollary}

\begin{example}[Unit cubes]\label{expl:unitCube}
Consider $\XX = \YY_1= [0, 1]^{d_1}$ and $\YY_2 = [0,1]^{d_2}$ and let $\c$ be the squared Euclidean norm where we embed $\XX$ into $\YY= \YC_1\times \YC_2$ by appending zeros (or embed via affine transformation). Then, it holds for $x \in \XX$ and $y = (y_1, y_2) \in \YY$ that
\begin{equation*}
	\c(x, y) = \norm{x - y_1}^2_2 + \norm{y_2}_2^2\,,
\end{equation*}
and \autoref{thm:eot_lca_primal} (or \autoref{thm:eot_lca_ortho}) reduces the $(d_1 + d_2)$-dimensional EOT problem to only $d_1$ dimensions.
\end{example}

The representation of the EOT cost in \autoref{thm:eot_lca_primal} yields the following additional interpretation of the LCA principle. Let $\mu \in \sP(\XX)$, $\nu \in \sP(\YY)$ with empirical versions $\hmu_n$, $\hnu_n$, then under the assumptions of \autoref{thm:eot_lca_primal} we see that
\begin{align*}
	\EV{\abs{\EOT(\hmu_n, \hnu_n) - \EOT(\mu, \nu)}} &\leq \EV{\abs{\OT_{\c_1, \eps}(\hmu_n, \hnu_{n,1}) - \OT_{\c_1,\eps}(\mu, \nu_1)}} \\
	&\qquad\qquad + \EV*{ \abs*{ \int_{\YY_2} \c_2 \de{[\hnu_{n,2} - \nu_2]} }}\,.
\end{align*}
The second term admits under a finite second moment $\int_{\YY_2} \c_2^2 \de{\nu_2}<\infty$ a statistical error of order $n^{-1/2}$ that is independent of $\eps$. Hence, the $\eps$-dependence of the statistical error for $\EOT(\hmu_n, \hnu_n)$ reduces to that of $\OT_{\c_1, \eps}(\hmu_n, \hnu_{n,1})$ which depends on the simpler measure $\hnu_{n,1}$. In particular, the convergence rate for $\OT_{\c_1, \eps}(\hmu_n, \hnu_{n,1})$ only depends on the regularity of $\c_1$ on the smaller space $\XX\times \YY_1$ (compared to $\XC\times \YC$).

Moreover, note that \autoref{thm:eot_lca_primal} can also be employed to obtain lower bounds for the statistical error of empirical EOT. More specially, the reverse triangle yields that
\begin{align*}
	\EV{\abs{\EOT(\hmu_n, \hnu_n) - \EOT(\mu, \nu)}} &\geq \EV*{ \abs*{ \int_{\YY_2} \c_2 \de{[\hnu_{n,2} - \nu_2]} }} \\
	&\qquad\qquad - \EV{\abs{\OT_{\c_1, \eps}(\hmu_n, \hnu_{n,1}) - \OT_{\c_1,\eps}(\mu, \nu_1)}}\,.
\end{align*}
Hence, if the second term on the right-hand side contributes a comparably small statistical error, the statistical error of $\EOT(\hmu_n, \hnu_n)$ approximately matches that of $\int_{\YY_2} \c_2 \de{\hnu_{n,2}}$. This implies that, although the $\epsilon$-dependency is only determined by the less complex space, the underlying ($\epsilon$-independent) constant can be affected by the more complex space.

This is a scaling effect and arises if $\c_2$ admits a large variance w.r.t.\ $\nu$; we will revisit this observation later in our simulations for squared Euclidean costs. To suppress this effect we will mainly work in the subsequent section under \autoref{ass:eot_cost}, i.e., a cost function that is uniformly bounded by one. However, for our result on sub-Gaussian measures in \autoref{subsec:sub_Gaussian} we cannot simply rescale the cost function and arrive at underlying constants which do depend on the ambient dimension (\autoref{thm:lca_sg}).

\section{Sample Complexity} \label{sec:sample_comp}

In this section, we focus on concrete statistical implications of the LCA principle for the empirical EOT cost and employ \autoref{thm:eot_lca} in various scenarios to derive upper bounds for the statistical error. To apply \autoref{thm:eot_lca}, we need to bound the uniform metric entropy of the function class $\FC{}$. Recall that each $\phi \in \FC{}$ can be written as
\begin{equation*}
	\phi(x) = - \eps \log \int_{\YY} \exp_\eps(\psi(y) - \c(x, y)) \dP{\xi}{y}\,,
\end{equation*}
for some function $\psi : \YY \to \R$ and probability measure $\xi \in \sP(\YY)$. Based on this integral representation we see that suitable regularity properties of the partially evaluated cost $\{\c(\argdot, y)\}_{y\in \YC}$ are transmitted to $\phi$, which allows us to bound the uniform metric entropy of $\FC{}$. Another important aspect for these bounds is that the domain $\XX$ is not too complex. We model this by assuming that $\XX$ can be represented as a finite union $\bigcup_{i=1}^I g(\sU_i)$ of $I \in \N$ spaces $\sU_i$ embedded via $g_i : \sU_i \to \XX$ where $\c(g_i(\argdot), y)$ needs to satisfy certain assumptions. This kind of structural assumption encompasses the setting where $\XC$ consists of finitely many points, but also covers the scenario that $\XC$ is a smooth compact sub-manifold embedded in a high-dimensional Euclidean space and where $(g_i, \sU_i)_{i \in \idn{I}}$ corresponds to the collection of charts (see \citealt{lee2013smooth} for comprehensive treatment).

The following subsections explore various settings which are based on this approach: they all rely on imposing regularity on the cost $c$ and the space $\XC$ to establish suitable uniform metric entropy bounds for $\FC{}$. For exposition, we relegate all proofs to Appendix \ref{appendix:A2} while technical insights on the uniform metric entropy are detailed in \autoref{sec:unif_metric_entropy}.

\subsection{Semi-Discrete}

First, we consider the semi-discrete case, i.e., when $\XC$ consists of finitely many points, while $\YC$ is a general Polish space. Then, under uniformly bounded costs, the function class $\FC{}$ can be understood as a set of uniformly bounded vectors and its uniform metric entropy is thus particularly simple to control.

\begin{theorem}[Semi-discrete LCA] \label{thm:lca_sd}
	Let \autoref{ass:eot_cost} hold and suppose that $\XX = \mset{x_1, \ldots, x_I}$. Then, for all probability measures $\mu \in \sP(\XX)$, $\nu \in \sP(\YY)$ and any of the empirical estimators $\hEOT$ from \eqref{eq:eot_emp_est} it holds for an implicit universal constant that
	\begin{equation*}
		\EV{\abs{\hEOT - \EOT(\mu, \nu)}} \lesssim \sqrt{I} n^{-1/2}\,.
	\end{equation*}
\end{theorem}

Notably, the obtained bound is independent of the regularization parameter $\eps$ and the parametric rate $n^{-1/2}$ is attained. This is in line with convergence rates for semi-discrete unregularized OT, see Theorem~3.2 from \citet{Hundrieser2022}, which is the limiting behavior $\eps \searrow 0$. \autoref{thm:lca_sd} also slightly improves upon the bound given by \citet[Example~4]{Stromme2023}, asserting no dependence on $\eps$ and leaving out the condition on Lipschitz continuity of the cost.

\subsection{Lipschitz Cost}

Next, we assume a more general representation for the space $\XC$ and impose a Lipschitz continuity condition on the cost function.

\begin{assumption}{Lip} \label{ass:lip}
	It holds that $\XX = \bigcup_{i=1}^I g_i(\sU_i)$ for $I \in \N$ connected metric spaces $(\sU_i, d_i)$ and maps $g_i : \sU_i \to \XX$ such that $\c(g_i(\argdot), y)$ is $1$-Lipschitz w.r.t.\ $d_i$ for all $y \in \YY$.
\end{assumption}

Using the scaling property from \autoref{rem:eot_rescale_mean_abs_dev}, note that other uniform Lipschitz constants than~$1$ can be reduced to the above case. A central consequence of this condition is that for any element $\phi\in \FC{}$ the composition $\phi\circ g_i\colon \sU_i \to \RR$ is again $1$-Lipschitz. The uniform metric entropy of the class of uniformly bounded, $1$-Lipschitz functions on a metric space is well-studied \citep{Kolmogorov1961} and implies the following result.

\begin{theorem}[Lipschitz LCA] \label{thm:lca_lip}
	Let \autoref{ass:eot_cost} and \autoref{ass:lip} hold and suppose that there exists some $k > 0$ such that for all $i=1,\ldots, I$ it holds that
	\begin{equation*}
		\covnum{\delta}{\sU_i}{d_i} \lesssim \delta^{-k} \qquad \text{for $\delta > 0$ sufficiently small.}
	\end{equation*}
	Then, for all probability measures $\mu \in \sP(\XX)$, $\nu \in \sP(\YY)$ and any of the empirical estimators $\hEOT$ from \eqref{eq:eot_emp_est} it holds for an implicit constant which  only depends on $\sU_1, \dots, \sU_I$ that
	\begin{equation*}
		\EV{\abs{\hEOT - \EOT(\mu, \nu)}} \lesssim
		\begin{cases}
			n^{-1/2} & k < 2 \,, \\
			n^{-1/2} \log (n+1) & k = 2 \,, \\
			n^{-1/k} & k > 2\,.
		\end{cases}
	\end{equation*}
\end{theorem}

The rate for $k \geq 2$ may appear suboptimal in terms of $n$ when compared to the results by \citet{Rigollet2022,Stromme2023}. However, we note that the obtained bound does not depend on the regularization parameter $\eps$. This is due to the fact that it does not affect the Lipschitz constant of the function class $\FC{}$. Similar to \autoref{cor:rates_eps_tradeoff}, we see that the above bound might be superior if $\eps$ decreases sufficiently fast. As for the semi-discrete case, the upper bounds are identical to that of the empirical OT under Lipschitz costs, see Theorem~3.3 from \citet{Hundrieser2022}.

Furthermore, note that \autoref{ass:lip} requires connected metric spaces. For general metric spaces, slightly worse uniform metric entropy bounds are available, see Lemma~A.2 from \citet{Hundrieser2022}. In this setting, the assertion of \autoref{thm:lca_lip} remains valid at the price of an additional $\log(n+1)$-term for $k \geq 2$ \citep[Remark~3.3 and Appendix~B]{Staudt2023Convergence}.

\subsection{Semi-Concave Cost}

In addition to Lipschitz continuity of the cost function $\c$, we now assume semi-concavity. More specifically, we suppose that $\c$ is Lipschitz continuous and semi-concave in the first component with a uniform modulus over the second component. A function $f : \sU\to\R$ on a bounded, convex subset $\sU\subseteq\R^s$ is called $\Lambda$-semi-concave with modulus $\Lambda \geq 0$ if the function
\begin{equation*}
	u \mapsto f(u) - \Lambda \norm{u}_2^2
\end{equation*}
is concave. With this definition at our disposal we state the following set of assumptions.
\begin{assumption}{SC} \label{ass:semi_con}
	It holds that $\XX = \bigcup_{i=1}^I g_i(\sU_i)$ for $I \in \N$ bounded, convex subsets $\sU_i \subseteq \R^s$ and maps $g_i : \sU_i \to \XX$ such that $\c(g_i(\argdot), y)$ is $1$-Lipschitz w.r.t.\ $\norm{}_2$ and $1$-semi-concave for all $y \in \YY$.
\end{assumption}

The additional assumption of semi-concavity improves the available uniform metric entropy bounds for the function class $\FC{}$ \citep{Bronshtein1976, Guntuboyina2013}. Again, invoking the scaling property from \autoref{rem:eot_rescale_mean_abs_dev}, the Lipschitz constants and semi-concavity moduli can be assumed to be $1$.

\begin{theorem}[Semi-concave LCA] \label{thm:lca_sc}
	Let \autoref{ass:eot_cost} and \autoref{ass:semi_con} hold. Then, for all probability measures $\mu \in \sP(\XX)$, $\nu \in \sP(\YY)$ and any of the empirical estimators $\hEOT$ from \eqref{eq:eot_emp_est} it holds for an implicit constant which only depends on $\sU_1, \dots, \sU_I$ that
	\begin{equation*}
		\EV{\abs{\hEOT - \EOT(\mu, \nu)}} \lesssim
		\begin{cases}
			n^{-1/2} & s < 4 \,, \\
			n^{-1/2} \log (n+1) & s = 4 \,, \\
			n^{-2/s} & s > 4\,.
		\end{cases}
	\end{equation*}
\end{theorem}

Note that for $s \geq 4$ the rates are strictly slower than $n^{-1/2}$ but do not depend on $\eps$, recall also \autoref{cor:rates_eps_tradeoff}. In particular, we see in the proof that the regularization parameter $\eps$ does not affect the semi-concavity modulus of the function class $\FC{}$. As for the semi-discrete and Lipschitz case, the rates are identical for OT under semi-concave costs \citep[Theorem 3.8]{Hundrieser2022}.

\subsection{H\"older Cost} \label{subsec:hoelder_cost}

Overall, unregularized OT is not capable of leveraging higher degree of smoothness of the underlying cost function for faster convergence rates \citep{manole2021sharp}. In stark contrast, the entropic $(\c, \eps)$-transform transmits smoothness of the cost to the EOT potentials. Smoothness of the cost function has indeed been employed by several works \citep{Genevay2019,Mena2019,Chizat2020} for upper bounds on the statistical error of the empirical EOT cost with a polynomial dependency in $\eps^{-1}$ determined by the ambient dimension. In what follows, we show that this dimensional dependency obeys the LCA principle.

\begin{assumption}{Hol} \label{ass:hoelder}
	It holds that $\XX = \bigcup_{i=1}^I g_i(\sU_i)$ for $I \in \N$ bounded, convex subsets $\sU_i \subseteq \R^s$ with nonempty interior and maps $g_i : \sU_i \to \XX$ such that $\c(g_i(\argdot), y)$ is $\alpha$-times continuously differentiable with $\alpha \in \N$ and bounded partial derivatives uniformly in~$y \in \YY$.
\end{assumption}

Under this assumption, we show that the function classes $\FC{} \circ g_i$ are $\alpha$-H\"older, for which suitable uniform metric entropy bounds are available \citep[Theorem 2.7.1]{Vaart1996}. To formalize this, we introduce some additional notation. Let $\sU \subseteq \R^s$ be bounded and convex with nonempty interior. For $k \in \idn{s}^{\kappa}$ with $\kappa \in \N$, define $\abs{k} \coloneqq \kappa$ and the differential operator
\begin{equation*}
	\De^k \defeq \frac{\partial^\kappa }{\partial u_{k_\kappa} \cdots \partial u_{k_1}}\,.
\end{equation*}
Furthermore, for $\alpha > 0$ let $\sint{\alpha} = \max\mset{m \in \N_0 \mid m < \alpha}$ and for a function $f : \sU \to \R$ set
\begin{equation*}
	\norm{f}_\alpha \defeq \max_{\abs{k} \leq \sint{\alpha}} \sup_{u} \norm{\De^k f(u)} + \max_{\abs{k} = \sint{\alpha}} \sup_{u,v} \frac{\abs{\De^kf(u) - \De^kf(v)}}{\norm{u-v}^{\alpha - \sint{\alpha}}}\,,
\end{equation*}
where the supremum is taken over $u,v \in \mathring{\sU} : u \neq v$. In addition, we consider for $M > 0$ the class of $\alpha$-H\"older functions on $\sU$ with norm bounded by $M$,
\begin{equation*}
	\sC^\alpha_M(\sU) \defeq \mset{f : \sU \to \R \text{ continuous with } \norm{f}_\alpha \leq M}\,.
\end{equation*}

Under \autoref{ass:eot_cost} and \autoref{ass:hoelder}, a simple consequence of the dominated convergence theorem is that the first $\alpha$ partial derivatives of the functions in $\FC{} \circ g_i$ exist. By definition of the entropic $(\c, \eps)$-transform, it follows that they adhere to a certain recursive structure, see \autoref{lemma:deriv_rec}. Using this, \citet{Genevay2019} show that the $\kappa$-th partial derivatives are bounded in uniform norm by a polynomial in $\eps^{-1}$ of order $\kappa - 1$. We adapt their proof to make the dependence on the cost function more explicit.

\begin{lemma}[Bounds for derivatives] \label{lemma:infnorm_bound}
	Let \autoref{ass:eot_cost} and \autoref{ass:hoelder} hold. Denote the quantities
	\begin{equation*}
		C_{i,m} \defeq \sup_{\abs{j} = m} \norminf{\De^j [\c \circ (g_i, \id_\YY)]}\,.
	\end{equation*}
	and define
	\begin{equation*}
		C^{(i,1)} \defeq C_{i,1}\,, \qquad C^{(i, \kappa+1)} \defeq \max\left(C_{i,\kappa+1}, C^{(i,\kappa)}, \max_{m=1,\ldots,\kappa} C_{i,m} C^{(i,\kappa)}\right)\,.
	\end{equation*}
	Then, it holds for all $\phi \in \FC{}$, $k \in \idn{s}^\kappa$, $\kappa \leq \alpha$ that
	\begin{equation*}
		\norminf{\De^k[\phi \circ g_i]} \lesssim (\eps \land 1)^{-(\kappa - 1)} C^{(i,\kappa)} \,,
	\end{equation*}
	where the implicit constant only depends on $\kappa$.
\end{lemma}

The above bounds can be used to uniformly bound the $\alpha$-H\"older norm of functions in $\FC{} \circ g_i$. As a consequence, we obtain the following uniform metric entropy bound.

\begin{proposition} \label{prop:unif_metric_ent_hoelder}
	Let \autoref{ass:eot_cost} and \autoref{ass:hoelder} hold. Then, it follows for $\delta > 0$ that
	\begin{equation*}
		\log \covnum{\delta}{\FC{}}{\norminf{}} \lesssim
		\left( \sum_{i=1}^{I} [C^{(i, \alpha)}]^{s/\alpha} \right) (\eps \land 1)^{-s \frac{\alpha - 1}{\alpha}} \delta^{-s/\alpha}\,,
	\end{equation*}
	where the implicit constant only depends on $s$, $\alpha$ and $\sU_1, \ldots \sU_I$.
\end{proposition}

An application of \autoref{thm:eot_lca} directly yields the following LCA result.

\begin{theorem}[H\"older LCA] \label{thm:lca_hoelder}
	Let \autoref{ass:eot_cost} and \autoref{ass:hoelder} hold. Then, for all probability measures $\mu \in \sP(\XX)$, $\nu \in \sP(\YY)$ and any of the empirical estimators $\hEOT$ from \eqref{eq:eot_emp_est} it holds that
	\begin{equation*}
		\EV{\abs{\hEOT - \EOT(\mu, \nu)}} \lesssim \left(1 + \sum_{i=1}^I [C^{(i, \alpha)}]^{\frac{s}{2 \alpha}} \right) (\eps \land 1)^{- \frac{s}{2} \frac{\alpha - 1}{\alpha}}
		\begin{cases}
			n^{-1/2} & s/\alpha < 2 \,, \\
			n^{-1/2} \log (n+1) & s/\alpha = 2 \,, \\
			n^{-\alpha/s} & s/\alpha > 2\,,
		\end{cases}
	\end{equation*}
	where the implicit constant only depends on $s$, $\alpha$ and $\sU_1, \ldots, \sU_I$.
\end{theorem}

Note that in the bound of \autoref{thm:lca_hoelder} we have a trade-off between the exponent of $\eps$ and $n$. Namely, with increasing smoothness $\alpha$ the rate in $\eps$ or $n$ get worse or better, respectively. In particular, in the case that \autoref{ass:hoelder} is satisfied for $\alpha > s/2$, we get the parametric rate $n^{-1/2}$. This condition is met by costs that are smooth in the first component with bounded derivatives uniformly over the second.

\begin{example}[Squared Euclidean norm] \label{ex:sq_euclidean_norm}
	Let $B_r(0) \defeq \mset{x \in \R^d \mid \norm{x}_2 \leq r}$ be the centered ball of radius $r \geq 1$ and assume that we have $\XX = B_r(0) = \YY$ with the scaled and squared Euclidean norm $\c(x, y) = \frac{1}{4} \norm{x - y}_2^2$. As $\norminf{\c} = r^2 \geq 1$, we rescale via \autoref{rem:eot_rescale_mean_abs_dev} and \autoref{rem:rescale_norm2_cost}. This way, we obtain pushforwards $\mu^{r^2}, \nu^{r^2}, \hmu_n^{r^2}, \hnu^{r^2}_n$ that are all supported on $B_1(0)$. Hence, for $d < 4$ we can apply \autoref{thm:lca_sc} and for $d > 4$ \autoref{thm:lca_hoelder} with $\alpha \defeq \ceil{d / 2} + 1$ to obtain the $n^{-1/2}$-rate. Putting everything together, it follows for $n \in \N$,
	\begin{align*}
		\EV{\abs{\EOT(\hmu_n, \hnu_n) - \EOT(\mu, \nu)}}
		&= r^2 \EV{\abs{\OT_{\c/r^2, \eps/r^2}(\hmu_n, \hnu_n) - \OT_{\c/r^2,\eps/r^2}(\mu, \nu)}} \\
		&= r^2 \EV{\abs{\OT_{\c, \eps/r^2}(\hmu_n^{r^2}, \hnu_n^{r^2}) - \OT_{\c,\eps/r^2}(\mu^{r^2}, \nu^{r^2})}} \\
		&\lesssim n^{-1/2} \begin{cases}
			r^2 & d < 4\,, \\
			r^{d \frac{\ceil{d / 2}}{\ceil{d / 2} + 1} + 2} (\eps \land r^2)^{-\frac{d}{2} \frac{\ceil{d / 2}}{\ceil{d / 2} + 1}} & d \geq 4\,.
		\end{cases}
	\end{align*}
	This is in line with bounds obtained by \citet[Example~3]{Stromme2023} and \citet[Lemma~5]{Chizat2020}. Notably, this bound can also be obtained by combining \autoref{thm:lca_hoelder} with the rescaling approach as described in \autoref{rem:rescaling_effect}, which treats more general cost functions.
\end{example}

\begin{remark}[Comparison of bounds] \label{rem:comp_bounds}
	Let \autoref{ass:eot_cost} and \autoref{ass:hoelder} hold, and let $0 < \eps \leq 1$ be arbitrary. We consider the following three cases:
	\begin{enumerate}
		\item For $\alpha = 1$, we are also in the setting of \autoref{ass:lip} w.r.t.\ $\norm{}_2$. Notably, \autoref{thm:lca_lip} and \autoref{thm:lca_hoelder} yield the same rates for the statistical error in $n$ and~$\eps$.
		\item Let $\alpha = 2$. Then, \autoref{ass:semi_con} holds. \autoref{thm:lca_sc} and \autoref{thm:lca_hoelder} again yield the same rates in $n$. However, under semi-concavity we have no dependence on $\eps$ whereas the H\"older condition has the factor $\eps^{-s/4}$.
		\item \label{enum:comp_alpha_geq_3} If $\alpha \geq 3$, as before we are in the setting of \autoref{ass:semi_con}. In this case, the H\"older condition yields better rates in $n$ whereas under semi-concavity we have no dependence on $\eps$. Hence, loosely speaking, for a fixed or slowly decreasing $\eps$ the statistical error obtained by \autoref{thm:lca_hoelder} is better. More specially, choosing $\eps = n^{-\gamma}$ for some $\gamma > 0$, \autoref{cor:rates_eps_tradeoff} yields that the bound obtained under semi-concavity yields a smaller error if and only~if
		\begin{equation*}
			\gamma > \left[ \frac{1}{s / \alpha \lor 2} - \frac{1}{s/2 \lor 2} \right] \frac{\alpha}{\alpha - 1} \frac{2}{s}\,.
		\end{equation*}
		In particular, for $s > 4$ and $s / \alpha < 2$ the above condition reduces to
		\begin{equation*}
			\gamma > \left[ \frac{1}{2} - \frac{2}{s} \right] \frac{\alpha}{\alpha - 1} \frac{2}{s} > \frac{s-4}{s^2}\,.
		\end{equation*}
		Notably, for increasing dimension $s$ it follows that the regime $(\frac{s-4}{s^2}, \infty)$, where the bounds induced by semi-concavity yield a smaller bound, gets larger.
	\end{enumerate}
\end{remark}

\begin{example}
	Let $c$ be the squared Euclidean norm and choose $\eps = n^{-\gamma}$ for some $\gamma > 0$.

	For $d \leq 4$, \autoref{thm:lca_sc} yields the parametric rate $n^{-1/2}$ (up to $\log(n + 1)$-term for $d=4$) without any $\eps$-dependency. This rate is faster than (or for $d=1$ equal to) the one obtained from \autoref{thm:lca_hoelder}, independently of $\gamma$.

	For $d > 4$, we have by \enumref{rem:comp_bounds}{enum:comp_alpha_geq_3} that the rate obtained under semi-concavity is faster to the Hölder bound with $\alpha \in \N$ if and only if
	\begin{equation*}
		\gamma > \left[ \frac{1}{d / \alpha \lor 2} - \frac{2}{d} \right] \frac{\alpha}{\alpha - 1} \frac{2}{d} \,.
	\end{equation*}
	Choosing $\alpha \defeq \ceil{d / 2} + 1$ such that $d/\alpha < 2$, the above inequality reduces to
	\begin{equation*}
		\gamma > \left[ \frac{1}{2} - \frac{2}{d} \right] \cdot \frac{2\ceil{d / 2} + 2}{\ceil{d / 2} d} > \frac{d - 4}{d^2}\,.
	\end{equation*}
\end{example}

\subsection{Squared Euclidean Norm with Sub-Gaussian Measures}\label{subsec:sub_Gaussian}

All the previously given statistical error bounds are derived under \autoref{ass:eot_cost} which requires the cost function to be bounded. In this subsection, we show a version of the LCA principle for a (partially) unbounded setting. Namely, we build on \citet{Mena2019} and consider the squared Euclidean norm as cost and require the measures to be sub-Gaussian \citep{vershynin2018high}. A probability measure $\mu \in \sP(\R^d)$ is called $\sigma^2$-sub-Gaussian for $\sigma > 0$ if
\begin{equation*}
	\int_{\R^d} \exp\left(\frac{\norm{x}_2^2}{2 d \sigma^2} \right) \dP{\mu}{x} \leq 2 \,.
\end{equation*}
We denote with $\SG_d(\sigma^2)$ the set of all $\sigma^2$-sub-Gaussian probability measures on $\R^d$. \citet{Mena2019} derive statistical error bounds for the empirical EOT cost under sub-Gaussian measures by controlling suprema over empirical processes using $L^2$-metric entropy bounds (and not the uniform norm as we do). However, this approach does not seem to directly produce a result that shows the LCA principle as it is not compatible with \autoref{lemma:entt_cov_num} which requires the use of the uniform norm. To circumvent this technical limitation, we impose the condition that one of the measures has bounded support. Note that while the general proof strategy stays the same, the unboundedness requires different arguments than the ones used in the proof of \autoref{thm:eot_lca}. The proofs can be found in \autoref{appendix:proofs}.

\begin{assumption}{SG} \label{ass:sg}
	It holds that $\XX = \bigcup_{i=1}^I g_i(\sU_i) \subseteq \R^d$ for $I \in \N$ bounded, convex subsets $\sU_i \subseteq \R^s$ with nonempty interior and maps $g_i : \sU_i \to \XX$ that are $\alpha$-times continuously differentiable with bounded partial derivatives where $\alpha \in \N$ with $s / \alpha < 2$. Furthermore, let $\YY = \R^d$, $\c(x, y) = \frac{1}{2}\norm{x - y}_2^2$ be the squared Euclidean norm, $\sigma^2 \geq 1$ and $\sup_{x \in \XX} \norm{x}_2 \leq r$ with $r \geq 1$.
\end{assumption}

Note that $\mu \in \sP(\XX)$ is always sub-Gaussian because of its bounded support. To apply results from \citet{Mena2019} directly, we assume that $\mu \in \sP(\XX) \cap \SG_d(\sigma^2)$ and $\nu \in \SG_d(\sigma^2)$, i.e., that $\sigma^2$ is large enough to include the sub-Gaussianity of both measures. Further, because of \autoref{rem:rescale_norm2_eps} we can fix $\eps = 1$ without loss of generality.

We proceed similar to the approach in \autoref{sec:eot}. First, we formulate a version of duality (compare with \autoref{thm:eot_duality}).

\begin{proposition}[Duality, {\citealt[Proposition~6]{Mena2019}}] \label{lemma:eot_duality_sg}
	Let \autoref{ass:sg} hold, fix $\eps = 1$ and define the function class
	\begin{equation*}
		\sF_\sigma \defeq \mset*{
			\begin{aligned}
				&\phi : \XX \to \R \text{ such that } \exists \xi \in \SG_d(\sigma^2), \psi : \YY \to \R \\
				& \text{ with } \phi = \entt{\psi}{\xi}\,, \norminf{\phi} \leq 6 d^2 r^2 \sigma^4 \\
				& \text{ and } \psi(y) - \frac{1}{2} \norm{y}_2^2 \leq d\sigma^2 + \sqrt{2 d} \sigma \norm{y}_2\; \forall y \in \YY
		\end{aligned}}\,.
	\end{equation*}
	Then, for any $\mu \in \sP(\XX) \cap \SG_d(\sigma^2)$ and $\nu \in \SG_d(\sigma^2)$ it holds that
	\begin{equation*}
		\EOT(\mu, \nu) = \max_{\phi \in \sF_\sigma} \int_{\XX} \phi \de{\mu} + \int_{\YY} \entt{\phi}{\mu} \de{\nu}\,.
	\end{equation*}
\end{proposition}

As the squared Euclidean norm is smooth, it follows in conjunction with the strong concentration assumption via the Leibniz integral rule that elements of $\sF_\sigma \circ g_i$ are $\alpha$-times differentiable. Note that the dominated convergence theorem is still applicable because of sub-Gaussianity. To employ the uniform metric entropy bounds for H\"older classes, we need uniform bounds for the partial derivatives as in \autoref{lemma:infnorm_bound}.

\begin{lemma}[Bounds for derivatives] \label{lemma:infnorm_bound_sg}
	Let \autoref{ass:sg} hold and fix $\eps = 1$. Define
	\begin{equation*}
		G_{i,m} \defeq \sup_{\abs{j} = m} \norminf{\De^j g_i}
	\end{equation*}
	and
	\begin{equation*}
		G^{(i, 0)} \defeq 1\,, \qquad G^{(i, \kappa+1)} \defeq \max(G_{i,\kappa+1}, G^{(i,\kappa)}, \max_{m=1,\ldots,\kappa} G_{i,m} G^{(i,\kappa)})\,.
	\end{equation*}
	Then, it holds for all $\phi \in \sF_\sigma$ and indices tuples $k \in \idn{s}^\kappa$, $1\leq \kappa \leq \alpha$, that
	\begin{equation*}
		\norminf{\De^k[\phi \circ g_i - \frac{1}{2} \norm{}_2^2 \circ g_i]} \lesssim G^{(i,\kappa)} \sigma^{3\kappa}\,,
	\end{equation*}
	where the implicit constant only depends on $\alpha$ and $d$.
\end{lemma}

Having uniform bounds on the partial derivatives, we arrive by combining \autoref{lemma:cov_num_union}, \autoref{lemma:cov_num_comp}, and \autoref{lemma:cov_num_hoelder} at the following novel uniform metric entropy estimate.

\begin{proposition} \label{prop:unif_metric_ent_sg}
	Let \autoref{ass:sg} hold and fix $\eps = 1$. Then, it follows for $\delta > 0$,
	\begin{equation*}
		\log\covnum{\delta}{\sF_\sigma}{\norminf{}} \lesssim \left( \sum_{i=1}^{I} [G^{(i, \alpha)}]^{s/\alpha} \right) \sigma^{3s} \delta^{-s/\alpha}\,,
	\end{equation*}
	where the implicit constant only depends on $\alpha$, $d$, $s$ and $\sU_1, \ldots, \sU_I$.
\end{proposition}

This result enables us to derive an LCA result for the setting where one measure is compactly supported while the other is sub-Gaussian.

\begin{theorem}[Partially unbounded LCA] \label{thm:lca_sg}
	Let \autoref{ass:sg} hold and fix $\eps = 1$. Then, for all probability measures $\mu \in \sP(\XX) \cap \SG_d(\sigma^2)$, $\nu \in \SG_d(\sigma^2)$ and any of the empirical estimators $\hEOT$ from \eqref{eq:eot_emp_est} it holds that
	\begin{equation*}
		\EV{\abs{\hEOT - \EOT(\mu, \nu)}} \lesssim  \left(\sum_{i=1}^I 1 + [G^{(i,\alpha)}]^{s / \alpha} \right) r^2 \sigma^{4\alpha \lor (4 + 3s/2) } n^{-1/2} \,,
	\end{equation*}
	where the implicit constant only depends on $\alpha$, $d$, $s$ and $\sU_1, \ldots, \sU_I$. In particular, for $s \geq 8$ by \autoref{ass:sg} we must have $\alpha > 4$ and the exponent of $\sigma$ equals $4\alpha$.
\end{theorem}

Using \autoref{rem:rescale_norm2_eps}, we obtain the following result for more general $\eps \neq 1$.

\begin{corollary} \label{cor:lca_sg}
	Let \autoref{ass:sg} hold and let $\eps > 0$. Then, for all probability measures $\mu \in {\sP(\XX) \cap \SG_d(\sigma^2)}$, $\nu \in \SG_d(\sigma^2)$ and any of the empirical estimators $\hEOT$ from \eqref{eq:eot_emp_est} it holds that
	\begin{equation*}
		\EV{\abs{\hEOT - \EOT(\mu, \nu)}} \lesssim \left(\sum_{i=1}^I 1 + [G^{(i,\alpha)}]^{s / \alpha} \right) r^2 \sigma^{4\alpha \lor (4 + 3s/2)} (\epsilon \land 1)^{-[2\alpha \lor (2 + 3s/4)] - s/2}n^{-1/2}\,,
	\end{equation*}
	where the implicit constant only depends on $\alpha$, $d$, $s$ and $\sU_1, \ldots, \sU_I$. In particular, for $s \geq 8$ by \autoref{ass:sg} we must have $\alpha > 4$ and the exponent of $\eps \land 1$ equals $-2\alpha-s/2$.
\end{corollary}

Note that the implicit constant in the above bound still depends on the dimension $d$ of the ground space $\R^d$. Nevertheless, only the dimension $s$ of the (lower-dimensional) $\sU_i$ influences the dependency of the bound on $\eps$, $\sigma^2$ and the $g_i$. Hence, it also shows the validity of the LCA principle.

Further, observe that if the second measure $\nu$ is compactly supported, the setting of \autoref{thm:lca_hoelder} is also satisfied (up to scaling of the cost function). However, as a trade-off, by staying under \autoref{ass:sg} and not making use of the compactness of the support, we get a worse dependency in $\eps$.

\section{Computational Complexity} \label{sec:comput_comp}

As seen in the previous section, there are several settings where the empirical EOT cost benefits from the LCA principle. However, in practice the plug-in estimator $\EOT(\hmu_n, \hnu_n)$ is in turn only \emph{approximated} using the Sinkhorn algorithm. For this reason, we now investigate if we can construct a computable estimator that reflects the LCA principle. For the considerations to follow, we rely on the analysis of the Sinkhorn algorithm by \citet{Marino2020} as well as \citet{dvurechensky2018computational}.

First, we briefly recall the Sinkhorn algorithm. Let $\mu \in \sP(\XX)$, $\nu \in \sP(\YY)$ be two probability measures. Given a bounded start potential $\psi_0 : \YY \to \R$, the Sinkhorn algorithm can be viewed as approximating dual optimizers by alternatingly applying the entropic $(\c, \eps)$-transform to it \citep{Marino2020}. More precisely, for $m \in \N$ we recursively define
\begin{equation*}
	\phi_m \defeq \begin{cases}
		\entt{\psi_{m-1}}{\nu} & \text{$m$ odd,} \\
		\phi_{m-1} & \text{$m$ even,}
	\end{cases} \qquad \psi_m \defeq \begin{cases}
	\psi_{m-1} & \text{$m$ odd,} \\
	\entt{\phi_{m-1}}{\mu} & \text{$m$ even.}
	\end{cases}
\end{equation*}
Then, for $m \to \infty$ the pair $(\phi_m, \psi_m)$ converges to a pair of dual optimizers of \eqref{eq:eot_gen_duality} for $\mu$ and~$\nu$ \citep{Marino2020}. To employ these potentials for a EOT cost estimator we define the transport plan
\begin{equation*}
	\de{\pi^{(m)}} \defeq \exp_\eps(\phi_m \oplus \psi_m - \c) \de{[\mu \otimes \nu]}\,,
\end{equation*}
whose marginal measures we denote by $\mu^{(m)}$ and $\nu^{(m)}$. Then, by definition of the $(\c,\eps)$-transform it follows that $\mu^{(m)} = \mu$ for odd $m$ and $\nu^{(m)} = \nu$ for even~$m$, asserting that $ \mu^{(m)}$ and $\nu^{(m)}$ are always probability measures. In fact, this construction also yields that $(\phi_{m+1}, \psi_{m+1})$ are for any $m\in \NN$ optimal potentials for the measures $(\mu^{(m)},\nu^{(m)})$ (see \autoref{lemma:sinkhornPot}), which implies the representation
\begin{equation*}
	\EOT(\mu^{(m)}, \nu^{(m)})= \int \phi_{m+1} \de{}{\mu^{(m)}} + \int \psi_{m+1} \de{}{\nu^{(m)}}\,.
\end{equation*}

Further, by convergence of the potentials $(\phi_m, \psi_m)$ to dual optimizers for $\mu$ and~$\nu$ as $m \to \infty$, the measures $(\mu^{(m)}, \nu^{(m)})$ also tend towards $(\mu,\nu)$. A suitable termination criterion for the Sinkhorn algorithm is therefore to stop once the difference between $(\mu^{(m)}, \nu^{(m)})$ and $(\mu,\nu)$ drops below a certain prespecified threshold. This difference can be quantified for odd $m$ by the TV-norm $\norm{\mu^{(m)} - \mu}_1$ and for even~$m$ by $\norm{\nu^{(m)} - \nu}_1$.

\begin{theorem}[Computational complexity] \label{thm:lca_sinkhorn}
Let \autoref{ass:eot_cost} hold and consider $\mu \in \sP(\XX)$, $\nu \in \sP(\YY)$. Denote with $\hEOT^{(m)}= \int \phi_{m+1} \de{}{\hmu_n^{(m)}} + \int \psi_{m+1} \de{}{\hnu_n^{(m)}}$ the empirical EOT cost estimator based on the Sinkhorn algorithm after $m+1$ iterations with empirical measures $\hmu_n$ and $\hnu_n$ as input. Furthermore, suppose for some $K_{\eps, n}>0$ that
	\begin{equation*}
		\EV{\abs{\EOT(\hmu_n, \hnu_n) - \EOT(\mu, \nu)}} \leq K_{\eps, n}\,.
	\end{equation*}
	Then, for some deterministic $L_{\eps, n} > 0$ we have  after $m = \floor{2 + 20 L_{\eps, n}^{-1} (3 \log n + \eps^{-1} )}$ iterations,
	\begin{equation*}
		\EV{\abs{\hEOT^{(m)} - \EOT(\mu, \nu)}} \leq L_{\eps, n} + K_{\eps, n}\,.
	\end{equation*}
\end{theorem}
\begin{proof}
Invoking \autoref{lemma:sinkhornPot} from Appendix \ref{appendix:A3} we know for any $m\in \N$ that the potentials $(\phi_{m+1},\psi_{m+1})$ are optimal for $(\hmu_n^{(m)}, \hnu_n^{(m)})$, asserting that $\hEOT^{(m)} = \EOT(\hmu_n^{(m)}, \hnu_n^{(m)})$. Further, recall that either $\hmu_n^{(m)} = \hmu_n$ or $\hnu_n^{(m)} = \hnu_n$, depending on the parity of $m$. We therefore consider w.l.o.g.\ the first case, i.e., $\hmu_n^{(m)} = \hmu_n$, the argument for the latter case is analogous. Using \autoref{lemma:eot_stab_bound}, we find that
	\begin{align*}
		\EV{\abs{\hEOT^{(m)} - \EOT(\hmu_n, \hnu_n)}}
		&= \EV{\abs{\EOT(\hmu_n, \hnu_n^{(m)}) - \EOT(\hmu_n, \hnu_n)}} \\
		&\leq 2 \EV*{\sup_{\phi \in \FC{}} \abs*{ \int_{\YY} \entt{\phi}{\hmu_n} \de{[\hnu_n^{(m)} - \hnu_n]}}} \\
		&\leq 5 \EV{\norm{\hnu_n^{(m)} - \hnu_n}_1}\,,
	\end{align*}
	where the last step uses that $\entt{\FC{}}{\hmu_n}$ is bounded in uniform norm by $5 / 2$. Choosing $\delta \defeq L_{\eps, n} / 5$, according to Theorem~1 from \citet{dvurechensky2018computational} we can achieve
	\begin{equation*}
		\norm{\hnu_n^{(m)} - \hnu_n}_1 \leq \delta
	\end{equation*}
	after $m = \floor{2 + 4 \delta^{-1} (\log n - \log \ell)}$ iterations, where we set
	\begin{equation*}
		\ell \defeq \min_{i,j=1,\ldots,n} \frac{1}{n^2} \exp_\eps(-\c(X_i, Y_j)) \,.
	\end{equation*}
	As $\ell \geq n^{-2} \exp_\eps(-\norminf{\c})$, it holds that $-\log \ell \leq 2 \log n + \eps^{-1}$. In particular, this implies that
	\begin{equation*}
		5 \EV{\norm{\hnu_n^{(m)} - \hnu_n}_1} \leq L_{\eps, n}
	\end{equation*}
	after $m = \floor{2 + 20 L_{\eps,n}^{-1} (3 \log n + \eps^{-1})}$ iterations. Finally, by assumption and the triangle inequality, we conclude that
	\begin{align*}
		\EV{\abs{\hEOT^{(m)} - \EOT(\mu, \nu)}} &\leq \EV{\abs{\hEOT^{(m)} - \EOT(\hmu_n, \hnu_n)}}
		+ \EV{\abs{\EOT(\hmu_n, \hnu_n) - \EOT(\mu, \nu)}} \\
		&\leq L_{\eps, n} + K_{\eps, n}\,. \qedhere
	\end{align*}
\end{proof}

\begin{remark}
	In the above proof, the inequality
	\begin{equation*}
		\abs{\hEOT^{(m)} - \EOT(\hmu_n, \hnu_n)} \leq L_{\eps, n}
	\end{equation*}
	is even met \emph{deterministically} after $m = \floor{2 + 20 L_{\eps,n}^{-1} (3 \log n + \eps^{-1})}$ iterations. We emphasize that this does not depend on the input measures $\hmu_n$ and $\hnu_n$.
\end{remark}

\begin{remark}
	For the dual perspective of the LCA principle (\autoref{thm:eot_lca}), we obtain with \autoref{thm:lca_sinkhorn} that the Sinkhorn estimator $\hEOT^{(m)}$ after $m = \Omega(K_\eps^{-1/2}n^{1/(k \lor 2)}[\log n + \eps^{-1}])$ steps satisfies the statistical error rate $\bigO(K_\eps^{1/2} n^{-1/(k \lor 2)})$ (up to log-terms). Since every Sinkhorn iteration encompasses $\bigO(n^2)$ arithmetic operations we conclude that the proposed estimator has a computational complexity of order $\bigO(K_\eps^{-1/2} n^{2 + 1 / (k \lor 2)} [\log n + \eps^{-1}])$. Herein, we observe a trade-off between statistical accuracy and computational effort. Lastly, let us point out that \autoref{thm:lca_sinkhorn} is also applicable to the upper bound \eqref{eq:BoundStromme} by \citet{Stromme2023}.
\end{remark}

\begin{remark}
	In \autoref{thm:lca_sinkhorn}, computation accuracy $L_{\eps, n}$ can always be chosen to be of order $n^{-1/2}$. As the bound on the mean absolute deviation $K_{\eps,n}$ is typically also at least of this order, we see that an estimator with similar statistical efficiency can be calculated at computational cost of order $\bigO(n^{2.5})$ (up to log-factors).
\end{remark}

\begin{example}
	Consider the setting of \autoref{ex:sq_euclidean_norm}, i.e., $\c$ is the squared Euclidean norm. Then, for $d < 4$ we can obtain the statistical error rate given in the aforementioned example with computational complexity $\bigO(\eps^{-1} n^{2.5})$ and else $\bigO(\eps^{d/2-1}n^{2.5})$, if $\eps$ is small enough. Hence, in high dimensions $d$ and for small regularization parameter $\eps$ we notice a reduced computational complexity due to the worse statistical accuracy.
\end{example}

Overall, we conclude that the Sinkhorn algorithm for empirical measures outputs an estimator for the empirical EOT cost that also adheres to the LCA principle. Provided that the statistical error $K_{\eps, n}$ is independent of the ambient dimension of the ground spaces, e.g., in settings where one measure is of low intrinsic dimension, \autoref{thm:lca_sinkhorn} yields a well-computable estimator with good statistical accuracy.

\section{Implications to the Entropic Gromov-Wasserstein Distance} \label{sec:gromov_wasserstein}

The Gromov-Wasserstein distance provides an OT based tool to quantify the dissimilarity between two metric measure spaces (a metric space equipped with a probability measure) up to isometry\footnote{Two metric measure spaces $(\XC, \rho_\XC, \mu)$ and $(\YC, \rho_\YC, \nu)$ are said to be isometric, if there exists a bijective isometry $\Gamma\colon (\XC, \rho_\XC) \mapsto (\YC, \rho_\YC)$ such that $\Gamma_{\#} \mu = \nu$. } \citep{Memoli2011,Sturm2012}. Hence, by modeling heterogeneous data as metric measure spaces the Gromov-Wasserstein distance serves as a conceptually appealing discrimination measure for registration invariant comparison, e.g., in the context of protein matching \citep{Weitkamp2022}. Unfortunately, practical usage of the Gromov-Wasserstein distance faces severe obstacles due to significant computational challenges. Indeed, for finitely supported measures computation of the Gromov-Wasserstein distance reduces to a non-convex quadratic assignment program, which are known to be NP-complete in general \citep{Commander2005}.

Motivated by the computational benefits of entropy regularization for the OT cost, the entropic Gromov-Wasserstein distance was introduced by \citet{peyre2016gromov,solomon2016entropic}. To formalize it w.r.t.\ the Euclidean norm, let $\XC \subseteq \R^s$ and $\YC\subseteq \R^d$ be Polish subsets and consider $\mu \in \sP(\XC)$ and $\nu \in \sP(\YC)$ which admit finite fourth moments. Then, their entropic $(2, 2)$-Gromov-Wasserstein distance for regularization parameter $\epsilon>0$ is defined as

\begin{align*}
	\GW_\eps(\mu, \nu) \defeq \inf_{\pi \in \Pi(\mu, \nu)} & \biggl[ \int_{\XX \times \YY} \int_{\XX \times \YY} \abs{ \norm{x-x'}_2^2 - \norm{y - y'}_2^2 }^2 \ddP{\pi}{x}{y} \ddP{\pi}{x'}{y'} \\
	&\qquad\qquad + \eps \KL(\pi \mid \mu \otimes \nu) \biggr]\,.
\end{align*}
The corresponding unregularized Gromov-Wasserstein distance is defined by omitting the entropy penalization term and denoted by $\GW_0$.

Remarkably, despite $\GW_\eps(\mu, \nu)$ also encompassing a non-convex optimization problem, the recent work by \citet{rioux2023entropic} show that it can be computed up to arbitrary precision using a gradient method that employs the Sinkhorn algorithm. In particular, for input measures with $n$ support points it admits a computational complexity of order $\bigO(n^2)$ (up to polylogarithmic terms). This makes it a viable tool for statistical data analysis. To analyze its sample complexity consider the empirical plug-in estimators
\begin{equation} \label{eq:egw_emp_est}
	\hEGW \in \mset{ \GW_\eps(\mu, \hnu_n), \GW_\eps(\hmu_n, \nu), \GW_\eps(\hmu_n, \hnu_n)}\,.
\end{equation}
In this context \citet[Theorem~2]{Zhang2022} show for $4$-sub-Weibull probability measures $\mu \in \sP(\XX)$ and $\nu \in \sP(\YY)$ with concentration parameter $\sigma > 0$ that
\begin{equation}\label{eq:boundZhang}
	\EV{\abs{ \hEGW - \GW_\eps(\mu, \nu)}} \lesssim (1 + \sigma^4) n^{-1/2} + \eps \left( 1 + \left[ \frac{\sigma}{\sqrt{\eps}} \right]^{{9 \ceil*{\frac{s \lor d}{2}} + 11}} \right) n^{-1/2}\,.
\end{equation}
In this upper bound, we see that the parametric rate is attained but the polynomial scaling in $\eps^{-1}$ depends on the maximum dimension of $s$ and $d$. The analysis of the authors hinges on the following representation of the entropic Gromov-Wasserstein distance for centered probability measures $\mu$, $\nu$,
\begin{equation*}
	\GW_{\eps}(\mu, \nu) = \GW_{1,1}(\mu, \nu) + \GW_{2,\eps}(\mu, \nu)\,,
\end{equation*}
where the two terms on the right-hand side are defined as
\begin{align*}
	\GW_{1,1}(\mu, \nu) \defeq &
	\int_{\XX \times \XX} \norm{x-x'}_2^{4} \dP{\mu}{x} \dP{\mu}{x'} + \int_{\YY \times \YY} \norm{y-y'}_2^{4} \dP{\nu}{y} \dP{\nu}{y'} \\
	&\qquad\qquad - 4 \int_{\XX \times \YY} \norm{x}_2^2 \norm{y}_2^2 \dP{\mu}{x} \dP{\nu}{y}\,,\\
	\GW_{2,\eps}(\mu, \nu) \defeq &
	\inf_{\pi \in \Pi(\mu, \nu)} \biggl[ -4 \int_{\XX \times \YY} \norm{x}_2^2 \norm{y}_2^2 \ddP{\pi}{x}{y} - 8 \sum_{i,j=1}^{s,d} \left( \int_{\XX \times \YY} x_i y_j \ddP{\pi}{x}{y}\right)^2 \\
	&\qquad\qquad + \eps \KL(\pi \mid \mu \otimes \nu) \biggr]\,.
\end{align*}
Note that this decomposition does not directly hold for the plug-in estimators \eqref{eq:egw_emp_est} as the empirical distributions are in general not centered. To this end, \citet{Zhang2022} debias the empirical measures which contributes an additional $\sigma^2 n^{-1/2}$-term (see their Lemma 2). Then, the empirical plug-in estimator for the first term $\GW_{1,1}$ is a Monte-Carlo estimator and can be analyzed via $V$-statistics, contributing the first term on the right-hand side of \eqref{eq:boundZhang}. For the empirical estimator to the second term they link $\GW_{2,\eps}$ to the EOT cost w.r.t.\ a class of cost functions and obtain statistical error bounds by controlling empirical processes uniformly over the class of cost functions. Applying an adjusted version of \autoref{lemma:entt_cov_num}, we can confirm an LCA principle for the entropic $(2,2)$-Gromov-Wasserstein distance. As in \autoref{sec:eot_lca_dual}, we require the cost(s) to be bounded and therefore impose the following compactness assumption.

\begin{assumption}{GW} \label{ass:egw}
	Let $\XX \subseteq \R^{s}$ and $\YY \subseteq \R^{d}$ be compact with $s \leq d$ as well as $\diam(\XX) \leq r$ and $\diam(\YY) \leq r$ for some $r \geq 1$.
\end{assumption}

First, we give the aforementioned link between $\GW_{2,\eps}$ and a collection of EOT costs.

\begin{theorem}[Duality, {\citealt[Theorem~1]{Zhang2022}}] \label{thm:egw_duality}
	Let \autoref{ass:egw} hold. Denote the set of matrices $\sD \defeq \left[- r^2/2, r^2/2 \right]^{s \times d}$ and for $A \in \sD$ define the cost function
	\begin{equation*}
		\c_A : \XX \times \YY \,, \qquad (x, y) \mapsto -4 \norm{x}_2^2 \norm{y}_2^2 - 32 x^\transp A y\,.
	\end{equation*}
	Then, it holds for all $\mu \in \sP(\XX)$ and $\nu \in \sP(\YY)$ that
	\begin{equation*}
		\GW_{2, \eps}(\mu, \nu) = \min_{A \in \sD} 32 \norm{A}_2^2 + \OT_{\c_A, \eps}(\mu, \nu) \,.
	\end{equation*}
\end{theorem}

Upon defining for arbitrary $\tmu\in \PC(\XC)$ the two function classes
\begin{equation*}
	\FCd \defeq \bigcup_{A \in \sD} \sF_{\c_A, \eps}\,, \qquad \enttd{\FCd}{\tmu} \defeq \bigcup_{A \in \sD} \entta{\FCd}{A}{\tmu}\,,
\end{equation*}
an application of \autoref{thm:egw_duality} in combination with \autoref{lemma:eot_stab_bound} yields the following stability bound for $\GW_{2,\eps}$. The proof of this result and subsequent assertion of this section are detailed in Appendix \ref{appendix:A4}

\begin{lemma}[Stability bound] \label{lemma:egw_stab_bound}
	Let \autoref{ass:egw} hold. Then, it holds for any pairs of probability measures $\mu, \tmu \in \sP(\XX)$ and $\nu, \tnu \in \sP(\YY)$ that
	\begin{align*}
		\abs{ \GW_{2,\eps}(\tmu, \tnu) - \GW_{2, \eps}(\mu, \nu)}
		&\leq 2 \sup_{A \in \sD} \abs{\OT_{\c_A, \eps}(\tmu, \tnu) - \OT_{\c_A, \eps}(\mu, \nu)} \\
		&\leq 4 \sup_{\phi \in \FCd} \abs*{ \int_{\XX} \phi \de{[\tmu - \mu]}} + 4 \sup_{\psi \in \enttd{\FCd}{\tmu}} \abs*{\int_{\YY} \psi \de{[\tnu - \nu]}}\,.
	\end{align*}
\end{lemma}

Hence, bounding $\EV{\abs{\GW_\eps(\hmu_n, \hnu_n) - \GW_\eps(\mu, \nu)}}$ reduces to controlling two empirical processes over the function classes $\FCd$ and $\enttd{\FCd}{\hmu_n}$ (i.e., by selecting $\tilde \mu \defeq\hmu_n$). As the class of cost functions $\mset{\c_A}_{A \in \sD}$ is Lipschitz continuous in $A$ w.r.t. to the uniform norm $\norminf{}$, we observe that the union over entropic cost transforms do not increase the uniform metric entropy of a function class by much.

\begin{lemma} \label{lemma:enttd_cov_num}
	Let \autoref{ass:egw} hold and let $\tmu \in \sP(\XX)$. Consider a function class $\sF\subseteq \Lexp(\tmu)$ on $\XX$. Then, it holds for the union over $(\c_A, \epsilon)$-transformed function classes $\enttd{\sF}{\tmu}\defeq \bigcup_{A \in \sD} \entta{\sF}{A}{\tmu}$ and any $\delta > 0$ that
	\begin{equation}\label{eq:entropyBoundGW}
		\covnum{\delta}{\enttd{\sF}{\tmu}}{\norminf{}} \leq \covnum{\delta/4}{\sF}{\norminf{}} \covnum{\delta / [64 r^2]}{\sD}{\norminf{}}\,.
	\end{equation}
\end{lemma}

A modified version of \autoref{thm:eot_lca}, adjusted to \autoref{lemma:enttd_cov_num}, yields the following LCA result.

\begin{theorem}[Entropic Gromov-Wasserstein LCA] \label{thm:lca_egw}
	Let \autoref{ass:egw} hold. Then, for all probability measures $\mu \in \sP(\XX)$, $\nu \in \sP(\YY)$ and any of the empirical estimators $\hEGW$ from \eqref{eq:egw_emp_est} it holds that
	\begin{equation}\label{eq:GW_bound2}
		\EV{\abs{ \hEGW - \GW_\eps(\mu, \nu)}} \lesssim r^4 n^{-1/2} + r^{4(s \land d) + 4} (\eps \land r^4) ^{-(s \land d)/2} n^{-1/2}\,,
	\end{equation}
	where the implicit constant only depends on $s$, $d$ and $\XX$.
\end{theorem}

Note that the implicit constant in \eqref{eq:GW_bound2} still depends on $s \lor d$. However, the dependency in $\eps^{-1}$ is determined by $s \land d$ and thus obeys the LCA principle.

\begin{remark}
	 Following our arguments from \autoref{subsec:hoelder_cost}, we would like to point out that \autoref{ass:egw} could be refined to assuming that $\XX$ is a union $\bigcup_{i=1}^I g_i(\sU_i)$ for $I \in \N$ bounded, convex subsets $\sU_i \subseteq \R^s$ with nonempty interior and maps $g_i : \sU_i \to \XX$ that are $\alpha$-times continuously differentiable with bounded partial derivatives where $\alpha \in \N$.
	 Then, for all probability measures $\mu \in \sP(\XX)$, $\nu \in \sP(\YY)$ and any of the empirical estimators $\hEGW$ from \eqref{eq:egw_emp_est} it holds for sufficiently small $\eps$ that
	\begin{equation*}
		\EV{\abs{ \hEGW - \GW_\eps(\mu, \nu)}} \lesssim \eps^{-s/2}
		\begin{cases}
			n^{-1/2} & s/\alpha < 2 \,, \\
			n^{-1/2} \log (n+1) & s/\alpha = 2 \,, \\
			n^{-\alpha/s} & s/\alpha > 2\,.
		\end{cases}
	\end{equation*}
	An explicit dependency on the bounds for the partial derivatives can be obtained along the lines of \autoref{lemma:infnorm_bound} and is omitted here.
\end{remark}

\begin{remark}[Unregularized Gromov-Wasserstein LCA] \label{rem:unreg_gw}
\citet{Zhang2022} also derive a representation of the unregularized $(2,2)$-Gromov-Wasserstein distance as an infimum of unregularized OT costs over a suitable class of cost functions. Hence, invoking methods for statistical error bounds on the empirical unregularized OT cost by \citet[Section 3.3]{Hundrieser2022}, the validity of the LCA principle can also be confirmed for the unregularized $(2,2)$-Gromov-Wasserstein distance. More specifically, under \autoref{ass:egw} it follows,
	\begin{equation*}
		\EV{\abs{ \widehat{\GW}_{0,n} - \GW_0(\mu, \nu)}} \lesssim \begin{cases}
			n^{-1/2} & s \land d < 4\,, \\
			n^{-1/2} \log(n + 1) & s \land d = 4\,,\\
			n^{-2/(s \land d)} & s \land d > 4\,,
		\end{cases}
	\end{equation*}
	where the implicit constant only depends on $s$, $d$, $r$ and $\XX$. A proof is detailed in \autoref{appendix:proofs}.
\end{remark}

\section{Simulations} \label{sec:sims}

In the previous sections, we analyzed various settings where the LCA principle holds for empirical EOT. We now investigate whether the LCA principle can also be observed numerically. More specifically, for probability measures $\mu \in \sP(\XX)$ and $\nu \in \sP(\YY)$, various $\eps$ and one of the empirical estimators from \eqref{eq:eot_emp_est} we approximate the mean absolute deviation
\begin{equation*}
	\Delta_n \defeq \EV{\abs{\hEOT - \EOT(\mu, \nu)}}
\end{equation*}
for different sample sizes $n$ by Monte Carlo simulations with $1000$ repetitions. Due to a lack of explicit formulas, the population quantity $\EOT(\mu, \nu)$ is also approximated using Monte Carlo simulations with $1000$ repetitions and a large sample size specified below for each simulation setting. The empirical estimators $\hEOT$ are approximated via the Sinkhorn algorithm such that the error of the first marginal is less than $10^{-8}$ w.r.t.\ the norm $\norm{}_1$.\footnote{The code used for our simulations can be found under \url{https://gitlab.gwdg.de/michel.groppe/eot-lca-simulations}.}

We also consider the Sinkhorn divergence in our simulations. Assuming that $\XX = \YY$ and thus $\c : \XX\times\XX\to\R$, the Sinkhorn divergence between $\mu$ and $\nu$ is defined as
\begin{equation*}
	S_{\c, \eps}(\mu, \nu) \defeq \EOT(\mu, \nu) - \frac{1}{2} \EOT(\mu, \mu) - \frac{1}{2} \EOT(\nu, \nu)\,.
\end{equation*}
The last two terms have a debiasing effect such that $S_{\c, \eps}(\mu, \nu) = 0$ for $\mu = \nu$. Under certain assumptions on the cost function, the Sinkhorn divergence is even positive definite \citep{Feydy2019}.
Moreover, by its definition it is not clear if it also benefits from the LCA principle. Indeed, all available bounds for the statistical error of $\EOT(\hnu_n, \hnu_n)$ depend at least on the intrinsic dimension of $\nu$ and do not suggest that the LCA principle holds.

Overall, we examine the following simulation settings:
\begin{enumerate}
	\item Cube: For $d_1 \in \idn{10}$ and $d_2 = 5$ take
	\begin{equation*}
		\mu = \unifdist( [0, 1]^{d_1} \times \mset{0}^{d_1 \lor d_2 - d_1} )\,, \qquad \nu = \unifdist([0, 1]^{d_2} \times \mset{0}^{d_1 \lor d_2 - d_2})
	\end{equation*}
	and as the cost the by $d_1 \lor d_2$ normalized squared Euclidean norm $\norm{}_2^2$ or $1$-norm $\norm{}_1$. For $n \in \mset{100 k \mid k \in \idn{10}}$ we compute the two-sample estimator $\hEOT = \EOT(\hmu_n, \hnu_n)$. The true value $\EOT(\mu, \nu)$ is approximated using $n = 6000$ samples.
	\item Semi-discrete: For each $I \in \mset{5, 10, 50}$ and $d \in \mset{10, 100, 1000}$, we define
	\begin{equation*}
		\mu = \frac{1}{I} \sum_{i=1}^{I} \delta_{x^{(I, d)}_i}\,, \qquad \nu = \unifdist[0,1]^d\,,
	\end{equation*}
	where $x^{(I, d)}_{1}, \ldots, x^{(I, d)}_I$ are fixed and drawn i.i.d.\ from $\unifdist[0, 1]^d$, and as the cost the uniform norm $\norminf{}$. Now, for $n \in \mset{100, \ldots, 5000}$ we calculate the one-sample estimator $\EOT(\mu, \hnu_n)$ and use $n = 20000$ samples to approximate $\EOT(\mu, \nu)$.

	\item Sinkhorn divergence: We again consider the cube setting with cost $\norm{}_2^2$ but instead of the EOT cost $\OT_{\c, \eps}$ we employ the Sinkhorn divergence $S_{\c, \eps}$.
\end{enumerate}

\begin{figure}[t!]
	\centerline{\includegraphics{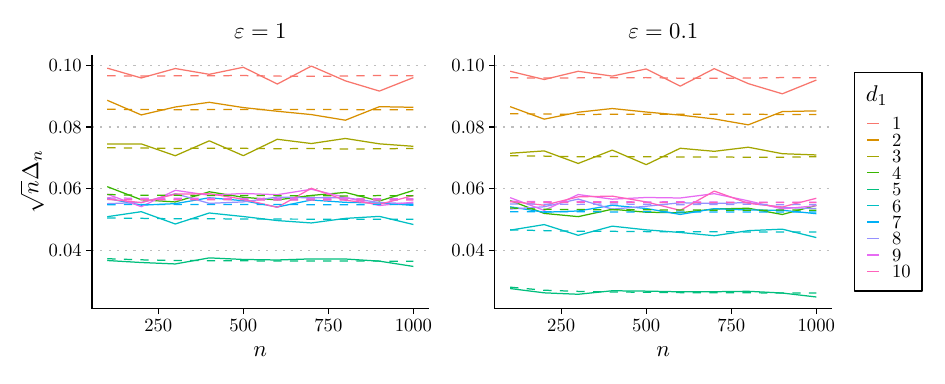}}
	\caption{Simulations of the mean absolute deviation $\Delta_n$ (solid) and the by $\sqrt{2/\pi}$ scaled asymptotic standard deviation of the fluctuations $\sqrt{n}[\EOT(\hmu_n, \hnu_n) - \EOT(\mu, \nu)]$ (dashed) in the cube setting with cost $\norm{}_2^2$.} \label{fig:sim_sq_eucl_norm}
\end{figure}

\begin{figure}[t!]
	\centerline{\includegraphics{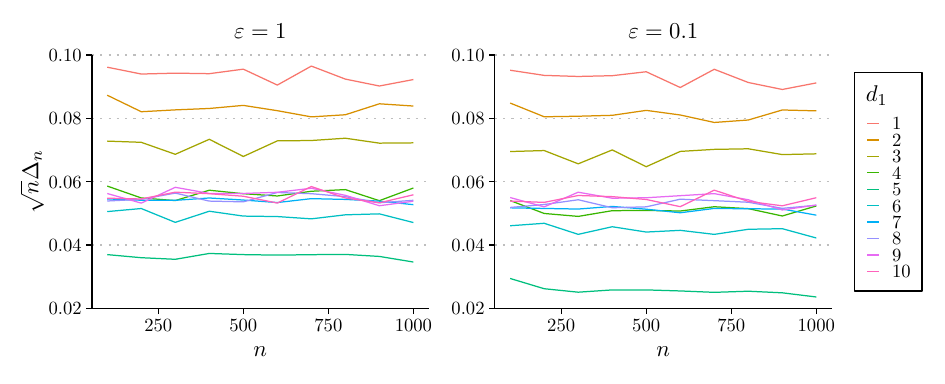}}
	\caption{Simulations of the mean absolute deviation $\Delta_n$ in the cube setting with cost $\norm{}_1$.} \label{fig:sim_L1_norm}
\end{figure}

\autoref{fig:sim_sq_eucl_norm} and \autoref{fig:sim_L1_norm} show the results for the cube setting with smooth cost $\norm{}_2^2$ and non-smooth $\norm{}_1$, respectively. As they are very similar we therefore focus on the former. In particular, we see for fixed $\eps>0$ that $\Delta_n$ roughly has a convergence rate of order $n^{-1/2}$, which is in line with \autoref{ex:sq_euclidean_norm} for $\norm{}_2^2$, and the combination of \citet{Rigollet2022} (recall \eqref{eq:EOT_Rigollet}) with our projective perspective on the LCA principle (\autoref{sec:eot_lca_primal}) for $\norm{}_1$. Furthermore, we observe that the underlying constant decreases from $d_1 = 1$ to $d_1 = 5$, and from $d_1 = 6$ to $d_1 = 10$ approximately stays the same. In the latter case, it seems as if the constant only depends on the smaller dimension $d_2 = 5$, thus corroborating the LCA principle.

At first glance the behavior of the constant for $d_1 = 1$ to $d_1 = 5$ appears surprising as it decreases with growing dimension. This behavior can be explained with \autoref{thm:eot_lca_primal}. Here (and also similarly with $\norm{}_1$ as cost), we have $\XX = [0,1]^{d_1}$, $\YY_1 = \XX$ and $\YY_2 = [0,1]^{d_2 - d_1}$ with $\c_1(x, y_1) = \norm{x-y_1}_2^2$ and $\c_2(y_2) = \norm{y_2}_2^2$. Note, that $\YY_2$ has the highest dimension for $d_1 = 1$ and is empty for $d_1 = 5 = d_2$. Hence, the dependence on the additional part $\c_2$ decreases from $d_1 = 1$ to $d_1 = 5$. This suggests that the constant for the statistical error of the integral term $\int_{\YY_2} \c_2 \de{\hnu_{n,2}}$ dominates the one for $\OT_{\c_1,\eps}(\hmu_n, \hnu_{n,1})$. Further, note that in this setting the fluctuations $\sqrt{n}[\EOT(\hmu_n, \hnu_n) - \EOT(\mu, \nu)]$ asymptotically admit a zero-mean normal distribution with variance $\sigma^2$ equal to the sum of the variances of the optimal potentials \citep[Theorem~1.1]{GonzalezSanz2023}. Hence, we roughly have that $\sqrt{n} \Delta_n \approx \sqrt{2 / \pi} \sigma$. Approximating said variance within our Monte Carlo simulation, we see in \autoref{fig:sim_sq_eucl_norm} that the scaled standard deviation seems to behave similar to the statistical error $\Delta_n$ in \autoref{fig:sim_sq_eucl_norm} (which also explains the latter). In particular, the variance obeys the LCA principle.

\begin{figure}[t!]
	\centerline{\includegraphics{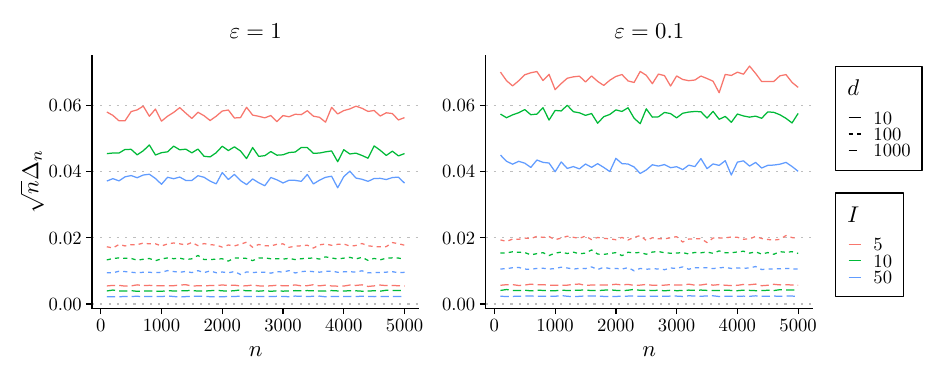}}
	\caption{Simulations of the mean absolute deviation $\Delta_n$ in the semi-discrete setting.} \label{fig:sim_semi_discrete}
\end{figure}

\begin{figure}[t!]
	\centerline{\includegraphics{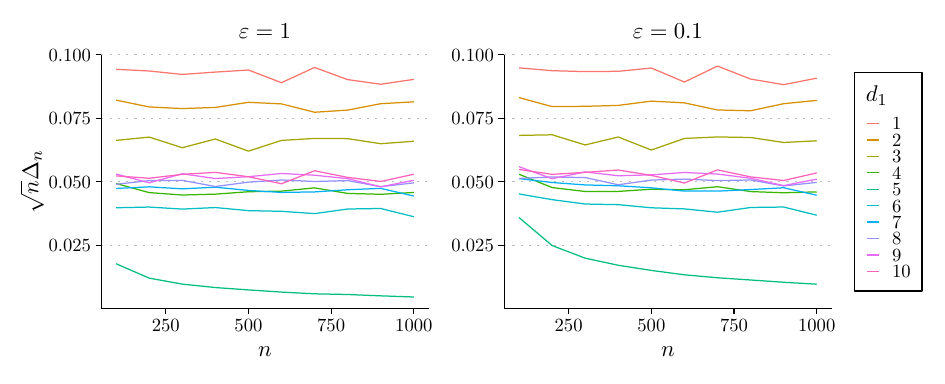}}
	\caption{Simulations of the mean absolute deviation $\Delta_n$ for the Sinkhorn divergence.} \label{fig:sim_sinkhorn_div}
\end{figure}

\autoref{fig:sim_semi_discrete} showcases the simulation results for the semi-discrete setting. We observe that the mean absolute deviation $\Delta_n$ decreases with higher $d$ and $I$. This indicates that the normalization by $d_1 \lor d_2$ is only proper in the sense that it achieves $\norminf{\c} \leq 1$, but does not capture the sharp dependency in the underlying constant for the convergence. A possible explanation for the behavior in $I$ is due to the fact  that with more points ($I$) the expected distance at which mass is assigned with respect to the uniform norm $\norminf{}$ decreases. Meanwhile, for increasing dimension ($d$) the average distance between two random points in the unit cube increases to one, and thus the cost function tends to be more homogenous with increasing dimension. Note in particular that for a constant cost function $c(x,y) \equiv a$, the entropic OT cost fulfills $T_{c,\epsilon}(\mu, \nu) = a$ for every $\mu, \nu \in \PC(\RR^d)$ and $\epsilon>0$. This indicates that statistical problem of estimating the population EOT cost in the semi-discrete setting rather simplifies as the dimension increases.
In particular, this behavior is remarkable as it suggests for the semi-discrete setting that the empirical EOT cost might benefit from additional structure.

Lastly, \autoref{fig:sim_sinkhorn_div} shows the results for the Sinkhorn divergence. We see a behavior which is in line with the cube setting, except for $d_1 = 5$ where the error decreases with increasing sample size. This is a consequence of the fact that the Sinkhorn divergence converges at the rate $n^{-1}$ under identical population measures \citep{gonzalez2022weak, Goldfeld2022LimitTF}. Moreover, the Sinkhorn divergence also seems to be affected by the LCA principle. However, recalling the discussion of the cube setting, note that in this case for the debiasing terms $\EOT(\mu, \mu)$ and $\EOT(\nu, \nu)$ in the setting of \autoref{thm:eot_lca_primal} the additional part $\c_2$ is zero. As a consequence, the estimation error caused by the debasing terms appears for these medium dimensional settings rather negligible in comparison to the estimation error for the estimation of $\EOT(\mu,\nu)$.

\section{Discussion} \label{sec:discussion}

In this work, we showed that the empirical EOT cost, just like the empirical OT cost, adheres to the LCA principle. More precisely, for suitably bounded costs, we derived an upper bound for the statistical estimation error of the empirical EOT cost in terms of $n$ and $\eps$ which only depends on the simpler probability measure.  We stress that this holds for the empirical EOT estimator and no additional knowledge of the underlying space is necessary, i.e., the estimator automatically adapts.
Further, we observed that the empirical EOT cost can be approximated using the Sinkhorn algorithm in at most $\bigO(n^{2.5})$ arithmetic operations such that the resulting quantity still obeys the LCA principle.
Most of our results are derived under boundedness of the cost function. This allows us to leverage bounds on the uniform metric entropy of the function class $\FC{}$ which in turn is needed for the crucial \autoref{lemma:entt_cov_num}. Nevertheless, \autoref{thm:lca_sg} shows that the LCA principle can also hold in a (partially) unbounded setting. As of now, this setting is limited to the squared Euclidean norm. We leave a more extensive analysis in the case of unbounded costs to future work.

Such an analysis will most likely entail the use of concentration constraints for the probability measures that are tailored to the cost function (like sub-Gaussianity for the squared Euclidean norm). More specifically, based on the convergence behavior for empirical OT in unbounded settings \citep{fournier2015rate,Staudt2023Convergence}, we conjecture for non-negative costs dominated by $c(x,y) \leq \kappa\left(\|x\|^p + \|y\|^p\right)$ for some $\kappa>0$ that finite moments of order $2p+\delta$ for $\delta>0$ for the two population measures are sufficient to infer parametric convergence rates in $n$. We also conjecture that under sufficient moment concentration the entropic LCA principle manifests for totally unbounded settings, i.e., that the dependency in $\epsilon$  only depends on the smoothness of the cost function and the minimum intrinsic dimension of the two population measures.

Moreover, for a complete statistical analysis of the empirical EOT cost it is critical to also obtain complementing lower bounds on the convergence rates. Such lower bounds will likely depend on the dual optimizers for empirical and population measures. Indeed, recently derived distributional limits by \citet{GonzalezSanz2023} assert that the empirical estimator $\EOT(\hmu_n,\hnu_n)$ asymptotically fluctuates around its population counterpart $\EOT(\mu,\nu)$ at variance $\textup{Var}_{X\sim \mu}[\phi(X)] + \textup{Var}_{Y\sim\nu}[\psi(Y)]$, where $(\phi, \psi)$ are optimal EOT potentials for $\mu$ and $\nu$. For fixed $\eps > 0$ this implies the parametric rate $n^{-1/2}$ to be sharp in $n$, nevertheless, the sharp dependency of the mean absolute deviation in terms of $\eps$ still remains open.

\section*{Acknowledgements}

The research of M. Groppe and S. Hundrieser is supported by the Research Training Group 2088 ``\emph{Discovering structure in complex data: Statistics meets Optimization and Inverse Problems}'', which is funded by the Deutsche Forschungsgemeinschaft (DFG, German Research Foundation). The authors thank Axel Munk and Marcel Klatt for fruitful discussions. Further, the authors acknowledge helpful comments by three anonymous referees, and Gonzalo Mena for spotting a mistake in a previous version of \autoref{thm:lca_sg}.

\bibliography{References.bib}

\begin{thebibliography}{95}
\providecommand{\natexlab}[1]{#1}
\providecommand{\url}[1]{\texttt{#1}}
\expandafter\ifx\csname urlstyle\endcsname\relax
  \providecommand{\doi}[1]{doi: #1}\else
  \providecommand{\doi}{doi: \begingroup \urlstyle{rm}\Url}\fi

\bibitem[Altschuler et~al.(2017)Altschuler, Niles-Weed, and
  Rigollet]{Altschuler2017}
Jason Altschuler, Jonathan Niles-Weed, and Philippe Rigollet.
\newblock Near-linear time approximation algorithms for optimal transport via
  {S}inkhorn iteration.
\newblock In I.~Guyon, U.~Von~Luxburg, et~al., editors, \emph{Advances in
  Neural Information Processing Systems}, volume~30, 2017.

\bibitem[Altschuler et~al.(2022)Altschuler, Niles-Weed, and
  Stromme]{altschuler2022asymptotics}
Jason~M Altschuler, Jonathan Niles-Weed, and Austin~J Stromme.
\newblock Asymptotics for semidiscrete entropic optimal transport.
\newblock \emph{SIAM Journal on Mathematical Analysis}, 54\penalty0
  (2):\penalty0 1718--1741, 2022.

\bibitem[Arjovsky et~al.(2017)Arjovsky, Chintala, and Bottou]{Arjovsky2017}
Martin Arjovsky, Soumith Chintala, and L{\'e}on Bottou.
\newblock {W}asserstein generative adversarial networks.
\newblock In Doina Precup and Yee~Whye Teh, editors, \emph{International
  Conference on Machine Learning}, pages 214--223. PMLR, 2017.

\bibitem[Bayraktar et~al.(2022)Bayraktar, Eckstein, and
  Zhang]{bayraktar2022stability}
Erhan Bayraktar, Stephan Eckstein, and Xin Zhang.
\newblock Stability and sample complexity of divergence regularized optimal
  transport.
\newblock \emph{Preprint arXiv:2212.00367}, 2022.

\bibitem[Bernton et~al.(2022)Bernton, Ghosal, and Nutz]{bernton2022entropic}
Espen Bernton, Promit Ghosal, and Marcel Nutz.
\newblock Entropic optimal transport: Geometry and large deviations.
\newblock \emph{Duke Mathematical Journal}, 171\penalty0 (16):\penalty0
  3363--3400, 2022.

\bibitem[Bertsekas(1981)]{Bertsekas1981}
Dimitri~P. Bertsekas.
\newblock A new algorithm for the assignment problem.
\newblock \emph{Mathematical Programming}, 21\penalty0 (1):\penalty0 152--171,
  1981.

\bibitem[Bertsekas and Castanon(1989)]{Bertsekas1989}
Dimitri~P. Bertsekas and David~A. Castanon.
\newblock The auction algorithm for the transportation problem.
\newblock \emph{Annals of Operations Research}, 20\penalty0 (1):\penalty0
  67--96, 1989.

\bibitem[Bigot et~al.(2019)Bigot, Cazelles, and Papadakis]{bigot2019CentralLT}
J{\'e}r{\'e}mie Bigot, Elsa Cazelles, and Nicolas Papadakis.
\newblock Central limit theorems for entropy-regularized optimal transport on
  finite spaces and statistical applications.
\newblock \emph{Electronic Journal of Statistics}, 13\penalty0 (2):\penalty0
  5120--5150, 2019.

\bibitem[Boissard and Le~Gouic(2014)]{boissard2014mean}
Emmanuel Boissard and Thibaut Le~Gouic.
\newblock On the mean speed of convergence of empirical and occupation measures
  in {W}asserstein distance.
\newblock \emph{Annales de l'Institut Henri Poincar\'e, Probabilit{\'e}s et
  Statistiques}, 50\penalty0 (2):\penalty0 539--563, 2014.

\bibitem[Bonneel and Digne(2023)]{bonneel2023survey}
Nicolas Bonneel and Julie Digne.
\newblock A survey of optimal transport for computer graphics and computer
  vision.
\newblock In \emph{Computer Graphics Forum}, volume~42, pages 439--460. Wiley
  Online Library, 2023.

\bibitem[Bronshtein(1976)]{Bronshtein1976}
Efim~Mikhailovich Bronshtein.
\newblock $\varepsilon$-entropy of convex sets and functions.
\newblock \emph{Siberian Mathematical Journal}, 17\penalty0 (3):\penalty0
  393--398, 1976.

\bibitem[Chizat et~al.(2020)Chizat, Roussillon, L\'{e}ger, Vialard, and
  Peyr\'{e}]{Chizat2020}
L\'{e}na\"{i}c Chizat, Pierre Roussillon, Flavien L\'{e}ger,
  Fran\c{c}ois-Xavier Vialard, and Gabriel Peyr\'{e}.
\newblock Faster {W}asserstein distance estimation with the {S}inkhorn
  divergence.
\newblock In H.~Larochelle, M.~Ranzato, et~al., editors, \emph{Advances in
  Neural Information Processing Systems}, volume~33, pages 2257--2269. Curran
  Associates, Inc., 2020.

\bibitem[Commander(2005)]{Commander2005}
Clayton~W. Commander.
\newblock A survey of the quadratic assignment problem, with applications.
\newblock \emph{Morehead Electronic Journal of Applicable Mathematics},
  4:\penalty0 MATH--2005--01, 2005.

\bibitem[Constantine and Savits(1996)]{Constantine1996}
G.~Constantine and T.~Savits.
\newblock A multivariate {F}aa di {B}runo formula with applications.
\newblock \emph{Transactions of the American Mathematical Society},
  348\penalty0 (2):\penalty0 503--520, 1996.

\bibitem[Courty et~al.(2014)Courty, Flamary, and Tuia]{Courty2014}
Nicolas Courty, R{\'e}mi Flamary, and Devis Tuia.
\newblock Domain adaptation with regularized optimal transport.
\newblock In Toon Calders and Florian Esposito, editors, \emph{Machine Learning
  and Knowledge Discovery in Databases}, pages 274--289. Springer, 2014.

\bibitem[Courty et~al.(2017)Courty, Flamary, Habrard, and
  Rakotomamonjy]{Courty2017}
Nicolas Courty, R{\'e}mi Flamary, Amaury Habrard, and Alain Rakotomamonjy.
\newblock Joint distribution optimal transportation for domain adaptation.
\newblock In I.~Guyon, U.~Von~Luxburg, et~al., editors, \emph{Advances in
  Neural Information Processing Systems}, volume~30, 2017.

\bibitem[Cuturi(2013)]{Cuturi2013}
Marco Cuturi.
\newblock {S}inkhorn distances: Lightspeed computation of optimal transport.
\newblock In C.J. Burges, L.~Bottou, et~al., editors, \emph{Advances in Neural
  Information Processing Systems}, volume~26. Curran Associates, Inc., 2013.

\bibitem[Deb et~al.(2021)Deb, Ghosal, and Sen]{Deb2021}
Nabarun Deb, Promit Ghosal, and Bodhisattva Sen.
\newblock Rates of estimation of optimal transport maps using plug-in
  estimators via barycentric projections.
\newblock In M.~Ranzato, A.~Beygelzimer, et~al., editors, \emph{Advances in
  Neural Information Processing Systems}, volume~34. Curran Associates, Inc.,
  2021.

\bibitem[del Barrio et~al.(1999)del Barrio, Cuesta-Albertos, Matr{\'a}n, and
  Rodr{\'i}guez-Rodr{\'i}guez]{Barrio1999}
Eustasio del Barrio, Juan~A. Cuesta-Albertos, Carlos Matr{\'a}n, and
  Jes{\'u}s~M. Rodr{\'i}guez-Rodr{\'i}guez.
\newblock Tests of goodness of fit based on the ${L}^2$-{W}asserstein distance.
\newblock \emph{The Annals of Statistics}, 27\penalty0 (4):\penalty0
  1230--1239, 1999.

\bibitem[del Barrio et~al.(2023)del Barrio, Gonz{\'a}lez-Sanz, Loubes, and
  Niles-Weed]{del2022improved}
Eustasio del Barrio, Alberto Gonz{\'a}lez-Sanz, Jean-Michel Loubes, and
  Jonathan Niles-Weed.
\newblock An improved central limit theorem and fast convergence rates for
  entropic transportation costs.
\newblock \emph{SIAM Journal on Mathematics of Data Science}, 5\penalty0
  (3):\penalty0 639--669, 2023.

\bibitem[Delalande(2022)]{delalande2022nearly}
Alex Delalande.
\newblock Nearly tight convergence bounds for semi-discrete entropic optimal
  transport.
\newblock In Gustau Camps-Valls, Francisco J.~R Ruiz, et~al., editors,
  \emph{International Conference on Artificial Intelligence and Statistics},
  pages 1619--1642. PMLR, 2022.

\bibitem[Dragomirescu and Ivan(1992)]{dragomirescu1992smallest}
F~Dragomirescu and C~Ivan.
\newblock The smallest convex extensions of a convex function.
\newblock \emph{Optimization}, 24\penalty0 (3-4):\penalty0 193--206, 1992.

\bibitem[Dudley(1969)]{dudley1969speed}
Richard~Mansfield Dudley.
\newblock The speed of mean {G}livenko-{C}antelli convergence.
\newblock \emph{The Annals of Mathematical Statistics}, 40\penalty0
  (1):\penalty0 40--50, 1969.

\bibitem[Dvurechensky et~al.(2018)Dvurechensky, Gasnikov, and
  Kroshnin]{dvurechensky2018computational}
Pavel Dvurechensky, Alexander Gasnikov, and Alexey Kroshnin.
\newblock Computational optimal transport: Complexity by accelerated gradient
  descent is better than by {S}inkhorn's algorithm.
\newblock In Jennifer Dy and Andreas Krause, editors, \emph{International
  Conference on Machine Learning}, pages 1367--1376. PMLR, 2018.

\bibitem[Eckstein and Nutz(2024)]{Eckstein2022}
Stephan Eckstein and Marcel Nutz.
\newblock Convergence rates for regularized optimal transport via quantization.
\newblock \emph{Preprint arXiv:2208.14391}, 49\penalty0 (2):\penalty0
  1223--1240, 2024.

\bibitem[Evans and Matsen(2012)]{Evans2012}
Steven~N. Evans and Frederick~A. Matsen.
\newblock The phylogenetic {K}antorovich--{R}ubinstein metric for environmental
  sequence samples.
\newblock \emph{Journal of the Royal Statistical Society: Series B (Statistical
  Methodology)}, 74\penalty0 (3):\penalty0 569--592, 2012.

\bibitem[Feydy et~al.(2019)Feydy, S{\'e}journ{\'e}, Vialard, Amari, Trouv{\'e},
  and Peyr{\'e}]{Feydy2019}
Jean Feydy, Thibault S{\'e}journ{\'e}, Fran{\c{c}}ois-Xavier Vialard, Shun-ichi
  Amari, Alain Trouv{\'e}, and Gabriel Peyr{\'e}.
\newblock Interpolating between optimal transport and {MMD} using {S}inkhorn
  divergences.
\newblock In Kamalika Chaudhuri and Masashi Sugiyama, editors, \emph{The 22nd
  International Conference on Artificial Intelligence and Statistics}, pages
  2681--2690. PMLR, 2019.

\bibitem[Flamary et~al.(2016)Flamary, Courty, Tuia, and
  Rakotomamonjy]{Flamary2016}
R.~Flamary, N.~Courty, D.~Tuia, and A.~Rakotomamonjy.
\newblock Optimal transport for domain adaptation.
\newblock \emph{IEEE Transactions on Pattern Analysis and Machine
  Intelligence}, 1, 2016.

\bibitem[Fournier and Guillin(2015)]{fournier2015rate}
Nicolas Fournier and Arnaud Guillin.
\newblock On the rate of convergence in {W}asserstein distance of the empirical
  measure.
\newblock \emph{Probability Theory and Related Fields}, 162\penalty0
  (3):\penalty0 707--738, 2015.

\bibitem[Galichon(2018)]{Galichon2018}
Alfred Galichon.
\newblock \emph{Optimal transport methods in economics}.
\newblock Princeton University Press, 2018.

\bibitem[Genevay et~al.(2019)Genevay, Chizat, Bach, Cuturi, and
  Peyr\'{e}]{Genevay2019}
Aude Genevay, L\'{e}na\"{i}c Chizat, Francis Bach, Marco Cuturi, and Gabriel
  Peyr\'{e}.
\newblock Sample complexity of {S}inkhorn divergences.
\newblock In Kamalika Chaudhuri and Masashi Sugiyama, editors, \emph{The 22nd
  International Conference on Artificial Intelligence and Statistics},
  volume~89 of \emph{Proceedings of Machine Learning Research}, pages
  1574--1583. PMLR, April 2019.

\bibitem[Gin{\'e} and Nickl(2015)]{Gine2015}
Evarist Gin{\'e} and Richard Nickl.
\newblock \emph{Mathematical Foundations of Infinite-Dimensional Statistical
  Models}.
\newblock Cambridge Series in Statistical and Probabilistic Mathematics.
  Cambridge University Press, 2015.

\bibitem[Goldfeld et~al.(2024{\natexlab{a}})Goldfeld, Kato, Rioux, and
  Sadhu]{Goldfeld2022}
Ziv Goldfeld, Kengo Kato, Gabriel Rioux, and Ritwik Sadhu.
\newblock Statistical inference with regularized optimal transport.
\newblock \emph{Information and Inference: A Journal of the IMA}, 13\penalty0
  (1):\penalty0 iaad056, 2024{\natexlab{a}}.

\bibitem[Goldfeld et~al.(2024{\natexlab{b}})Goldfeld, Kato, Rioux, and
  Sadhu]{Goldfeld2022LimitTF}
Ziv Goldfeld, Kengo Kato, Gabriel Rioux, and Ritwik Sadhu.
\newblock Limit theorems for entropic optimal transport maps and sinkhorn
  divergence.
\newblock \emph{Electronic Journal of Statistics}, 18\penalty0 (1):\penalty0
  980--1041, 2024{\natexlab{b}}.

\bibitem[Gonz{\'a}lez-Sanz and Hundrieser(2023)]{GonzalezSanz2023}
Alberto Gonz{\'a}lez-Sanz and Shayan Hundrieser.
\newblock Weak limits for empirical entropic optimal transport: Beyond smooth
  costs.
\newblock \emph{Preprint arXiv:2305.09745}, 2023.

\bibitem[Gonz\'alez-Sanz et~al.(2022)Gonz\'alez-Sanz, Loubes, and
  Niles-Weed]{gonzalez2022weak}
Alberto Gonz\'alez-Sanz, Jean-Michel Loubes, and Jonathan Niles-Weed.
\newblock Weak limits of entropy regularized optimal transport; potentials,
  plans and divergences.
\newblock \emph{Preprint arXiv:2207.07427}, 2022.

\bibitem[Grave et~al.(2019)Grave, Joulin, and Berthet]{grave2019unsupervised}
Edouard Grave, Armand Joulin, and Quentin Berthet.
\newblock Unsupervised alignment of embeddings with {W}asserstein procrustes.
\newblock In Kamalika Chaudhuri and Masashi Sugiyama, editors, \emph{The 22nd
  International Conference on Artificial Intelligence and Statistics}, pages
  1880--1890. PMLR, 2019.

\bibitem[Gulrajani et~al.(2017)Gulrajani, Ahmed, Arjovsky, Dumoulin, and
  Courville]{Gulrajani2017}
Ishaan Gulrajani, Faruk Ahmed, Martin Arjovsky, Vincent Dumoulin, and Aaron~C.
  Courville.
\newblock Improved training of {W}asserstein {GAN}s.
\newblock In I.~Guyon, U.~Von~Luxburg, et~al., editors, \emph{Advances in
  Neural Information Processing Systems}, volume~30, 2017.

\bibitem[Guntuboyina and Sen(2013)]{Guntuboyina2013}
Adityanand Guntuboyina and Bodhisattva Sen.
\newblock Covering numbers for convex functions.
\newblock \emph{IEEE Transactions on Information Theory}, 59\penalty0
  (4):\penalty0 1957--1965, 2013.

\bibitem[Hallin and Mordant(2022)]{hallin2022center}
Marc Hallin and Gilles Mordant.
\newblock Center-outward multiple-output {L}orenz curves and {G}ini indices a
  measure transportation approach.
\newblock \emph{Preprint arXiv:2211.10822}, 2022.

\bibitem[Hallin et~al.(2022)Hallin, Hlubinka, and
  Hudecov{\'a}]{hallin2022efficient}
Marc Hallin, Daniel Hlubinka, and {\v{S}}{\'a}rka Hudecov{\'a}.
\newblock Efficient fully distribution-free center-outward rank tests for
  multiple-output regression and {MANOVA}.
\newblock \emph{Journal of the American Statistical Association}, pages 1--17,
  2022.

\bibitem[Ho et~al.(2017)Ho, Nguyen, Yurochkin, Bui, Huynh, and
  Phung]{ho2017multilevel}
Nhat Ho, XuanLong Nguyen, Mikhail Yurochkin, Hung~Hai Bui, Viet Huynh, and Dinh
  Phung.
\newblock Multilevel clustering via {W}asserstein means.
\newblock In Doina Precup and Yee~Whye Teh, editors, \emph{International
  Conference on Machine Learning}, pages 1501--1509. PMLR, 2017.

\bibitem[Hundrieser et~al.(2024{\natexlab{a}})Hundrieser, Klatt, and
  Munk]{hundrieser2021limit}
Shayan Hundrieser, Marcel Klatt, and Axel Munk.
\newblock Limit distributions and sensitivity analysis for empirical entropic
  optimal transport on countable spaces.
\newblock \emph{The Annals of Applied Probability}, 34\penalty0 (1B):\penalty0
  1403--1468, 2024{\natexlab{a}}.

\bibitem[Hundrieser et~al.(2024{\natexlab{b}})Hundrieser, Staudt, and
  Munk]{Hundrieser2022}
Shayan Hundrieser, Thomas Staudt, and Axel Munk.
\newblock Empirical optimal transport between different measures adapts to
  lower complexity.
\newblock \emph{Annales de l'Institut Henri Poincar\'e, Probabilit{\'e}s et
  Statistiques [To Appear, preprint arXiv:2202.10434]}, 2024{\natexlab{b}}.

\bibitem[Kantorovitch(1942)]{Kantorovitch1942}
L.~Kantorovitch.
\newblock On the translocation of masses.
\newblock \emph{Doklady Akademii Nauk URSS}, 37:\penalty0 7--8, 1942.

\bibitem[Kantorovitch(1958)]{Kantorovitch1958}
L.~Kantorovitch.
\newblock On the translocation of masses.
\newblock \emph{Management Science}, 5\penalty0 (1):\penalty0 1--4, 1958.

\bibitem[Klatt et~al.(2020)Klatt, Tameling, and Munk]{klatt2020empirical}
Marcel Klatt, Carla Tameling, and Axel Munk.
\newblock Empirical regularized optimal transport: Statistical theory and
  applications.
\newblock \emph{SIAM Journal on Mathematics of Data Science}, 2\penalty0
  (2):\penalty0 419--443, 2020.

\bibitem[Kolmogorov and Tikhomirov(1961)]{Kolmogorov1961}
A.~N. Kolmogorov and V.~M. Tikhomirov.
\newblock $\epsilon$-entropy and $\epsilon$-capacity of sets in functional
  spaces.
\newblock In S.N. Cernikov, N.V. Cernikova, et~al., editors, \emph{Twelve
  Papers on Algebra and Real Functions}, American Mathematical Society
  Translations--series 2, pages 277--364. American Mathematical Society, 1961.

\bibitem[Lee(2013)]{lee2013smooth}
John~M Lee.
\newblock \emph{Introduction to smooth manifolds}, volume 218 of \emph{Graduate
  Texts in Mathematics}.
\newblock Springer, 2013.

\bibitem[Lin and Zha(2008)]{lin2008riemannian}
Tong Lin and Hongbin Zha.
\newblock {R}iemannian manifold learning.
\newblock \emph{IEEE Transactions on Pattern Analysis and Machine
  Intelligence}, 30\penalty0 (5):\penalty0 796--809, 2008.

\bibitem[Luo et~al.(2023)Luo, Yang, and Wei]{luo2023improved}
Jianzhou Luo, Dingchuan Yang, and Ke~Wei.
\newblock Improved complexity analysis of the {S}inkhorn and {G}reenkhorn
  algorithms for optimal transport.
\newblock \emph{Preprint arXiv:2305.14939}, 2023.

\bibitem[Manole and Niles-Weed(2024)]{manole2021sharp}
Tudor Manole and Jonathan Niles-Weed.
\newblock Sharp convergence rates for empirical optimal transport with smooth
  costs.
\newblock \emph{The Annals of Applied Probability}, 34\penalty0 (1B):\penalty0
  1108--1135, 2024.

\bibitem[Marino and Gerolin(2020)]{Marino2020}
Simone~Di Marino and Augusto Gerolin.
\newblock An optimal transport approach for the {S}chr{\"o}dinger bridge
  problem and convergence of {S}inkhorn algorithm.
\newblock \emph{Journal of Scientific Computing}, 85\penalty0 (2):\penalty0
  1--28, 2020.

\bibitem[M{\'e}moli(2011)]{Memoli2011}
F.~M{\'e}moli.
\newblock {G}romov--{W}asserstein distances and the metric approach to object
  matching.
\newblock \emph{Foundations of computational mathematics}, 11\penalty0
  (4):\penalty0 417--487, 2011.

\bibitem[Mena and Niles-Weed(2019)]{Mena2019}
Gonzalo Mena and Jonathan Niles-Weed.
\newblock Statistical bounds for entropic optimal transport: Sample complexity
  and the central limit theorem.
\newblock In H.~Wallach, H.~Larochelle, et~al., editors, \emph{Advances in
  Neural Information Processing Systems}, volume~32. Curran Associates, Inc.,
  2019.

\bibitem[Monge(1781)]{Monge1781}
G.~Monge.
\newblock M{\'e}moire sur la th{\'e}eorie des d{\'e}blais et des remblais.
\newblock \emph{Histoire de l'Acad{\'e}mie Royale des Sciences de Paris}, pages
  666--704, 1781.

\bibitem[Mordant and Segers(2022)]{Mordant2022}
Gilles Mordant and Johan Segers.
\newblock Measuring dependence between random vectors via optimal transport.
\newblock \emph{Journal of Multivariate Analysis}, 189:\penalty0 104912, 2022.

\bibitem[Munk and Czado(1998)]{Munk1998}
Axel Munk and Claudia Czado.
\newblock Nonparametric validation of similar distributions and assessment of
  goodness of fit.
\newblock \emph{Journal of the Royal Statistical Society: Series B (Statistical
  Methodology)}, 60\penalty0 (1):\penalty0 223--241, 1998.

\bibitem[Nies et~al.(2021)Nies, Staudt, and Munk]{Nies2021}
Thomas~Giacomo Nies, Thomas Staudt, and Axel Munk.
\newblock Transport dependency: Optimal transport based dependency measures.
\newblock \emph{Preprint arXiv:2105.02073}, 2021.

\bibitem[Niles-Weed and Rigollet(2022)]{niles2022estimation}
Jonathan Niles-Weed and Philippe Rigollet.
\newblock Estimation of {W}asserstein distances in the spiked transport model.
\newblock \emph{Bernoulli}, 28\penalty0 (4):\penalty0 2663--2688, 2022.

\bibitem[Nutz(2021)]{nutz2021introduction}
Marcel Nutz.
\newblock Introduction to entropic optimal transport.
\newblock \emph{Lecture notes, Columbia University}, 2021.

\bibitem[Nutz and Wiesel(2022)]{nutz2022entropic}
Marcel Nutz and Johannes Wiesel.
\newblock Entropic optimal transport: Convergence of potentials.
\newblock \emph{Probability Theory and Related Fields}, 184\penalty0
  (1-2):\penalty0 401--424, 2022.

\bibitem[Orlin(1988)]{Orlin1988}
James Orlin.
\newblock A faster strongly polynomial minimum cost flow algorithm.
\newblock In \emph{Proceedings of the Twentieth annual ACM symposium on Theory
  of Computing}, pages 377--387, 1988.

\bibitem[Pal(2024)]{pal2019difference}
Soumik Pal.
\newblock On the difference between entropic cost and the optimal transport
  cost.
\newblock \emph{The Annals of Applied Probability}, 34\penalty0 (1B):\penalty0
  1003--1028, 2024.

\bibitem[Panaretos and Zemel(2019)]{Panaretos2019}
Victor~M. Panaretos and Yoav Zemel.
\newblock Statistical aspects of {W}asserstein distances.
\newblock \emph{Annual Review of Statistics and Its Application}, 6:\penalty0
  405--431, 2019.

\bibitem[Panaretos and Zemel(2020)]{panaretos2020invitation}
Victor~M Panaretos and Yoav Zemel.
\newblock \emph{An invitation to statistics in {W}asserstein space}.
\newblock Springer Nature, 2020.

\bibitem[Peyr{\'e} and Cuturi(2019)]{Peyre2019}
G.~Peyr{\'e} and M.~Cuturi.
\newblock Computational optimal transport.
\newblock \emph{Foundations and Trends in Machine Learning}, 11\penalty0
  (5-6):\penalty0 355--607, 2019.

\bibitem[Peyr{\'e} et~al.(2016)Peyr{\'e}, Cuturi, and Solomon]{peyre2016gromov}
Gabriel Peyr{\'e}, Marco Cuturi, and Justin Solomon.
\newblock {G}romov-{W}asserstein averaging of kernel and distance matrices.
\newblock In Maria~Florina Balcan and Kilian~Q. Weinberger, editors,
  \emph{International Conference on Machine Learning}, pages 2664--2672. PMLR,
  2016.

\bibitem[Pooladian and Niles-Weed(2021)]{pooladian2021entropic}
Aram-Alexandre Pooladian and Jonathan Niles-Weed.
\newblock {E}ntropic estimation of optimal transport maps.
\newblock \emph{Preprint arXiv:2109.12004}, 2021.

\bibitem[Rachev and R{\"u}schendorf(1998{\natexlab{a}})]{Rachev1998a}
S.~Rachev and L.~R{\"u}schendorf.
\newblock \emph{Mass transportation problems - Volume I: Theory}.
\newblock Springer, 1998{\natexlab{a}}.

\bibitem[Rachev and R{\"u}schendorf(1998{\natexlab{b}})]{Rachev1998b}
S.~Rachev and L.~R{\"u}schendorf.
\newblock \emph{Mass transportation problems - Volume II: Applications}.
\newblock Springer, 1998{\natexlab{b}}.

\bibitem[Rigollet and Stromme(2022)]{Rigollet2022}
Philippe Rigollet and Austin~J. Stromme.
\newblock On the sample complexity of entropic optimal transport.
\newblock \emph{Preprint arXiv:2206.13472}, 2022.

\bibitem[Rioux et~al.(2023)Rioux, Goldfeld, and Kato]{rioux2023entropic}
Gabriel Rioux, Ziv Goldfeld, and Kengo Kato.
\newblock Entropic {G}romov-{W}asserstein distances: Stability, algorithms, and
  distributional limits.
\newblock \emph{Preprint arXiv:2306.00182}, 2023.

\bibitem[Santambrogio(2015)]{Santambrogio2015}
F.~Santambrogio.
\newblock \emph{Optimal Transport for Applied Mathematicians}.
\newblock Springer, 2015.

\bibitem[Schiebinger et~al.(2019)Schiebinger, Shu, Tabaka, Cleary, Subramanian,
  Solomon, Gould, Liu, Lin, Berube, et~al.]{Schiebinger2019}
Geoffrey Schiebinger, Jian Shu, Marcin Tabaka, Brian Cleary, Vidya Subramanian,
  Aryeh Solomon, Joshua Gould, Siyan Liu, Stacie Lin, Peter Berube, et~al.
\newblock Optimal-transport analysis of single-cell gene expression identifies
  developmental trajectories in reprogramming.
\newblock \emph{Cell}, 176\penalty0 (4):\penalty0 928--943, 2019.

\bibitem[Schmitzer(2019)]{Schmitzer2019}
Bernhard Schmitzer.
\newblock Stabilized sparse scaling algorithms for entropy regularized
  transport problems.
\newblock \emph{SIAM Journal on Scientific Computing}, 41\penalty0
  (3):\penalty0 A1443--A1481, 2019.

\bibitem[Solomon et~al.(2015)Solomon, De~Goes, Peyr{\'e}, Cuturi, Butscher,
  Nguyen, Du, and Guibas]{solomon2015convolutional}
Justin Solomon, Fernando De~Goes, Gabriel Peyr{\'e}, Marco Cuturi, Adrian
  Butscher, Andy Nguyen, Tao Du, and Leonidas Guibas.
\newblock Convolutional {W}asserstein distances: Efficient optimal
  transportation on geometric domains.
\newblock \emph{ACM Transactions on Graphics}, 34\penalty0 (4):\penalty0 1--11,
  2015.

\bibitem[Solomon et~al.(2016)Solomon, Peyr{\'e}, Kim, and
  Sra]{solomon2016entropic}
Justin Solomon, Gabriel Peyr{\'e}, Vladimir~G Kim, and Suvrit Sra.
\newblock Entropic metric alignment for correspondence problems.
\newblock \emph{ACM Transactions on Graphics}, 35\penalty0 (4):\penalty0 1--13,
  2016.

\bibitem[Sommerfeld and Munk(2018)]{sommerfeld2018inference}
Max Sommerfeld and Axel Munk.
\newblock Inference for empirical {W}asserstein distances on finite spaces.
\newblock \emph{Journal of the Royal Statistical Society. Series B (Statistical
  Methodology)}, 80\penalty0 (1):\penalty0 219--238, 2018.

\bibitem[Staudt and Hundrieser(2023)]{Staudt2023Convergence}
Thomas Staudt and Shayan Hundrieser.
\newblock Convergence of empirical optimal transport in unbounded settings.
\newblock \emph{Preprint arXiv:2306.11499}, 2023.

\bibitem[Stromme(2023)]{Stromme2023}
Austin~J. Stromme.
\newblock Minimum intrinsic dimension scaling for entropic optimal transport.
\newblock \emph{Preprint arXiv:2306.03398}, 2023.

\bibitem[Sturm(2012)]{Sturm2012}
Karl-Theodor Sturm.
\newblock The space of spaces: Curvature bounds and gradient flows on the space
  of metric measure spaces.
\newblock \emph{Preprint arXiv:1208.0434v2}, 2012.

\bibitem[Talwalkar et~al.(2008)Talwalkar, Kumar, and
  Rowley]{talwalkar2008large}
Ameet Talwalkar, Sanjiv Kumar, and Henry Rowley.
\newblock Large-scale manifold learning.
\newblock In \emph{2008 IEEE Conference on Computer Vision and Pattern
  Recognition}, pages 1--8. IEEE, 2008.

\bibitem[Tameling et~al.(2021)Tameling, Stoldt, Stephan, Naas, Jakobs, and
  Munk]{tameling2021colocalization}
Carla Tameling, Stefan Stoldt, Till Stephan, Julia Naas, Stefan Jakobs, and
  Axel Munk.
\newblock Colocalization for super-resolution microscopy via optimal transport.
\newblock \emph{Nature computational science}, 1\penalty0 (3):\penalty0
  199--211, 2021.

\bibitem[van~der Vaart and Wellner(1996)]{Vaart1996}
Aad~W. van~der Vaart and John~A. Wellner.
\newblock \emph{Weak Convergence and Empirical Processes: With Applications to
  Statistics}.
\newblock Springer Series in Statistics. Springer, 1996.

\bibitem[Vershynin(2018)]{vershynin2018high}
Roman Vershynin.
\newblock \emph{High-dimensional probability: An introduction with applications
  in data science}, volume~47.
\newblock Cambridge university press, 2018.

\bibitem[Villani(2003)]{Villani2003}
C.~Villani.
\newblock \emph{Topics in optimal transportation}.
\newblock American Mathematical Society, 2003.

\bibitem[Villani(2009)]{Villani2009}
C.~Villani.
\newblock \emph{Optimal Transport: Old and New}.
\newblock Springer, 2009.

\bibitem[von Luxburg and Bousquet(2004)]{Luxburg2004}
Ulrike von Luxburg and Olivier Bousquet.
\newblock Distance-based classification with {L}ipschitz functions.
\newblock \emph{Journal of Machine Learning Research}, 5:\penalty0 669--695,
  2004.

\bibitem[Wainwright(2019)]{Wainwright2019}
Martin~J. Wainwright.
\newblock \emph{High-Dimensional Statistics: A Non-Asymptotic Viewpoint}.
\newblock Cambridge Series in Statistical and Probabilistic Mathematics.
  Cambridge University Press, 2019.

\bibitem[Wang et~al.(2021)Wang, Cai, and Li]{Wang2021}
Shulei Wang, T.~Tony Cai, and Hongzhe Li.
\newblock Optimal estimation of {W}asserstein distance on a tree with an
  application to microbiome studies.
\newblock \emph{Journal of the American Statistical Association}, 116\penalty0
  (535):\penalty0 1237--1253, 2021.

\bibitem[Weed and Bach(2019)]{weed2019sharp}
Jonathan Weed and Francis Bach.
\newblock Sharp asymptotic and finite-sample rates of convergence of empirical
  measures in {W}asserstein distance.
\newblock \emph{Bernoulli}, 25\penalty0 (4A):\penalty0 2620--2648, 2019.

\bibitem[Weitkamp et~al.(2022)Weitkamp, Proksch, Tameling, and
  Munk]{Weitkamp2022}
Christoph~Alexander Weitkamp, Katharina Proksch, Carla Tameling, and Axel Munk.
\newblock Distribution of distances based object matching: Asymptotic
  inference.
\newblock \emph{Journal of the American Statistical Association}, 2022.

\bibitem[Zhang et~al.(2022)Zhang, Goldfeld, Mroueh, and
  Sriperumbudur]{Zhang2022}
Zhengxin Zhang, Ziv Goldfeld, Youssef Mroueh, and Bharath~K. Sriperumbudur.
\newblock {G}romov-{W}asserstein distances: Entropic regularization, duality,
  and sample complexity.
\newblock \emph{Preprint arXiv:2212.12848}, 2022.

\bibitem[Zhu et~al.(2018)Zhu, Liu, Cauley, Rosen, and Rosen]{zhu2018image}
Bo~Zhu, Jeremiah~Z Liu, Stephen~F Cauley, Bruce~R Rosen, and Matthew~S Rosen.
\newblock Image reconstruction by domain-transform manifold learning.
\newblock \emph{Nature}, 555\penalty0 (7697):\penalty0 487--492, 2018.

\end{thebibliography}

\appendix

\section{Omitted Proofs} \label{appendix:proofs}

\subsection{Duality and Complexity}\label{appendix:A2}

\begin{proof}[Proof of \autoref{thm:eot_duality}]
	First, note by \autoref{ass:eot_cost} that $\FC{} \subseteq \Lexp(\mu)$ and $\entt{\FC{}}{\mu} = \mset{\entt{\phi}{\mu} \mid \phi \in \FC{}} \subseteq \Lexp(\nu)$. Furthermore, from \autoref{rem:eot_dual_opt} we know that there exists a maximizing pair $\phi, \psi$ of \eqref{eq:eot_gen_duality} such that
	\begin{equation*}
		\phi = \entt{\psi}{\nu} \quad \text{ and } \quad \psi = \entt{\phi}{\mu}
	\end{equation*}
	as well as $\norminf{\phi}, \norminf{\psi} \leq 3 / 2$. This implies that $\phi \in \FC{}$ and thus
	\begin{equation*}
		\EOT(\mu, \nu) = \D{\mu}{\nu}(\phi, \psi) = \max_{\phi \in \FC{}} \D{\mu}{\nu}(\phi, \entt{\phi}{\mu})\,.
	\end{equation*}
	Moreover, by the Tonelli-Fubini theorem we have for $\phi \in \FC{}$ that
	\begin{align*}
		\int_{\XX \times \YY} \exp_\eps &(\phi(x) + \entt{\phi}{\mu}(y) - \c(x, y)) \dP{\mu}{x} \dP{\nu}{y} \\
		&= \int_{\YY} \exp_\eps(\entt{\phi}{\mu}(y)) \left[ \int_{\XX} \exp_\eps(\phi(x) - \c(x, y)) \dP{\mu}{x} \right] \dP{\nu}{y} \\
		&= \int_{\YY} \exp_\eps(\entt{\phi}{\mu}(y)) \exp_\eps(- \entt{\phi}{\mu}(y)) \dP{\nu}{y} = 1 \,,
	\end{align*}
	which yields
	\begin{equation*}
		\D{\mu}{\nu}(\phi, \entt{\phi}{\mu}) = \int_{\XX} \phi \de{\mu} + \int_{\YY} \entt{\phi}{\mu} \de{\nu}\,,
	\end{equation*}
	and we conclude \eqref{eq:fc_duality}.
\end{proof}

\begin{proof}[Proof of \autoref{lemma:eot_stab_bound}]
	We follow the proof of Proposition~2 from \citet{Mena2019}. First, according to \autoref{rem:eot_dual_opt} there are $\phi, \tphi \in \FC{}$ with
	 \begin{alignat*}{2}
	 	\phi &= \entt{\psi}{\nu}\,, & \qquad\qquad \psi &= \entt{\phi}{\mu}\,, \\
	 	\tphi &= \entt{\tpsi}{\nu}\,, & \tpsi &= \entt{\tphi}{\tmu}\,,
	 \end{alignat*}
 	such that
	\begin{equation*}
		\EOT(\mu, \nu) = \D{\mu}{\nu}(\phi, \psi)\,, \qquad \EOT(\tmu, \nu) = \D{\tmu}{\nu}(\tphi, \tpsi)\,.
	\end{equation*}
	By optimality, it holds that
	\begin{align*}
		\D{\mu}{\nu}(\tphi, \tpsi) - \D{\tmu}{\nu}(\tphi, \tpsi)
		&\leq \D{\mu}{\nu}(\phi, \psi) - \D{\tmu}{\nu}(\tphi, \tpsi) \\
		&\leq \D{\mu}{\nu}(\phi, \psi) - \D{\tmu}{\nu}(\phi, \psi)\,,
	\end{align*}
	which implies
	\begin{align*}
		\abs{\EOT(\mu, \nu) - \EOT(\tmu, \nu)}
		&= \abs{\D{\mu}{\nu}(\phi, \psi) - \D{\tmu}{\nu}(\tphi, \tpsi)} \\
		&\leq \abs{\D{\mu}{\nu}(\tphi, \tpsi) - \D{\tmu}{\nu}(\tphi, \tpsi)} \\
		&\qquad\qquad + \abs{\D{\mu}{\nu}(\phi, \psi) - \D{\tmu}{\nu}(\phi, \psi)}\,.
	\end{align*}
	As $\tphi = \entt{\tpsi}{\nu}$, we obtain using the Tonelli-Fubini theorem that
	\begin{align*}
		\int_{\XX \times \YY} \exp_\eps&(\tphi \oplus \tpsi - \c) \de{[(\mu - \tmu) \otimes \nu]} \\
		&= \int_{\XX} \exp_\eps(\tphi(x)) \int_{\YY} \exp_\eps(\tpsi(y) - \c(x, y)) \dP{\nu}{y} \dP{[\mu - \tmu]}{x} \\
		&= \int_{\XX} \exp_\eps(\tphi(x)) \exp_\eps(-\tphi(x)) \dP{[\mu - \tmu]}{x} = 0\,,
	\end{align*}
	which yields for the first term
	\begin{equation*}
		\abs{\D{\mu}{\nu}(\tphi, \tpsi) - \D{\tmu}{\nu}(\tphi, \tpsi)} = \abs*{\int_{\XX} \tphi \de{[\mu - \tmu]}} \leq \sup_{\phi \in \FC{}} \abs*{\int_{\XX} \phi \de{[\mu - \tmu]}}\,.
	\end{equation*}
	Analogously, we get the same bound for the second term $\abs{\D{\mu}{\nu}(\phi, \psi) - \D{\tmu}{\nu}(\phi, \psi)}$ and thus
	\begin{equation*}
		\abs{\EOT(\mu, \nu) - \EOT(\tmu, \nu)} \leq 2 \sup_{\phi \in \FC{}} \abs*{\int_{\XX} \phi \de{[\mu - \tmu]}}\,.
	\end{equation*}
	Similarly, we obtain that
	\begin{equation*}
		\abs{\EOT(\tmu, \nu) - \EOT(\tmu, \tnu)} \leq 2\sup_{\phi \in \FC{}} \abs*{ \int_{\YY} \entt{\phi}{\tmu} \de{[\nu - \tnu]}}\,.
	\end{equation*}
	Using the triangle inequality and combining the two bounds, we obtain the assertion.
\end{proof}

\begin{proof}[Proof of \autoref{lemma:entt_cov_num}] Note by assumption on $\sF$ that its $(c, \epsilon)$-transform is well-defined. W.l.o.g.\ we can assume that $N \defeq \covnum{\delta / 2}{\sF}{\norminf{}} < \infty$, else the asserted inequality is vacuous. Let $\mset{\tilde{\phi}_1, \ldots, \tilde{\phi}_N}$ be a $\delta/2$-covering for $\sF$ w.r.t $\norminf{}$. For each $i=1,\ldots, N$, we can pick a $\phi_i \in \sF$ such that $\norminf{\phi_i - \tilde{\phi}_i} \leq \delta / 2$. Using the triangle inequality, we see that $\mset{\phi_1, \ldots, \phi_N}$ is a $\delta$-covering for $\sF$ w.r.t.\ $\norminf{}$. By construction, note that every function of the $\delta$-covering is measurable and its entropic transform well-defined. Let $\phi \in \sF$, then there exists a $\phi_i$ such that $\norminf{\phi - \phi_i} \leq \delta$ or equivalently $\phi_i - \delta \leq \phi \leq \phi_i + \delta$. The monotonicity of the entropic $(\c, \eps)$-transform yields that
	\begin{equation*}
		\entt{\phi_i}{\tmu} - \delta = \entt{(\phi_i + \delta)}{\tmu} \leq \entt{\phi}{\tmu} \leq \entt{(\phi_i - \delta)}{\tmu} = \entt{\phi_i}{\tmu} + \delta\,.
	\end{equation*}
	Hence, we see that $\norminf{\entt{\phi}{\tmu} - \entt{\phi_i}{\tmu}} \leq \delta$ and $\mset{\entt{\phi_1}{\tmu}, \ldots, \entt{\phi_N}{\tmu}}$ is a $\delta$-covering for $\entt{\sF}{\tmu}$ w.r.t.\ $\norminf{}$.
\end{proof}

\begin{remark}
	In \autoref{lemma:entt_cov_num} and its proof the scale $\delta / 2$ appears in $\covnum{\delta / 2}{\sF}{\norminf{}}$ to sidestep potential measurability issues. Indeed, if we knew that the optimal covering of $\sF$ for every $\delta$ consisted only of measurable functions, then the factor $1 / 2$ would not be needed. However, this is not the case, in general. The measurability is crucial since the entropic $(c, \eps)$-transform requires the integration of the function.
\end{remark}

\subsection{Sample Complexity}

W.l.o.g.\ we assume in the following proofs that $I = 1$ and $g_1 = \id_{\XX}$, i.e., $\sU_1 = \XX$. Based on \autoref{lemma:cov_num_comp} and \autoref{lemma:cov_num_union} this is not a genuine restriction.

\begin{proof}[Proof of \autoref{thm:lca_sd}]
	By definition, it holds for $\phi \in \FC{}$ that $\norminf{\phi} \leq 3 / 2$. Hence, it follows for all $i =1,\ldots, I$ that
	\begin{equation*}
		\covnum{\delta}{ \restr{\FC{}}{\mset{x_i}}}{\norminf{}} \leq \ceil{3 / \delta}
	\end{equation*}
	and using the union bound \autoref{lemma:cov_num_union} we obtain
	\begin{equation*}
		\log \covnum{\delta}{ \FC{}}{\norminf{}} \lesssim I \delta^{-1}\,.
	\end{equation*}
	An application of \autoref{thm:eot_lca} yields the assertion.
\end{proof}

\begin{proof}[Proof of \autoref{thm:lca_lip}]
	Note that by assumption $\c$ is $1$-Lipschitz in the first component. By Proposition~2.4 from \citet{Marino2020}, it follows that $\entt{\psi}{\nu} \in \FC{\nu}$ is also $1$-Lipschitz. Hence, the class $\FC{}$ is contained in the set of uniformly bounded $1$-Lipschitz functions. An application of \autoref{lemma:cov_num_lip} combined with \autoref{thm:eot_lca} yields the assertion.
\end{proof}

\begin{proof}[Proof of \autoref{thm:lca_sc}]
	Upon defining
	\begin{equation*}
		\trho : \XX \times \YY \to \R\,, \qquad (x, y) \mapsto \c(x, y) - \norm{x}_2^2\,,
	\end{equation*}
	$1$-semi-concavity of $\c$ means that for all $y \in \YY$, $t \in (0, 1)$ and $x_1, x_2 \in \XX$ it holds that
	\begin{equation*}
		\trho(tx_1 + (1-t) x_2, y) \geq t \trho(x_1, y) + (1-t) \trho(x_2, y)\,.
	\end{equation*}
	This and the H\"older inequality with $p \defeq 1 / t$, $q \defeq 1 / (1-t)$ yield for $\entt{\psi}{\xi} \in \FC{}$ that
	\begin{align*}
		\int_{\YY} \exp_\eps&(\psi(y) - \trho(tx_1 + (1-t)x_2, y)) \dP{\xi}{y} \\
		&\leq \int_{\YY} \exp_\eps(\psi(y) - \trho(x_1, y))^t \exp_\eps(\psi(y) - \trho(x_2, y))^{1-t} \dP{\xi}{y} \\
		&\leq \left[ \int_{\YY} \exp_\eps(\psi(y) - \trho(x_1, y)) \dP{\xi}{y} \right]^t \left[ \int_{\YY} \exp_\eps(\psi(y) - \trho(x_2, y)) \dP{\xi}{y} \right]^{1-t}\,,
	\end{align*}
	and therefore
	\begin{equation*}
		\enttrasf{\psi}{\trho}{\eps}{\xi}(t x_1 + (1-t) x_2) \geq t \enttrasf{\psi}{\trho}{\eps}{\xi}(x_1) + (1-t) \enttrasf{\psi}{\trho}{\eps}{\xi}(x_2)\,.
	\end{equation*}
	Since for any $x \in \XX$ it holds that
	\begin{equation*}
		\enttrasf{\psi}{\trho}{\eps}{\xi}(x) = \enttrasf{(\psi + \norm{}_2^2)}{\c}{\eps}{\xi}(x) = \enttrasf{\psi}{\c}{\eps}{\xi}(x) - \norm{x}_2^2\,,
	\end{equation*}
	we conclude that $\enttrasf{\psi}{\c}{\eps}{\xi}$ is 1-semi-concave. Further, note that \autoref{ass:semi_con} includes \autoref{ass:lip}. Hence, $\FC{}$ is contained in the set of bounded, $1$-Lipschitz and $1$-semi-concave functions.  \autoref{lemma:cov_num_sc} combined with \autoref{thm:eot_lca} thus yield the assertion.
\end{proof}

\begin{lemma}[Structure of derivatives, {\citealt[Lemma~1]{Genevay2019}}]\label{lemma:deriv_rec}
	Let \autoref{ass:eot_cost} and \autoref{ass:hoelder} (w.l.o.g.\ $I = 1$ and $g_1 = \id_{\XX}$) hold. Then, for $\phi = \entt{\psi}{\xi} \in \FC{}$, $k \in \idn{s}^\kappa$, $\kappa \leq \alpha$, and $x \in \mathring{\XX}$ it follows that
	\begin{equation*}
		\De^k \phi (x) = \int_{\YY} \Phi^k_{\kappa}(x, y) \gamma_\eps(x, y) \dP{\xi}{y}\,,
	\end{equation*}
	where
	\begin{equation*}
		\gamma_\eps \defeq \exp_\eps(\phi \oplus \psi - \c) \,, \qquad \Phi^k_1 \defeq \diffx{k_1} \c\,,
	\end{equation*}
	and for $m=2,\ldots, \kappa$ we set
	\begin{equation*}
		\Phi_m^k \defeq \diffx{k_m} \Phi_{m-1}^k + \frac{1}{\eps} \left[ \diffx{k_m} \phi - \diffx{k_m} \c \right] \Phi_{m-1}^k\,.
	\end{equation*}
\end{lemma}

\begin{proof}[Proof of \autoref{lemma:infnorm_bound}]
	W.l.o.g.\ we can assume that $g_i = \id_\XX$ and write $C_m = C_{i,m}$ and $C^{(m)} = C^{(i, m)}$. Recall \autoref{lemma:deriv_rec}. We show that for all $k \in \idn{d}^\kappa$, $\kappa \leq \alpha$, $m = 0,\ldots,\kappa-2$ and indices tuple $\abs{j} = 0, \ldots, m$ it holds that
	\begin{equation*}
		\norminf{\De^j \Phi^k_{\kappa-m}} \lesssim (\eps \land 1)^{-(\kappa-m+\abs{j}-1)} C^{(\kappa-m + \abs{j})}\,,
	\end{equation*}
	where the implicit constant only depends on $\kappa$. Note that the above bound also holds for $\kappa = 1$ by definition of $\Phi^k_1$. Then, we can conclude by noting that
	\begin{equation*}
		\norminf{\De^k\phi} \leq \norminf{\Phi^k_{\kappa}}\,.
	\end{equation*}

	We follow the proof of Lemma~2 from \citet{Genevay2019} and prove this by induction over $\kappa = \abs{k}$. W.l.o.g.\ we can assume that $0 < \eps \leq 1$. Indeed, for $\eps > 1$ in all the following manipulations the term $1/\eps$ can be bounded by $1$.

	We have the base case $\kappa = 2$ which implies $m = 0$ and $\abs{j} = 0$. Hence,
	\begin{align*}
		\abs{\De^j \Phi^k_{\kappa-m}}
		&= \abs{\Phi^k_2} = \abs*{\diffx{k_2} \diffx{k_1} \c + \frac{1}{\eps} \left[ \diffx{k_2} \phi - \diffx{k_2} \c \right] \diffx{k_1} \c} \\
		&\leq C_2 + \frac{1}{\eps} [C_1 + C_1] C_1 \lesssim \eps^{-1} C^{(2)} = \eps^{-(\kappa-m+\abs{j}-1)} C^{(\kappa-m + \abs{j})}\,.
	\end{align*}

	We now suppose the validity of the assertion for given $\kappa$ and verify it for $\kappa + 1$. To this end, assume that we have
	\begin{equation*}
		\norminf{\De^j \Phi^k_{\kappa-m}} \lesssim \eps^{-(\kappa-m+\abs{j}-1)} C^{(\kappa-m + \abs{j})}
	\end{equation*}
	for all $\abs{k} = \kappa$, $m=0,\ldots,\kappa-2$ and $\abs{j} = 0, \ldots, m$, and we want to extend this to $\norminf{\De^j \Phi^k_{\kappa+1-m}}$ for $\abs{k} = \kappa+1$, $m=0,\ldots, \kappa -1$ and $\abs{j} = 0,\ldots, m$.

	Denote with $k'$ the first $\kappa$ components of $k$. Then, as $\Phi^k_{m}$ only depends on the first $m$ components of $k$, note that $\Phi^k_m = \Phi^{k'}_{m}$ for all $m = 0, \ldots, \kappa$. Hence, it follows for $m = 1, \ldots, \kappa - 1$ and $\abs{j} = 0, \ldots, m-1$ by the induction assumption that
	\begin{equation*}
		\norminf{\De^j \Phi^k_{\kappa + 1 - m}} = \norminf{\De^j \Phi^{k'}_{\kappa - (m - 1)}} \lesssim \eps^{-(\kappa + 1 - m + \abs{j} - 1)} C^{(\kappa + 1 - m + \abs{j})}\,.
	\end{equation*}
	Consequently, it remains to show the validity of the bounds for $m=0\ldots,\kappa-1$ and $\abs{j} = m$. To this end, we employ reverse induction on $m$.

	Now, the base case is $m = \kappa-1$ and thus $\kappa + 1 - m = 2$, $\abs{j} = \kappa-1$. The multivariate Leibniz rule yields
	\begin{align*}
		\De^j \Phi_{\kappa+1-m}^k &= \De^j \Phi_2^k = \De^j \left( \diffx{k_2} \diffx{k_1} \c \right) + \frac{1}{\eps} \De^j \left( \left[ \diffx{k_2} \phi - \diffx{k_2} \c \right] \diffx{k_1} \c \right) \\
		&= \De^{(j, k_2, k_1)} \c + \frac{1}{\eps} \sum_{i \subseteq j} \binom{j}{i} \left[ \De^{(i,k_2)} \phi - \De^{(i,k_2)} \c \right] \De^{(j - i, k_1)} \c\,.
	\end{align*}
	As $\int_{\YY} \gamma_\eps(x, y) \dP{\xi}{y} = 1$, note that
	\begin{equation*}
		\norminf{\De^{(i, k_2)} \phi} \leq \norminf{\Phi^{(i, k_2)}_{\abs{i} + 1}} \lesssim \eps^{-(\abs{i} + 1 - 1)} C^{(\abs{i} + 1)}
	\end{equation*}
	by the first induction assumption as $\abs{(i, k_2)} = \abs{i} + 1 \leq \abs{j} + 1 = \kappa$. Consequently,
	\begin{align*}
		\abs{\De^j \Phi_{\kappa+1-m}^k} & \lesssim C_{\abs{j} + 2} + \eps^{-1} \sum_{i \subseteq j} \binom{j}{i} \left[\eps^{-\abs{i}} C^{(\abs{i} + 1)} + C_{\abs{i} + 1}\right] C_{\abs{j} - \abs{i} + 1} \\
		&\lesssim C^{(\kappa+1)} + \eps^{-1} \eps^{-\abs{j}} C^{(\abs{j} + 2)} \sum_{i \subseteq j} \binom{j}{i} \\
		&\lesssim \eps^{-(\kappa + 1 - m + \abs{j} - 1)} C^{(\kappa + 1 - m + \abs{j})}\,.
	\end{align*}

	Assume that we have the bounds for $\norminf{\De^i \Phi_{\kappa+1-m}^k}$ where $m \leq \abs{i} \leq \kappa -1$ and extend them to $m - 1$. For $\abs{j} = m - 1$, we get similarly to before
	\begin{align*}
		&\abs{\De^j \Phi^k_{\kappa+1-(m-1)}} = \abs{\De^j \Phi_{\kappa+2-m}^k} \\
		&= \abs[\Big]{\De^{(j, k_{\kappa+2-m})} \Phi^k_{\kappa+1-m} + \eps^{-1} \sum_{i \subseteq j} \binom{j}{i} \left[ \De^{(i, k_{\kappa+2-m})} \phi - \De^{(i, k_{\kappa+2-m})} \c \right] \De^{j-i} \Phi^k_{\kappa+1-m}} \\
		&\lesssim \eps^{-(\kappa+1 - m + \abs{j} + 1 - 1)} C^{(\kappa+1-m+\abs{j}+1)} \\
		&\qquad + \eps^{-1} \sum_{i \subseteq j} \binom{j}{i} [ \eps^{-\abs{i}} C^{(\abs{i} + 1)} + C_{\abs{i} + 1} ] \eps^{-(\kappa+1-m+\abs{j} - \abs{i} - 1)} C^{(\kappa+1-m+\abs{j} - \abs{i})} \\
		&\lesssim \eps^{-\kappa} C^{(\kappa+1)} + \eps^{-(\kappa-1)} C^{(\kappa+1)} \sum_{i \subseteq j} \binom{j}{i} \\
		&\lesssim \eps^{-\kappa} C^{(\kappa+1)} = \eps^{-(\kappa+1-(m-1)+\abs{j} - 1)} C^{(\kappa+1-(m-1)+\abs{j})}\,. \qedhere
	\end{align*}
\end{proof}

\begin{proof}[Proof of \autoref{prop:unif_metric_ent_hoelder}]
	W.l.o.g.\ assume that $0 < \eps \leq 1$ and let $\phi \in \FC{}$. Using \autoref{lemma:infnorm_bound}, we get for $\abs{k} = \alpha - 1$ that
	\begin{equation*}
		\norminf{\De^k \phi} \lesssim \eps^{-(\alpha-2)} C^{(\alpha - 1)}\,.
	\end{equation*}
	Similarly, we have for $x_1, x_2 \in \mathring{\XX}$ that
	\begin{equation*}
		\norm{\De^k \phi(x_1) - \De^k\phi(x_2)}_2 \leq \norm{\nabla \De^k \phi}_2 \norm{x_1 - x_2}_2 \lesssim \eps^{-(\alpha - 1)} C^{(\alpha)} \norm{x_1 - x_2}_2\,.
	\end{equation*}
	Since $\norm{x_1 - x_2}_2 \leq \diam(\XX)$, this yields
	\begin{equation*}
		\norm{\phi}_\alpha \lesssim \eps^{-(\alpha - 2)} C^{(\alpha - 1)} + \eps^{-(\alpha - 1)} C^{(\alpha)} \diam(\XX) \lesssim \eps^{-(\alpha - 1)} C^{(\alpha)}\,.
	\end{equation*}
	Hence, we conclude that $\FC{} \subseteq \sC^{\alpha}_{M}(\XX)$ for $M \defeq \eps^{-(\alpha-1)} C^{(\alpha)} K$ with some $K > 0$ that only depends on $\XX$ and $\alpha$. Now, \autoref{lemma:cov_num_hoelder} yields
	\begin{align*}
		\log \covnum{\delta}{\FC{}}{\norminf{}}
		&\leq \log \covnum{\delta}{\sC^{\alpha}_{M}(\XX)}{\norminf{}} \\
		&\lesssim M^{s/\alpha} \delta^{-s/\alpha} \lesssim [C^{(\alpha)}]^{s/\alpha} \eps^{-s (\alpha - 1) / \alpha} \delta^{-s/\alpha}\,. \qedhere
	\end{align*}
\end{proof}

\begin{proof}[Proof of \autoref{lemma:infnorm_bound_sg}]
	Write $g = g_i$ and $G_{m} = G_{i,m}$, $G^{(m)} = G^{(i,m)}$. We extend the bounds provided by \citet[Proposition~1]{Mena2019} to compositions $\phi \circ g$ with $\phi = \entt{\psi}{\xi} \in \sF_\sigma$ for some measure $\xi \in \SG_d(\sigma^2)$. Denoting $\bar{\phi} \defeq \phi \circ g - \frac{1}{2} \norm{}_2^2 \circ g$, we want to bound for $k \in \idn{s}^{\kappa}$, $\kappa \leq \alpha$, $u \in \mathring{\sU}$ the partial derivative
	\begin{equation*}
		\De^k \bar{\phi}(u) = - \De^k \log(\exp[-\bar{\phi}(u)])
	\end{equation*}
	via the multivariate Fa\'a di Bruno formula \citep{Constantine1996}. First, note that for all $x > 0$ it holds with some constant $\lambda_\kappa$ that
	\begin{equation*}
		\frac{\partial^\kappa}{\partial x} \log(x) = \lambda_k \frac{1}{x^\kappa}
	\end{equation*}
	Furthermore, we obtain similarly to \autoref{lemma:deriv_rec} that
	\begin{equation*}
		\De^k\exp(-\bar{\phi}(u)) = \int_{\YY} \Phi_{\kappa}^k(u, y) \exp( \psi(y) - \frac{1}{2} \norm{y}_2^2 + \scalp{g(u)}{y}) \dP{\xi}{y}\,,
	\end{equation*}
	where $\Phi_0^{k}(u, y) \defeq 1$ and for $m=1,\ldots,\kappa$ we set
	\begin{equation} \label{eq:eot_unbound_rec}
		\Phi_m^{k}(u, y) \defeq \diffu{k_m} \Phi_{m-1}^k(u, y) + \Phi^k_{m-1}(u, y) \sum_{q=1}^d \diffu{k_m} g(u)_{q} y_{q}\,.
	\end{equation}
	Hence, upon defining
	\begin{equation*}
		\Psi_k(u) \defeq \frac{ \int_{\YY} \Phi_{\kappa}^k(u, y) \exp( \psi(y) - \frac{1}{2} \norm{y}_2^2 + \scalp{g(u)}{y}) \dP{\xi}{y}}{\int_{\YY} \exp(\psi(y) - \frac{1}{2} \norm{y}_2^2 + \scalp{g(u)}{y}) \dP{\xi}{y}}\,,
	\end{equation*}
	it follows by the multivariate Fa\'a di Bruno formula that
	\begin{equation} \label{eq:eot_di_bruno}
		\De^k \bar{\phi}(u) = -\De^k \log(\exp[-\bar{\phi}(u)]) = \sum_{j_1, \ldots j_\kappa} \lambda_{\kappa, j_1, \ldots, j_\kappa} \prod_{i=1}^{\kappa} \Psi_{j_i}(u)\,,
	\end{equation}
	where the sum runs over all sub-indices $j_1, \ldots, j_\kappa$ that partition $k$, i.e., $(j_1, \ldots, j_\kappa)$ is equal to a permuted version of $k$, and the $\lambda_{\kappa, j_1, \ldots, j_\kappa}$ are some constants related to the derivatives of the logarithm. We show that
	\begin{equation*}
		\abs{\Psi_k(u)} \lesssim G^{(\kappa)} \sum_{r=1}^{\kappa} \sum_{q_1, \ldots, q_{r} = 1}^d \frac{ \int_{\YY} \prod_{t=1}^r \abs{y_{q_t}} \exp( \psi(y) - \frac{1}{2} \norm{y}_2^2 + \scalp{g(u)}{y}) \dP{\xi}{y}}{\int_{\YY} \exp(\psi(y) - \frac{1}{2} \norm{y}_2^2 + \scalp{g(u)}{y}) \dP{\xi}{y}}\,,
	\end{equation*}
	where the implicit constant only depends on $\kappa$. Together with \eqref{eq:eot_di_bruno} and using the bound by \citet[Lemma~3 in Appendix~B]{Mena2019} for each summand, we obtain the desired result.

	It suffices to show that for $k \in \idn{s}^\kappa$ it holds with $m = 0, \ldots, \kappa - 1$ and indices tuples $\abs{j} = 0, \ldots, m$ that
	\begin{equation*}
		\abs{\De^j \Phi^k_{\kappa - m}(u, y)} \lesssim G^{(\kappa - m + \abs{j})} \sum_{r=1}^{\kappa - m} \sum_{q_1, \ldots, q_{r} = 1}^d \prod_{t=1}^r \abs{y_{q_t}}\,,
	\end{equation*}
	where the implicit constant only depends on $\kappa$. Then, with $m = 0 = \abs{j}$ we can conclude.

	We prove this similar to \autoref{lemma:infnorm_bound} by double induction. In the following, we write $\Phi^{k}_m \equiv \Phi^{k}_m(u, y)$. First, we do induction over $\kappa = \abs{k}$.

	For the base case $\kappa = 1$ and thus $m = 0 = \abs{j}$, it holds due to \eqref{eq:eot_unbound_rec} that
	\begin{equation*}
		\abs{\De^j \Phi_{\kappa - m}^k} = \abs{\Phi_{1}^k} = \abs[\Big]{ \sum_{q=1}^d \diffu{k_1} g(u)_q y_q } \lesssim G^{(1)} \sum_{q=1}^d \abs{y_q} = G^{(\kappa - m)} \sum_{r=1}^{\kappa - m} \sum_{q_1, \ldots, q_{r} = 1}^d \prod_{t=1}^r \abs{y_{q_t}} \,.
	\end{equation*}

	Let the bound on $\abs{\De^j \Phi^k_{\kappa - m}}$ hold for all $\abs{k} = \kappa$, $m = 0,\ldots,\kappa-1$ and $\abs{j} = 0,\ldots,m$. We extend this to $\kappa + 1$. Again, we only have to bound the new diagonal $m=0,\ldots,\kappa$ and $\abs{j} = m$. We do this by reverse induction on $m$.

	The base case $m = \kappa = \abs{j}$ holds as
	\begin{align*}
		\abs{\De^j \Phi_{\kappa+1-m}^k} &= \abs{\De^j \Phi_1^k} = \abs[\Big]{ \sum_{q=1}^d \De^{(j, k_1)} g(u)_q y_q } \\
		&\lesssim G^{(\abs{j} + 1)} \sum_{q=1}^d \abs{y_q} = G^{(\kappa + 1 - m)} \sum_{r=1}^{\kappa + 1 - m} \sum_{q_1, \ldots, q_{r} = 1}^d \prod_{t=1}^r \abs{y_{q_t}}\,.
	\end{align*}

	Suppose that $\abs{\De^i \Phi^k_{\kappa + 1 - m}}$ is bounded as required for $m \leq \abs{i} \leq \kappa$. Then, we need to extend this to $m-1$. Using \eqref{eq:eot_unbound_rec} and the multivariate Leibniz rule yields for $\abs{j} = m - 1$ that
	\begingroup
	\allowdisplaybreaks
	\begin{align*}
		&\abs{\De^j \Phi_{\kappa + 1 - (m - 1)}^k} \\
		&\qquad = \abs[\Big]{\De^{(j, k_{\kappa + 2 - m})} \Phi^k_{\kappa + 1 - m} + \sum_{q=1}^d \De^{j}\left[ \Phi^k_{\kappa + 1 - m} \diffu{k_{\kappa + 2 - m}} g(u)_q y_q \right]} \\
		&\qquad= \abs[\Big]{\De^{(j, k_{\kappa + 2 - m})} \Phi^k_{\kappa + 1 - m} + \sum_{q=1}^d y_{q} \sum_{i \subseteq j} \binom{j}{i} \De^{(i, k_{\kappa + 2 -m})} [g(u)_q] \De^{j - i} \Phi^k_{\kappa + 1 - m}} \\
		&\qquad\lesssim G^{(\kappa + 1 - m + \abs{j} + 1)} \sum_{r=1}^{\kappa + 1 - m} \sum_{q_1, \ldots, q_{r} = 1}^d \prod_{t=1}^r \abs{y_{q_t}} \\
		&\qquad\qquad\qquad + \sum_{q=1}^d \abs{y_{q}} \sum_{i \subseteq j} \binom{j}{i} G^{(\abs{i}+1)} \left[ G^{(\kappa + 1 - m + \abs{j} - \abs{i})} \sum_{r=1}^{\kappa + 1 - m} \sum_{q_1, \ldots, q_{r} = 1}^d \prod_{t=1}^r \abs{y_{q_t}} \right] \\
		&\qquad\lesssim G^{(\kappa + 1 - m + \abs{j} + 1)} \sum_{r=1}^{\kappa + 1 - m} \sum_{q_1, \ldots, q_{r} = 1}^d \prod_{t=1}^r \abs{y_{q_t}} \\
		&\qquad\qquad\qquad + G^{(\kappa + 1 - m + \abs{j} + 1)} \sum_{r=1}^{\kappa + 1 - m + 1} \sum_{q_1, \ldots, q_{r} = 1}^d \prod_{t=1}^r \abs{y_{q_t}} \sum_{i \subseteq j} \binom{j}{i} \\
		&\qquad\lesssim G^{(\kappa + 1 - (m-1) + \abs{j})} \sum_{r=1}^{\kappa + 1 - (m - 1)} \sum_{q_1, \ldots, q_{r} = 1}^d \prod_{t=1}^r \abs{y_{q_t}}\,. \qedhere
	\end{align*}
	\endgroup
\end{proof}

\begin{lemma} \label{lemma:sigma_n}
	Let $\sigma > 0$ and $\mu \in \SG_d(\sigma^2)$, $\nu \in \SG_d(\sigma^2)$. Let $\sigma_n$ be the infimum over all $\tau > 0$ such that $\mu,\, \hmu_n,\, \nu, \,\hnu_n \in \SG_d(\tau^2)$. Then, it holds that
	\begin{align*}
		\Prob(\sigma^2_n > 6 \sigma^2) &\leq 4 n^{-1} \,, \\
		\EV{\sigma_n^{2k}} &\leq 2 k^k \sigma^{2k} \quad (k \in \N)\,.
	\end{align*}
\end{lemma}
\begin{proof}
	The second assertion follows from \citet[Lemma~B.4 in supplement]{Mena2019}. For the first assertion, denote
	\begin{align*}
		\tau_{1,n} &\defeq \EVV{X \sim \hmu_n}{ \exp_{4d\sigma^2}(\norm{X}_2^2)}\,, \quad &\tau_1 &\defeq \EVV{X \sim \mu}{ \exp_{4d\sigma^2}(\norm{X}_2^2)}\,, \\
		\tau_{2,n} &\defeq \EVV{Y \sim \hnu_n}{ \exp_{4d\sigma^2}(\norm{Y}_2^2)}\,, \quad &\tau_2 &\defeq \EVV{Y \sim \nu}{ \exp_{4d\sigma^2}(\norm{Y}_2^2)} \,.
	\end{align*}
	As in the proof of Lemma~B.4 in the supplement of \citet{Mena2019}, it follows that $\sigma_n^2 \leq 2 \sigma^2 \max(\tau_{1,n}, \tau_{2,n})$. Hence, it holds that
	\begin{equation*}
		\Prob(\sigma_n^2 > 6 \sigma^2) \leq \Prob(\max(\tau_{1,n}, \tau_{2,n}) > 3) \leq \Prob( \tau_{1,n} > 3) + \Prob(\tau_{2,n} > 3)\,.
	\end{equation*}
	We bound the first term and conclude, the second term can be dealt with analogously. Using the Chebyshev inequality, we get
	\begin{equation*}
		\Prob( \tau_{1,n} > 3) \leq \Prob( \abs{\tau_{1,n} - \tau_1} > 3 - \tau_1) \\
		\leq \frac{\Var{X \sim \mu}{\exp_{4d\sigma^2}(\norm{X}^2_2)}}{(3 - \tau_1)^2} n^{-1} \leq 2 n^{-1}\,,
	\end{equation*}
	where we used that by Sub-Gaussianity $\tau_1 \leq 2$ as well as
	\begin{equation*}
		\Var{X \sim \mu}{\exp_{4d\sigma^2}(\norm{X}^2_2)} \leq \EVV{X \sim \mu}{ \exp_{2d\sigma^2}(\norm{X}^2_2)} \leq 2\,. \qedhere
	\end{equation*}
\end{proof}

\begin{proof}[Proof of \autoref{thm:lca_sg}]
	First, recall that according to \autoref{lemma:cov_num_union} we can consider the case that $I = 1$. Given empirical probability measures $\hmu_n, \hnu_n$, we have due to \autoref{lemma:sigma_n} a random $\sigma_n^2$ such that $\mu, \hmu_n, \nu, \hnu_n$ are all in $\SG_d(\sigma^2_n)$. Similar to Corollary~2 in \citet{Mena2019} (and \autoref{lemma:eot_stab_bound}), we get
	\begin{align} \label{eq:sg_emp_bound}
		\EV{\abs{\EOT(\hmu_n, \hnu_n) - \EOT(\mu, \nu)}} &\leq 2 \EV*{ \sup_{\phi \in \sF_{\sigma_n}} \abs*{ \int_{\XX} \phi \de{[\mu - \hmu_n]} }} \\
		&\qquad\qquad+ 2 \EV*{\sup_{\phi \in \sF_{\sigma_n}} \abs*{ \int_{\YY} \entt{\phi}{\hmu_n} \de{[\nu - \hnu_n]}}} \nonumber \,.
	\end{align}
	We now bound the two terms separately and conclude. We decompose the first term of \eqref{eq:sg_emp_bound} into
	\begin{align*}
		\EV*{ \sup_{\phi \in \sF_{\sigma_n}} \abs*{ \int_{\XX} \phi \de{[\mu - \hmu_n]} }}
		&\leq \EV*{ \sigma_n^{4\alpha} \sup_{\phi \in \sF_{\sigma_n}} \abs*{ \int_{\XX} \sigma_n^{-4\alpha} [\phi - \frac{1}{2} \norm{}_2^2] \de{[\mu - \hmu_n]} }} \\
		&\qquad + \EV*{ \abs*{\int_{\XX} \frac{1}{2} \norm{}_2^2  \de{[\mu - \hmu_n]}} } \\
		&\lesssim (\EV{\sigma_n^{8\alpha}})^{1/2} \left( \EV*{ \sup_{\phi \in \sF_{\sigma_n}} \abs*{ \int_{\XX} \sigma_n^{-4\alpha} [\phi - \frac{1}{2} \norm{}_2^2] \de{[\mu - \hmu_n]} }^2} \right)^{1/2} \\
		&\qquad + r^2 n^{-1/2}\,,
	\end{align*}
	where the last step uses the Cauchy-Schwarz inequality. The expectation of $\sigma_n^{8\alpha}$ can be bounded via \autoref{lemma:sigma_n} by an explicit constant that depends on $\alpha$ times $\sigma^{8\alpha}$.
	Further, by definition the function class $\sF_{\sigma_n}$ is bounded by $6d^2r^2\sigma_n^4$, and hence $[\sF_{\sigma_n} - \frac{1}{2} \norm{}_2^2]$ is bounded by $6d^2r^2\sigma_n^4 + \frac{1}{2}r^2 \leq 8 d^2 r^2\sigma_n^4$. In conjunction with \autoref{lemma:infnorm_bound_sg}  we thus infer that
	\begin{equation*}
		\sigma_n^{-4\alpha} [\sF_{\sigma_n} - \frac{1}{2} \norm{}_2^2] \subseteq \sC_{ M }^\alpha(\sU) \text{ with } M \defeq [G^{(\alpha)}]^{s / \alpha} + 8  d^2 r^2\,,
	\end{equation*}
	where the first term in $M$ controls the derivatives of functions in $\sigma_n^{-4\alpha} [\sF_{\sigma_n} - \frac{1}{2} \norm{}_2^2]$ and the second term arises from our aforementioned upper bound. Note that the function class on the right-hand side is deterministic and independent of $\sigma_n^2$. Hence, we can apply Theorem~2.14.5 from \citet{Vaart1996}, Theorem~3.5.1 from \citet{Gine2015} and \autoref{lemma:cov_num_hoelder} to obtain
	\begin{align*}
		n &\EV*{ \sup_{\phi \in \sF_{\sigma_n}} \abs*{ \int_{\XX} \sigma_n^{-4\alpha} [\phi - \frac{1}{2} \norm{}_2^2] \de{[\mu - \hmu_n]} }^2}
		\\
		&\qquad\lesssim \left( \sqrt{n} \EV*{ \sup_{\phi \in \sF_{\sigma_n}} \abs*{ \int_{\XX} \sigma_n^{-4\alpha} [\phi - \frac{1}{2} \norm{}_2^2] \de{[\mu - \hmu_n]}}} + M \right)^2 \\
		&\qquad\lesssim \left( \EV*{\int_0^{M} \sqrt{\log 2 \covnum{\delta}{\sC_{M}^\alpha(\sU)}{ \norminf{}}} \de{\delta}} + M \right)^2  \\
		&\qquad\lesssim \left(  \int_0^{M} \sqrt{1 + M^{s/\alpha} \delta^{-s/\alpha}} \de{\delta} + M \right)^2 \\
		&\qquad\lesssim (M + M^{s / (2\alpha)} M^{1 - s /(2\alpha)} + M)^2 \lesssim M^2\,,
	\end{align*}
	where we used that $s / \alpha < 2$ and the implicit constant only depends on $\alpha$, $d$, $s$ and $\sU$. Combining these inequalities we obtain
	\begin{align*}
		\sqrt{n}\EV*{ \sup_{\phi \in \sF_{\sigma_n}} \abs*{ \int_{\XX} \phi \de{[\mu - \hmu_n]} }} \lesssim ([G^{(\alpha)}]^{s/\alpha} + d^2r^2) \sigma^{4\alpha} + r^2\,.
	\end{align*}
	For the second term of \eqref{eq:sg_emp_bound}, we decompose into the parts $\sigma_n^2 > 6 \sigma^2$ and $\sigma_n^2 \leq 6 \sigma^2$ and use the Cauchy-Schwarz inequality to obtain
	\begingroup
	\allowdisplaybreaks
	\begin{align*}
		&\EV*{\sup_{\phi \in \sF_{\sigma_n}} \abs*{ \int_{\YY} \entt{\phi}{\hmu_n} \de{[\nu - \hnu_n]}}} \\
		&\qquad = \EV*{ [ \indicfunc{ \sigma_n^2 > 6 \sigma^2} + \indicfunc{\sigma_n^2 \leq 6 \sigma^2}] \sup_{\phi \in \sF_{\sigma_n}} \abs*{ \int_{\YY} \entt{\phi}{\hmu_n} \de{[\nu - \hnu_n]}} }\\
		&\qquad\leq ( \Prob[\sigma_n^2 > 6 \sigma^2])^{1/2} \left( \EV*{ \left( \sup_{\phi \in \sF_{\sigma_n}} \int_{\YY} \abs{\entt{\phi}{\hmu_n}} \de{\nu} +  \sup_{\phi \in \sF_{\sigma_n}} \int_{\YY} \abs{\entt{\phi}{\hmu_n}} \de{\hnu_n} \right)^2 } \right)^{1/2} \\
		&\qquad\qquad + \EV*{\sup_{\phi \in \sF_{\sqrt{6} \sigma}} \abs*{ \int_{\YY} \entt{\phi}{\hmu_n} \de{[\nu - \hnu_n]}}}\,.
	\end{align*}
	\endgroup
	The probability can be bounded via \autoref{lemma:sigma_n}. Furthermore, employing that $\sF_{\sigma}$ is bounded in uniform norm by $6d^2r^2\sigma^4$, the definition of the entropic transform and that for $x \in \XX$, $y \in \YY:$ $\norm{x - y}^2 \leq r^2 + 2 r \norm{y}_2 + \norm{y}_2^2 \leq 4 r^2 + 4 r \norm{y}^2_2$, it follows for $\phi \in \sF_{\sigma}$ that
	\begin{equation} \label{eq:sg_entt_bound}
		\abs{\entt{\phi}{\hmu_n}(y)} \leq 8 d^2 r^2 \sigma^4 + 2r \norm{y}_2^2\,.
	\end{equation}
	Hence, it holds that
	\begin{equation*}
		\sup_{\phi \in \sF_{\sigma_n}} \int_{\YY} \abs{\entt{\phi}{\hmu_n}} \de{\nu} \leq 8 d^2 r^2 \sigma_n^4 + 2 r \EVV{Y \sim \nu}{\norm{Y}_2^2}
	\end{equation*}
	and similar for $\hnu_n$. Thus,
	\begin{align} \label{eq:sg_ev_entt_bound}
		&\left( \EV*{ \left( \sup_{\phi \in \sF_{\sigma_n}} \int_{\YY} \abs{\entt{\phi}{\hmu_n}} \de{\nu} +  \sup_{\phi \in \sF_{\sigma_n}} \int_{\YY} \abs{\entt{\phi}{\hmu_n}} \de{\hnu_n} \right)^2 } \right)^{1/2} \\
		&\qquad \leq ( \EV{ ( 16 d^2 r^2 \sigma_n^4 + 2r \EVV{Y \sim \nu}{\norm{Y}_2^2} + 2r \EVV{Y \sim \hnu_n}{\norm{Y}_2^2} )^2 } )^{1/2} \nonumber \\
		&\qquad \lesssim d^2 r^2 \sigma^4\,, \nonumber
	\end{align}
	where we used the Cauchy-Schwarz inequality and the moment bounds provided in \autoref{lemma:sigma_n} and Lemma~B.1 in the supplement of \citet{Mena2019}. Recalling \eqref{eq:sg_entt_bound}, we have
	\begin{align*}
		b^2 \defeq \sup_{\phi \in \sF_{\sqrt{6}\sigma}} \norm{\entt{\phi}{\hmu_n}}^2_{L^2(\hnu_n)} &\leq \EVV{Y \sim \hnu_n}{(8  d^2 r^2 (\sqrt{6}\sigma)^4 + 2r \norm{Y}_2^2)^2}\,\\
		&= \EVV{Y \sim \hnu_n}{(288 d^2 r^2 \sigma^4 + 2r \norm{Y}_2^2)^2}.
	\end{align*}
	Employing Theorem~3.5.1 from \citet{Gine2015} in combination with \autoref{lemma:entt_cov_num} and \autoref{prop:unif_metric_ent_sg}, we obtain
	\begin{align*}
		\sqrt{n} \EV*{\sup_{\phi \in \sF_{\sqrt{6}\sigma}} \abs*{ \int_{\YY} \entt{\phi}{\hmu_n} \de{[\nu - \hnu_n]}}}
		&\lesssim \EV*{ \int_0^{b} \sqrt{\log 2 \covnum{\delta / 2}{\sF_{\sqrt{6}\sigma}}{\norminf{}}} \de{\delta} } \\
		&\lesssim  \EV{ b + [G^{(\alpha)}]^{s / (2\alpha)} \sigma^{3s/2} b^{1 - s/(2\alpha)}} \\
		&\leq (1 + [G^{(\alpha)}]^{s / (2\alpha)} \sigma^{3s/2}) \EV{1+b} \\
		&\leq (1 + [G^{(\alpha)}]^{s / (2\alpha)} \sigma^{3s/2}) \left(1 + \EV{b^2}^{1/2}\right)\\
		&\lesssim (1 + [G^{(\alpha)}]^{s / (2\alpha)} \sigma^{3s/2}) d^2 r^2 \sigma^4,
	\end{align*}
	where for the last inequality we used Lemma~B.1 in the supplement of \citet{Mena2019}, \autoref{lemma:sigma_n} and  $d,r, \sigma\geq 1$ to upper bound $ \EV{b^2}^{1/2}$, and the implicit constant only depends on $\alpha$, $d$, $s$, and $\sU$. Putting everything together, we obtain that
	\begin{align*}
		\sqrt{n} \EV{\abs{\EOT(\hmu_n, \hnu_n) - \EOT(\mu, \nu)}} &\lesssim ([G^{(\alpha)}]^{s/\alpha} + r^2d^2) \sigma^{4\alpha} + r^2 + d^2 r^2 \sigma^4 \\
		&\qquad + (1 + [G^{(\alpha)}]^{s / (2\alpha)} \sigma^{3s/2}) d^2 r^2 \sigma^4 \\
		&\lesssim (1 + [G^{(\alpha)}]^{s / \alpha} ) r^2 \sigma^{4\alpha \lor (4 + 3s/2) }\,,
	\end{align*}
	where the implicit constant only depends on $\alpha$, $d$, $s$, and $\sU$.
\end{proof}

\begin{proof}[Proof of \autoref{cor:lca_sg}]
	Consider the case $\hEOT = \EOT(\hmu_n, \hnu_n)$, the one-sample plug-in estimators can be dealt with analogously. Furthermore, w.l.o.g.\ we can assume that $0 < \eps < 1$. By \autoref{rem:rescale_norm2_eps}, it holds that
	\begin{equation*}
		\EV{\abs{\EOT(\hmu_n, \hnu_n) - \EOT(\mu, \nu)}} = \eps \EV{ \abs{ \OT_{c,1}(\hmu_n^\eps, \hnu_n^\eps) - \OT_{c,1}(\mu^\eps, \nu^\eps) } }\,,
	\end{equation*}
	where  $\mu^\eps$ is supported on $\eps^{-1/2} \XX = \bigcup_{i=1}^I \eps^{-1/2} g_i(\sU_i)$ with $\sup_{x \in \eps^{-1/2} \XX} \norm{x}_2 \leq \eps^{-1/2} r$ and $\mu^{\eps}$, $\nu^{\eps}$ are $\sigma^2 / \eps$-sub-Gaussian. Hence, it follows from \autoref{thm:lca_sg} that
	\begin{align*}
		&\EV{\abs{\EOT(\hmu_n, \hnu_n) - \EOT(\mu, \nu)}} \\
		&\qquad\qquad \lesssim \eps \left(\sum_{i=1}^I 1 + [G_\eps^{(i,\alpha)}]^{s / \alpha} \right) [\eps^{-1/2} r]^2 [\eps^{-1/2} \sigma]^{4\alpha \lor (4 + 3s/2)}  n^{-1/2}\,,
	\end{align*}
	where $G_\eps^{(i,\alpha)}$ are the constants from \autoref{lemma:infnorm_bound_sg} where $g_i$ is substituted with $\eps^{-1/2} g_i$. As $\eps < 1$, it follows from the definition that $G_\eps^{(i,\alpha)} \leq \eps^{-\alpha/2} G^{(i, \alpha)}$. As a consequence,
	\begin{align*}
		&\EV{\abs{\EOT(\hmu_n, \hnu_n) - \EOT(\mu, \nu)}} \\
		&\qquad\qquad \lesssim \left(\sum_{i=1}^I 1 + [G^{(i,\alpha)}]^{s / \alpha} \right)  r^2 \sigma^{4\alpha \lor (4 + 3s/2)} \epsilon^{-[2\alpha \lor (2 + 3s/4)] - s/2}n^{-1/2}\,. \qedhere
	\end{align*}
\end{proof}

\subsection{Computational Complexity}\label{appendix:A3}

For our computational analysis of a computable estimator for the empirical EOT cost we make use of the following characterization of dual potentials which arise from Sinkhorn iterations. Its proof is based on an insight which was previously formulated for cost chosen as the squared Euclidean norm by \citet[Proof of Theorem 6]{pooladian2021entropic}.

\begin{lemma}[Potentials from Sinkhorn algorithm]\label{lemma:sinkhornPot}
	Let \autoref{ass:eot_cost} hold and consider probability measures $\mu\in \PC(\XC), \nu \in\PC(\YC)$. For $\psi_0\in \Lexp(\nu)$ define its single and double $(c, \epsilon)$-transform as $\phi \coloneqq \entt{\psi_0}{\nu}$ and $\psi\coloneqq \entt{\phi}{\mu}$
	Further, define the measure $\tilde \nu$ on $\YC$ by
	\begin{equation*}
		\frac{\de{\tilde \nu}}{\de{\nu}}(y) \defeq \int _\XC\exp_\eps\left(\phi(x) + \psi_0(y) - c(x,y) \right) \dP{\mu}{x}\,.
	\end{equation*}
	Then, the measure $\tilde \nu$ is a probability measure on $\YC$. Further, the potentials $(\phi, \psi)$ are the dual optimizers of \eqref{eq:eot_gen_duality} for $\mu$ and $\tilde \nu$, and it holds that
	\begin{equation*}
		\EOT(\mu, \tilde \nu)= \int_{\XX} \phi \de{\mu} + \int_{\YY} \psi \de{\tilde \nu} = \max_{\phi \in \FC{}} \int_{\XX} \phi \de{\mu} + \int_{\YY} \entt{\phi}{\mu} \de{\tilde \nu}\,.
	\end{equation*}
\end{lemma}
\begin{proof}
	For the first assertion note that $\tilde \nu$ has a non-negative density with respect to $\nu$, hence it suffices to show that $\tilde \nu$ integrates to one. Invoking the Tonelli-Fubini theorem it follows by definition of the $(\c,\eps)$-transform of $\psi_0$ that
	\begin{align*}
		\tilde\nu(\YC) &= \int_\YC \int_\XC \exp_\eps\left(\phi(x) + \psi_0(y) - c(x,y) \right) \dP{\mu}{x} \dP{\nu}{y}\\
		&=\int_\XC \exp_\eps(\phi(x))\left[\int_\YC \exp_\eps\left(\psi_0(y) - c(x,y) \right) \dP{\nu}{y} \right]\dP{\mu}{x} \\
		&=\int_\XC \exp_\eps(\phi(x)-\phi(x))\dP{\mu}{x} = 1.
	\end{align*}
	For the second assertion we employ the characterization for optimality of EOT plans and potentials from \autoref{thm:eot_gen_duality}. To this end, we define the measure $\pi$ on $\XC\times \YC$,
	\begin{equation*}
		\de{\pi} \defeq \exp_\eps(\phi \oplus \psi - \c) \de{[\mu \otimes \tilde \nu]}
	\end{equation*}
	and verify that $\pi\in \Pi(\mu, \tilde \nu)$. Once this is confirmed, the remaining assertions follow at once from \autoref{thm:eot_gen_duality} and \autoref{thm:eot_duality}. To show the marginal constraints we first observe for any $y\in \YC$ that the marginal density of $\pi$ on $\YC$ fulfills
	\begin{align*}
		&\int_\XC \exp_\eps\left(\phi(x)+ \psi(y) - \c(x,y)\right) \dP{\mu}{x}\\
		= & \exp_\eps\left(\entt{\phi}{\mu}(y)\right)\int_\XC \exp_\eps\left(\phi(x) - \c(x,y)\right) \dP{\mu}{x}\\
		=& \exp_\eps\left(\entt{\phi}{\mu}(y) -\entt{\phi}{\mu}(y)\right) = 1.
	\end{align*}
	For the marginal density of $\pi$ on $\XC$ we note by definition of $\phi$ and $\tnu$ for any $x \in \XX$ that
	\begin{align*}
		&\int_\YC \exp_\eps\left(\phi(x)+ \psi(y) - \c(x,y)\right) \dP{\tilde \nu}{y}\\
		= &\int_\YC \frac{\exp_\eps\left(\phi(x) - \c(x,y)\right)}{\int_\XC \exp_\eps\left(\phi(\tilde x) - \c(\tilde x,y)\right)\dP{\mu}{\tilde x}} \dP{\tilde \nu}{y}\\
		= &\int_\YC \frac{\exp_\eps\left(\phi(x) + \psi_0(y)- \c(x,y)\right)}{\int_\XC \exp_\eps\left(\phi(\tilde x) + \psi_0(y)- \c(\tilde x,y)\right)\dP{\mu}{\tilde x}} \dP{\tilde \nu}{y}\\
		=&\int_\YC \exp_\eps\left(\phi(x)+ \psi_0(y) - \c(x,y)\right) \dP{\nu}{y}\\
		=&\exp_\eps\left(\entt{\psi_0}{\nu}(x)\right) \int_\YC \exp_\eps\left(\psi_0(y) - \c(x,y)\right) \dP{\nu}{y}\\
		=& \exp_\eps\left(\entt{\psi_0}{\nu}(x) - \entt{\psi_0}{\nu}(x)\right) = 1.\qedhere
	\end{align*}
\end{proof}

\subsection{Implications to the Entropic Gromov-Wasserstein Distance}\label{appendix:A4}

\begin{proof}[Proof of \autoref{lemma:egw_stab_bound}]
	For $A \in \sD$, denote
	\begin{equation*}
		U^{\mu, \nu}_{\eps}(A) \defeq 32 \norm{A}_2^2 + \OT_{c_A, \eps}(\mu, \nu)\,,
	\end{equation*}
	such that by \autoref{thm:egw_duality} it holds
	\begin{equation*}
		\GW_{2,\eps}(\mu, \nu) = \min_{A \in \sD} U^{\mu, \nu}_\eps(A)\,.
	\end{equation*}
	In particular, there exist $A, \tilde{A} \in \sD$ such that
	\begin{equation*}
		\GW_{2,\eps}(\mu, \nu) = U^{\mu, \nu}_\eps(A)\,, \quad \GW_2(\tmu, \tnu) = U^{\tmu, \tnu}_\eps(\tilde{A})\,.
	\end{equation*}
	By optimality, it follows that
	\begin{equation*}
		U^{\mu,\nu}_\eps( A ) - U^{\tmu, \tnu}_\eps( A ) \leq U^{\mu, \nu}_\eps(A) - U^{\tmu, \tnu}_\eps(\tilde{A}) \leq U^{\mu, \nu}_\eps(\tilde{A}) - U^{\tmu, \tnu}_\eps(\tilde{A})\,.
	\end{equation*}
	Hence,
	\begin{align*}
		\abs{\GW_{2,\eps}(\mu, \nu) - \GW_{2,\eps}(\tmu, \tnu)} &\leq \abs{U^{\mu,\nu}_\eps( A ) - U^{\tmu, \tnu}_\eps( A ) } + \abs{ U^{\mu, \nu}_\eps(\tilde{A}) - U^{\tmu, \tnu}_\eps(\tilde{A}) } \\
		& = \abs{ \OT_{c_A, \eps}(\mu, \nu) - \OT_{c_A, \eps}(\tmu, \tnu) } + \abs{ \OT_{c_{\tilde{A}}, \eps}(\mu, \nu) - \OT_{c_{\tilde{A}},\eps}(\tmu, \tnu) } \\
		&\leq 2 \sup_{A \in \sD} \abs{ \OT_{c_A, \eps}(\mu, \nu) - \OT_{c_A, \eps}(\tmu, \tnu) }\,.
	\end{align*}
	An application of \autoref{lemma:eot_stab_bound} yields the second inequality.
\end{proof}

\begin{proof}[Proof of \autoref{lemma:enttd_cov_num}]
First note by the assumption on $\sF$ that elements of $\enttd{\sF}{\tmu}$ are well-defined and real-valued. W.l.o.g.\ we can assume that covering numbers on the right-hand side of \eqref{eq:entropyBoundGW} are finite since otherwise the bound is vacuous.
	As in the proof of \autoref{lemma:entt_cov_num}, we construct a $\delta/2$-covering $\mset{\phi_1, \ldots, \phi_N} \subseteq \sF$ of $\sF$ with $N \defeq \covnum{\delta / 4}{\sF}{\norminf{}}$. Furthermore, let $\mset{A_1, \ldots, A_M} \subseteq \sD$ with $M \defeq \covnum{\delta / [64 r^2]}{\sD}{\norminf{}}$ be a $\delta/[64 r^2]$-covering of $\sD$. We show that
	\begin{equation*}
		\mset[\Big]{\entta{\phi_i}{A_j}{\tmu} \;\Big|\; i \in \mset{1, \ldots, N},\, j \in \mset{1, \ldots, M}}
	\end{equation*}
	is a $\delta$-covering of $\enttd{\sF}{\tmu}$ and conclude. For $\psi \in \enttd{\sF}{\tmu}$, by definition there is a $\phi \in \sF$ and $A \in \sD$ such that $\psi = \entta{\phi}{A}{\tmu}$. In particular, there exists $\phi_i$ and $A_j$ with $\norminf{\phi - \phi_i} \leq \delta / 2$ and $\norminf{A - A_j} \leq \delta / [64 r^2]$. The triangle inequality yields that
	\begin{align*}
		\norminf{\entta{\phi}{A}{\tmu} - \entta{\phi_i}{A_j}{\tmu}}
		&\leq \norminf{\entta{\phi}{A}{\tmu} - \entta{\phi_i}{A}{\tmu}} + \norminf{\entta{\phi_i}{A}{\tmu} - \entta{\phi_i}{A_j}{\tmu}} \\
		&\leq \norminf{\phi - \phi_i} + \norminf{\c_A - \c_{A_j}} \\
		&\leq \norminf{\phi - \phi_i} + 32 r^2 \norminf{A - A_j} \leq \delta / 2 + \delta / 2 = \delta\,. \qedhere
	\end{align*}
\end{proof}

\begin{proof}[Proof of \autoref{thm:lca_egw}]
	Consider the two-sample estimator $\hEGW = \GW_{\eps}(\hmu_n, \hnu_n)$ and note that the one-sample estimators from \eqref{eq:egw_emp_est} can be handled similarly. First, note that by the proof of Theorem~2 from \citet{Zhang2022} and \autoref{ass:egw}, it holds that
	\begin{equation*}
		\EV{\abs{ \hEGW - \GW_\eps(\mu, \nu)}} \lesssim r^4 n^{-1/2} + \EV{\abs{\GW_{2,\eps}(\hmu_n, \hnu_n) - \GW_{2,\eps}(\mu, \nu)}}\,.
	\end{equation*}
	Hence, it remains to bound the second term involving $\GW_{2,\eps}$. Following the proof of \autoref{thm:eot_lca} with \autoref{lemma:egw_stab_bound} and \autoref{lemma:enttd_cov_num}, we see that if there exist constants $K_\eps,\, k > 0$ such that for $\delta > 0$ suffices small it holds
	\begin{equation*}
		\log \covnum{\delta / 4}{ \FCd }{ \norminf{} } + \log \covnum{\delta / [64 r^2]}{\sD}{\norminf{}} \leq K_\eps \delta^{-k}\,,
	\end{equation*}
	then
	\begin{equation*}
		\EV{\abs{\GW_{2,\eps}(\hmu_n, \hnu_n) - \GW_{2,\eps}(\mu, \nu)}} \lesssim \sqrt{1 + K_\eps} n^{-1/2} \,.
	\end{equation*}
	Since for any $\delta > 0$ it holds,
	\begin{equation*}
		\log \covnum{\delta}{\sD}{\norminf{}} \lesssim s d r^2 \delta^{-1}\,,
	\end{equation*}
	it remains to show that the uniform covering numbers of $\FCd$ are suitably bounded. Furthermore, we need to rescale $\mset{\c_A}_{A\in\sD}$ by a constant such that \autoref{ass:eot_cost} is met for each element. To this end, note by \autoref{thm:egw_duality} for $a > 0$ that
	\begin{equation*}
		\frac{1}{a} \GW_{2,\eps}(\mu, \nu) = \min_{A \in \sD} \frac{32}{a} \norm{A}_2^2 + \OT_{\c_A / a, \eps / a}(\mu, \nu)\,,
	\end{equation*}
	which only depends on the scaled cost functions $\c_A / a$. Furthermore, we have uniformly over all $A \in \sD$ that
	\begin{equation} \label{eq:bound_derivatives_cA}
		\norminf{\De^k \c_A} \leq \begin{cases}
			20 r^4 & \text{if } \abs{k} = 0\,, \\
			(8 + 16 \sqrt{d}) r^3 & \text{if } \abs{k} = 1\,, \\
			8 r^2 & \text{if } \abs{k} = 2 \,, \\
			0 & \text{if } \abs{k} > 2\,,
		\end{cases}
	\end{equation}
	and can therefore set $a \defeq 20r^4$. In particular, we see that the functions classes $\sF_{\c_A, \eps}$ are $\alpha$-H\"older smooth for any $\alpha \in \N$ with uniform Hölder constant over all $A \in \sD$. Hence, as in the proof of \autoref{prop:unif_metric_ent_hoelder} we obtain the desired upper bound on the uniform covering numbers of $\FCd$. Putting everything together, via the rescaling
	\begin{equation*}
		\EV{\abs{\GW_{2,\eps}(\hmu_n, \hnu_n) - \GW_{2,\eps}(\mu, \nu)}} = a \EV{\abs{a^{-1} \GW_{2,\eps}(\hmu_n, \hnu_n) - a^{-1} \GW_{2,\eps}(\mu, \nu)}}\,,
	\end{equation*}
	and \autoref{rem:rescaling_effect}, the assertion follows.
\end{proof}

\begin{proof}[Proof of \autoref{rem:unreg_gw}]
	We follow the proof of \autoref{thm:lca_egw} and consider $\widehat{\GW}_{0,n} = \GW_{0}(\hmu_n, \hnu_n)$. First, note that since the decomposition $\GW_{\eps}(\mu, \nu) = \GW_{1,1}(\mu, \nu) + \GW_{2,\eps}(\mu, \nu)$ for centered $\mu$, $\nu$ also holds in the unregularized case $\eps = 0$ \citep[Section~4]{Zhang2022}, it remains to bound the statistical error of $\GW_{2,0}(\hmu_n, \hnu_n)$. Considering Corollary~1 from \citet{Zhang2022} and the proofs of Lemma~2.1 and Theorem~2.2 from \citet{Hundrieser2022}, we see that \autoref{lemma:egw_stab_bound} and \autoref{lemma:enttd_cov_num} remain valid for $\eps = 0$, where the entropic $(\c, \eps)$-transform is to be understood as the (measure-independent) $\c$-transform. Hence, we can apply an adjusted version of Theorem~2.2 from \citet{Hundrieser2022} and are left with providing suitable bounds on the uniform metric entropy of the function class $\sF_{\sD, 0}$. To this end, note that the cost functions $\mset{\c_A}_{A\in\sD}$ are semi-concave and Lipschitz continuous in the first component with uniform moduli only depending on $r$ and $d$, see \eqref{eq:bound_derivatives_cA}. Thus, $\sF_{\sD, 0}$ is contained in the class of uniformly bounded, Lipschitz continuous and semi-concave functions on $\XX$ and \autoref{lemma:cov_num_sc} provides the required bound on the uniform metric entropy.
\end{proof}

\section{Uniform Metric Entropy} \label{sec:unif_metric_entropy}

In this section, we give bounds on the uniform metric entropy of certain function classes. Recall that the uniform metric entropy is given by the logarithm of the covering numbers with respect to the uniform norm $\norminf{}$. For $\delta > 0$, the covering numbers of a function class $\sF$ on $\XX$ w.r.t.\ $\norminf{}$ are in turn defined as
\begin{equation*}
	\covnum{\delta}{\sF}{\norminf{}} \defeq \inf \mset{ n \in \N \mid \exists f_1, \ldots,f_n : \XX \to \R \text{ s.t. } \sup_{f \in \sF} \min_{1 \leq i \leq n} \norminf{f - f_i} \leq \delta}\,.
\end{equation*}

To apply \autoref{thm:eot_lca}, we are interested in the uniform metric entropy of the class $\FC{}$ introduced in \autoref{thm:eot_duality}. Motivated by the following two lemmata, we consider in \autoref{sec:sample_comp} the setting of $\XX = \bigcup_{i=1}^I g_i(\sU_i)$.

\begin{lemma}[Union bound, {\citealt[Lemma~3.1]{Hundrieser2022}}] \label{lemma:cov_num_union}
	Let $\sF$ be a class of functions on $\XX = \bigcup_{i=1}^I \XX_i$ for $I\in\N$ subsets $\XX_i \subseteq \XX$. Furthermore, denote with $\restr{\sF}{\XX_i} \defeq \mset{ \restr{\phi}{\XX_i} : \XX_i \to \R \mid \phi \in \sF}$ the collection of functions in $\sF$ restricted to $\XX_i$. Then, it follows for any $\delta > 0$ that
	\begin{equation*}
		\log \covnum{\delta}{\sF}{\norminf{}} \leq \sum_{i=1}^{I} \log \covnum{\delta}{\restr{\sF}{\XX_i}}{\norminf{}}\,.
	\end{equation*}
\end{lemma}

\begin{lemma}[Composition bound, {\citealt[Lemma~A.1]{Hundrieser2022}}] \label{lemma:cov_num_comp}
	Let $g : \sU \to \XX$ be a surjective map between the sets $\sU$ and $\XX$, and denote with $\sF$ a function class on $\sU$. Then, it follows for the composed function class $\sF \circ g \defeq \mset{\phi \circ g \mid \phi \in \sF}$ on $\XX$ for any $\delta > 0$ that
	\begin{equation*}
		\covnum{\delta}{\sF}{\norm{}_{\infty}} \leq \covnum{\delta}{\sF \circ g}{\norminf{}}\,.
	\end{equation*}
\end{lemma}

Hence, bounding the uniform metric entropy for the setting $\XX = \bigcup_{i=1}^I g_i(\sU_i)$ reduces to controlling $\FC{} \circ g_i$ for one $i$ (provided that they all are of a similar structure). Furthermore, denoting $\trho_i \defeq \c \circ (g_i, \id_{\YY})$, it holds for $\entt{\psi}{\xi} \in \FC{}$ that $\entt{\psi}{\xi} \circ g_i = \enttrasf{\psi}{\trho_i}{\eps}{\xi}$. In particular, this implies that $\FC{} \circ g_i = \sF_{\trho_i, \eps}$. By the definition of the entropic $(\c, \eps)$-transform, it can thus be seen that certain properties of $\trho_i$ are inherited to the function class $\FC{} \circ g_i$. In the following, we give uniform metric entropy bounds for function classes that contain $\FC{} \circ g_i$ for suitably chosen $g_i$. These results can then be combined with \autoref{thm:eot_lca}.
\begin{lemma}[{\citealt[Inequality (238)]{Kolmogorov1961}}] \label{lemma:cov_num_lip}
	Let $(\XX, d_\XX)$ be a connected metric space and $k > 0$ such that
	\begin{equation*}
		\covnum{\delta}{\XX}{d_\XX} \lesssim \delta^{-k} \qquad \text{for $\delta > 0$ sufficiently small.}
	\end{equation*}
	Then, it follows for the same values of $\delta$ and the class $\sF$ of $1$-Lipschitz continuous functions  on $\XX$  {which are bounded by one} that
	\begin{equation*}
		\log\covnum{\delta}{\sF}{\norminf{}} \lesssim \delta^{-k}\,,
	\end{equation*}
	where the implicit constant only depends on $k$ and $\XX$.
\end{lemma}

\begin{lemma}
	\label{lemma:cov_num_sc}
	Let $\XX \subseteq \R^s$ be bounded and convex with $s\in \NN$. Then, it follows for $\delta > 0$ sufficiently small and the class $\sF$ of $1$-Lipschitz continuous and $1$-semi-concave functions on $\XX$ {which are bounded by one} that
	\begin{equation*}
		\log\covnum{\delta}{\sF}{\norminf{}} \lesssim \delta^{-s/2}\,,
	\end{equation*}
	where the implicit constant only depends on $\XX$.
\end{lemma}
\begin{proof}The assertion follows from the proof of Lemma A.3 in  \citet{Hundrieser2022}; for the sake of completeness we spell it out. By rotation and translation we may assume that $\XC$ is contained in the linear subspace $V = \RR^{\tilde s}\times \{0\}^{s-\tilde s}$ for $\tilde s \leq s$ and admits non-empty relative interior. By definition of $\sF$, every function $\tilde f\colon \XC\to \RR, x \mapsto f(x) - \|x\|^2$ for $f\in \sF$ is concave, $L$-Lipschitz with $L\coloneqq 1 + 2\diam(\XC)$ on $\XC$ and bounded in absolute value by $ 1 + \diam(\XC)^2$. Hence, by \citet[Theorem 1 and Remark 2(ii)]{dragomirescu1992smallest} there exists a concave $L$-Lipschitz extension $\overline f$ of $\tilde f$ onto a compact cube $\mathcal{D}\subseteq V$ with non-empty relative interior such that $\XC\subseteq \mathcal{D}$. In particular, $\overline f$ is absolutely bounded by $B \coloneqq 1+\diam(\XC)^2 + L\diam(\mathcal{D})$, and thus $\overline f$ is contained in the class $C_{B,L}(\mathcal{D})$ of concave functions that are bounded by $B$ and $L$-Lipschitz. We thus conclude for $\delta>0$ sufficiently small that
	\begin{align*}
		\log N(\delta, \sF, \|\cdot\|_{\infty, \XC})&= \log N(\delta, \sF - \|\cdot\|^2, \|\cdot\|_{\infty, \XC})\leq \log N(\delta, C_{B,L}(\mathcal{D}), \|\cdot\|_{\infty, \XC})\\
		& \leq \log N(\delta, C_{B,L}(\mathcal{D}), \|\cdot\|_{\infty, \mathcal{D}}) \lesssim  \delta^{-\tilde s/2}\lesssim \delta^{-s/2},
	\end{align*}
	where the second to last inequality follows by uniform metric entropy bounds on the class $C_{B,L}(\mathcal{D})$ detailed in \citet{Bronshtein1976} or \citet{Guntuboyina2013}. In particular, the suppressed constant depends on $\mathcal{D}$, $B$ and $L$, which again depends on $\XC$.
\end{proof}

\begin{lemma}[{\citealt[Theorem 2.7.1]{Vaart1996}}] \label{lemma:cov_num_hoelder}
	Let $\XX \subset \R^s$ be bounded and convex with nonempty interior. Then, it follows for $\delta > 0$ and the class $\sC^\alpha_M(\XX)$ of $\alpha$-H\"older smooth functions with $\alpha > 0$ and $M > 0$ that
	\begin{equation*}
		\log \covnum{\delta}{\sC^\alpha_M(\XX)}{\norminf{}} \lesssim M^{s/\alpha}\delta^{-s/\alpha} \,,
	\end{equation*}
	where the implicit constant only depends on $s$, $\alpha$ and $\XX$.
\end{lemma}

\section{Rescaling Properties} \label{sec:rescaling_prop}

This appendix summarizes some useful insights on how the entropic optimal transport cost and corresponding convergence statements change under rescaling.

\begin{remark}[Rescaling] \label{rem:eot_rescale_mean_abs_dev}
	Suppose that $c$ is a bounded cost function that does not satisfy \autoref{ass:eot_cost}, i.e., it holds that $\norminf{\c} \in (1, \infty)$. As for any $a > 0$, $b \in \R$ and $(\mu, \nu)\in \sP(\XX)\times \sP(\YY)$ the EOT cost fulfills the rescaling property $$\OT_{a \c + b, \eps}(\mu, \nu) = a \cdot \OT_{\c, \eps / a}(\mu, \nu) + b,$$ we obtain for the empirical estimators $\hEOT$ in \eqref{eq:eot_emp_est} with $a \defeq \norminf{c}$ that
	\begin{equation*}
		\EV{\abs{\hEOT - \EOT(\mu, \nu)}} = a \EV{\abs{\hOT_{\c / a, \eps / a, n} - \OT_{\c / a, \eps / a}(\mu, \nu)}}\,,
	\end{equation*}
	where the underlying cost function on the right-hand side satisfies \autoref{ass:eot_cost}.
\end{remark}

\begin{remark}[Rescaling of squared Euclidean distance] \label{rem:rescale_norm2_cost} \label{rem:rescale_norm2_eps}
	Denote by $c$ the squared Euclidean distance, consider measurable sets $\XX,\, \YY \subseteq \R^d$ and let $r > 0$. For probability measures $\mu \in \sP(\XX)$, $\nu \in \sP(\YY)$ denote by $\mu^{r^2}$ the pushforward of $\mu$ w.r.t.\ to the map $x \mapsto r^{-1} x$ and likewise define $\nu^{r^2}$. Then, since $\c(r^{-1}x, r^{-1}y) = r^{-2} \c(x, y)$, we observe for any $\eps > 0$ that
	\begin{align}\label{eq:rescaling_space}
		\OT_{\c/r^2,\eps}(\mu, \nu) = \OT_{\c, \eps}(\mu^{r^2}, \nu^{r^2})\,,
	\end{align}
	where $\mu^{r^2}$ and $\nu^{r^2}$ are supported on $r^{-1} \XX$ and $r^{-1}\YY$, respectively. Moreover, by combining \autoref{rem:eot_rescale_mean_abs_dev} with the rescaling property \eqref{eq:rescaling_space} for $r = \epsilon^{1/2}>0$, it also follows that
	\begin{equation*}
		\EOT(\mu, \nu) = \eps \OT_{\c/\eps, 1}(\mu, \nu) =  \eps \OT_{\c, 1}(\mu^\eps, \nu^\eps).
	\end{equation*}
	In particular, $\mu^\eps$ and $\nu^{\eps}$ are supported in $\eps^{-1/2} \XX$ and $\eps^{-1/2} \YY$, respectively. In addition, if $\mu$ and $\nu$ are $\sigma^2$-sub-Gaussian for some $\sigma^2>0$, then $\mu^\eps$ and $\nu^{\eps}$ are $\sigma^2 / \eps$-sub-Gaussian.
\end{remark}

\begin{remark}[Constants in convergence statements under rescaling] \label{rem:rescaling_effect}
	Under \autoref{ass:eot_cost} and \autoref{ass:semi_con} we have the constraint that the cost function as well as the Lipschitz and semi-concavity moduli must all be bounded by $1$. As mentioned before, if one of these constraints is violated, then we can rescale appropriately via \autoref{rem:eot_rescale_mean_abs_dev} and still obtain bounds for the statistical error. We now discuss how this rescaling affects the constants.
	\begin{enumerate}
		\item First, we treat \autoref{ass:lip} and \autoref{ass:semi_con} together (for the former, ignore the semi-concavity). Denote by $L$ the Lipschitz modulus of the cost and by $\Lambda$ its semi-concavity modulus. If $a \defeq [\norminf{\c} \lor L \lor \Lambda] \geq 1$, then rescaling yields the additional factor $a$ for the statistical error bounds in \autoref{thm:lca_lip} and \autoref{thm:lca_sc}. This is due to the fact they do not depend on $\eps$.
		\item For \autoref{thm:lca_hoelder}, the effect of rescaling is more complicated. Indeed, it affects the H\"older constants $C^{(i, \alpha)}$ as well as the entropic regularization parameter $\eps$. W.l.o.g.\ consider the case that $I = 1$, $g_1 = \id_{\XX}$ and drop the index $i$. Let $a \defeq \norminf{\c} > 1$ and denote the normalized cost $\trho \defeq \c / a$ with its version $\tilde{C}^{(\alpha)}$ of $C^{(\alpha)}$. Due to the recursive and increasing structure, the relationship of $\tilde{C}^{(\alpha)}$ and $C^{(\alpha)}$ is not straightforward. To simplify this, let
		\begin{equation*}
			C \defeq \max_{\abs{j}=1,\ldots,\alpha} \norminf{\De^j\c}\,, \qquad \tilde{C} \defeq \max_{\abs{j}=1,\ldots,\alpha} \norminf{\De^j \trho}\,.
		\end{equation*}
		Then, it holds that $C = a \tilde{C}$ and
		\begin{equation*}
			\tilde{C}^{(\alpha)} \leq \begin{cases}
				\tilde{C}^\alpha & \tilde{C} \geq 1\,, \\
				\tilde{C} & \tilde{C} < 1\,,
			\end{cases} = a^{-\eta} C^{\eta} \qquad \text{with }\eta \defeq \begin{cases}
			\alpha & C \geq a \,,\\
			1 & C < a \,.
		\end{cases}
		\end{equation*}
		Using this in combination with \autoref{thm:lca_hoelder} yields that
		\begin{equation*}
			\EV{\abs{\hEOT - \EOT(\mu, \nu)}} \lesssim a^{1 + \frac{s}{2} [\frac{\alpha - 1}{\alpha} - \frac{\eta}{\alpha}]} C^{\frac{\eta s}{2\alpha}} (\eps \land a)^{- \frac{s}{2} \frac{\alpha - 1}{\alpha}} 		\begin{cases}
				n^{-1/2} & s/\alpha < 2 \,, \\
				n^{-1/2} \log (n+1) & s/\alpha = 2 \,, \\
				n^{-\alpha/s} & s/\alpha > 2\,,
			\end{cases}
		\end{equation*}
		Hence, we observe that the rescaling affects the statistical error bound polynomially.
	\end{enumerate}
\end{remark}

\end{document}